\documentclass{amsart}
\usepackage{amsmath}
\usepackage{amsthm}
\usepackage{amssymb}
\usepackage{MnSymbol}
\usepackage[shortlabels]{enumitem}
\usepackage{fullpage}
\usepackage[all,arc,curve,color,frame]{xy}
\usepackage[OT2,T1]{fontenc}
\DeclareSymbolFont{cyrletters}{OT2}{wncyr}{m}{n}
\DeclareMathSymbol{\Sha}{\mathalpha}{cyrletters}{"58}
\def\PP{\mathbb{P}}
\def\ZZ{\mathbf{Z}}
\def\RR{\mathbf{R}}
\def\CC{\mathbf{C}}

\def\QQ{\mathbf{Q}}
\def\FF{\mathbb{F}}

\def\OO{\mathcal{O}}

\def\inv{^{-1}}
\def\p{\mathfrak{p}}
\def\CA{\mathcal{A}}
\def\CY{\mathcal{Y}}
\def\CX{\mathcal{X}}
\def\CZ{\mathcal{Z}}

\DeclareMathOperator{\tr}{tr}
\DeclareMathOperator{\disc}{disc}
\DeclareMathOperator{\N}{N}

\DeclareMathOperator{\Div}{Div}
\DeclareMathOperator{\Pic}{Pic}

\DeclareMathOperator{\Frob}{Frob}
\DeclareMathOperator{\Ver}{Ver}
\DeclareMathOperator{\Spec}{Spec}
\DeclareMathOperator{\Gal}{Gal}
\DeclareMathOperator{\End}{End}
\DeclareMathOperator{\Hom}{Hom}
\DeclareMathOperator{\Aut}{Aut}

\newtheorem{theorem}{Theorem}[section]
\newtheorem*{theorem*}{Theorem}
\newtheorem{corollary}[theorem]{Corollary}

\newtheorem{lemma}[theorem]{Lemma}
\newtheorem{example}[theorem]{Example}
\newtheorem*{example*}{Example}
\newtheorem*{corollary*}{Corollary}
\theoremstyle{remark}
\newtheorem{remark}[theorem]{Remark}
\theoremstyle{definition}
\newtheorem{definition}[theorem]{Definition} 

\title{Twists of Shimura Curves}
\author{James Stankewicz}
\date{\today}
\thanks{The author was partially supported by NSF VIGRE grant DMS-0738586 and the University of Georgia Dissertation completion award. The work in this paper formed a part of the author's Ph.D. thesis under Pete L. Clark and Dino Lorenzini. The author would also like to thank Patrick Morton and John Voight for helpful conversations.}

\begin{document}
\begin{abstract}Consider a Shimura curve $X^D_0(N)$ over the rational numbers. We determine criteria for the twist by an Atkin-Lehner involution to have points over a local field. As a corollary we give a new proof of the theorem of Jordan-Livn\'e on $\mathbf{Q}_p$ points when $p\mid D$ and for the first time give criteria for $\mathbf{Q}_p$ points when $p\mid N$. We also give congruence conditions for roots modulo $p$ of Hilbert class polynomials.\end{abstract}

\maketitle

Let $D$ be the squarefree product of an even number of primes and let $N$ be a squarefree integer coprime to $D$. The Shimura curves $X^D_0(N)_{/\QQ}$ are natural generalizations of the classical modular curves $X_0(N)$, which we realize here as $X^1_0(N)_{/\QQ}$. Shimura first defined these curves over $\QQ$ \cite{Sh71} and also showed that $X^D_0(N)(\RR)$ is nonempty if and only if $D=1$. Later, conditions for $X^D_0(N)(\QQ_p)$ to be nonempty were determined when $p\mid D$ first by Jordan and Livn\'e \cite[Theorem 5.6]{JoLi} in the case $N=1$ and in the general case by Ogg \cite[Th\'eor\`eme]{OggMauvaise}.

In this paper, we give comprehensive criteria for the presence of $\QQ_p$-rational points on all Atkin-Lehner twists of $X^D_0(N)$ including the trivial twist, $X^D_0(N)$. Therefore as a consequence, we recover the theorem of Jordan and Livn\'e and for the first time give criteria for $\QQ_p$-points when $p\mid N$ and $D>1$. We note that conjecturally, these twists and their combinations form \emph{all} twists of $X^D_0(N)$ for all but finitely many pairs of $D$ and $N$ \cite{Kontogeorgis}. Let $C^D(N,d,m)$ denote the twist of $X^D_0(N)$ by $\QQ(\sqrt d)$ and the Atkin-Lehner involution $w_m$ as in Definition \ref{AtkinLehnerOO}. Particular cases of interest are the twists by the full Atkin-Lehner involution $w_{DN}$. In that case we have the following.  








\begin{corollary*}[\ref{SuperspecialFullAL}] If $p\nmid DN$ is inert in $\QQ(\sqrt d)$, $C^D(N,d,DN)(\QQ_p)$ is nonempty.\end{corollary*}

\begin{theorem*}[\ref{RamifiedMainTheorem}, partial] Suppose that $p\nmid 2DN$ is a prime which is ramified in $\QQ(\sqrt d)$. Then $C^D(N,d,DN)(\QQ_p)\ne \emptyset$ if and only if one of the following occurs.

\begin{itemize}

\item $\left(\frac{-DN}{p}\right) = 1$ and a certain Hilbert Class Polynomial has a root modulo $p$.

\item  $\left(\frac{-DN}{p}\right) = -1$, $2\nmid D$, $\left(\frac{-p}{q}\right) = -1$ for all primes $q\mid D$, and $\left(\frac{-p}{q}\right) = 1$ for all primes $q\mid N$ such that $q\ne 2$

\item  $2\mid D$, $\left(\frac{-DN}{p}\right) = -1$, $p\equiv \pm 3 \bmod 8$, $\left(\frac{-p}{q}\right)= -1$ for all primes $q\mid (D/2)$, and $\left(\frac{-p}{q}\right) = 1$ for all primes $q\mid N$.

\end{itemize}

\end{theorem*}

\begin{corollary*}[\ref{CDfullALCor}] Let $p\mid D$ be a prime which is unramified in $\QQ(\sqrt d)$. Let $p_i$, $q_j$ be primes such that $D/p = \prod_i p_i$ and $N = \prod_j q_j$.

\begin{itemize}

\item If $p$ is split in $\QQ(\sqrt d)$, then $C^D(N,d,DN) \cong X^D_0(N)$ over $\QQ_p$ and $X^D_0(N)(\QQ_p)$ is nonempty if and only if one of the following two cases occurs.

\begin{enumerate}

\item $p=2$, $p_i\equiv 3 \bmod 4$ for all $i$, and $q_j \equiv 1 \bmod 4$ for all $j$

\item $p\equiv 1 \bmod 4$, $D = 2p$, and $N=1$

\end{enumerate}

\item If $p$ is inert in $\QQ(\sqrt d)$ then $C^D(N,d,DN)(\QQ_p)$ is nonempty.

\end{itemize}
\end{corollary*}

\begin{corollary*}[\ref{mainDRCor}]
Let $p$ be a prime dividing $N$ such that $p$ is unramified in $\QQ(\sqrt d)$. Then $C^D(N,d,DN)(\QQ_p)$ is nonempty if and only if 

\begin{itemize}

\item $p$ is split in $\QQ(\sqrt d)$ and either $D=1$ or

\begin{itemize}

\item $p=2$, $D = \prod_i p_i$ with each $p_i \equiv 3\bmod 4$, and $N/p = \prod_j q_j$ with each $q_j\equiv 1\bmod 4$, or

\item $p=3$, $D = \prod_i p_i$ with each $p_i \equiv 2\bmod 3$, and $N/p = \prod_j q_j$ with each $q_j\equiv 1\bmod 3$, or

\item $TF'(D,N,1,p)>0$ where $TF'$ is as in Definition \ref{defntfprime}.

\end{itemize}

\item $p$ is inert in $\QQ(\sqrt d)$ with $Dp = \prod_i p_i$, $N/p = \prod_j q_j$ such that one of the following holds.

\begin{itemize}

\item $p=2$, for all $i$, $p_i\equiv 3\bmod 4$ and for all $j$, $q_j \equiv 1\bmod 4$.

\item $p \equiv 3 \bmod 4$, $D=1$ and $N = p$ or $2p$.
\end{itemize}
\end{itemize}
\end{corollary*}


We also give infinite families of examples of twists which have $\QQ_v$-rational points for all places $v$ of $\QQ$.

\begin{example*}[\ref{Shimura2qExample}] Suppose that $q$ is an odd prime and consider $X^{2q}_0(1)_{/\QQ}$, a curve of genus $g$. Note that this curve is hyperelliptic over $\QQ$ if and only if $q \in\{13,19,29,31,37,43,47,67,73,97,103\}$\cite[Theorem 7]{OggReal}. Let $p\equiv 3\bmod 8$ be a prime such that $\left(\frac{-p}{q}\right) = -1$ and such that for all odd primes $\ell$ less than $4g^2$, $\left(\frac{-p}{\ell}\right) = -1$. Let the twist of $X^{2q}_0(1)$ by $\QQ(\sqrt{-p})$ and $w_{2q}$ be denoted by $C^{2q}(1,-p,2q)_{/\QQ}$. Then $C^{2q}(1,-p,2q)$ has $\QQ_v$-rational points for all places $v$ of $\QQ$. \end{example*}

If $q=13$, then the genus of $X^{26}_0(1)$ is two. Therefore $X^{26}_0(1)$ is hyperelliptic, and has the following explicit model, where $w_{2q}$ is identified with the hyperelliptic involution \cite{genustwoshimura}: $$y^2 = -2x^6 + 19x^4 -24x^2 -169.$$ 
Hence, an explicit model for $C^{26}(1,-p,2q)$ is given by the affine equation $$y^2 = 2px^6 - 19px^4 +24px^2 + 169p.$$ The primes less than 2000 satisfying the congruence conditions in the above example are $p = 67,163,$ and $1747$. It can be checked that the explicit model of $C^{26}(1,-67,26)$ has at least the rational points $\left(\frac{\pm 9}{5}, \frac{\pm 10988}{125}\right),$ and that $C^{26}(1,-163,26)$ has at least the rational points $\left(\frac{\pm 67}{35},\frac{\pm 5270116}{42875}\right)$ . If $p= 1747$, a point search in \texttt{sage} \cite{sage} failed to produce any rational points and the \texttt{TwoCoverDescent} command in \texttt{MAGMA} did not determine if $C^{26}(1,-1747,26)$ has no rational points.

\begin{example*}[\ref{PrimeLevelTwistExample}] Let $q\equiv 3\bmod 4$ be a prime and consider the curve $X_0(q)_{/\QQ}$. Let $p\equiv 1\bmod 4$ be a prime such that $\left(\frac{p}{q}\right) = -1$ and let $C^1(q,p,q)_{/\QQ}$ denote the twist of $X_0(q)$ by $\QQ(\sqrt p)$ and $w_q$. Then $C^1(q,p,q)$ has $\QQ_v$-rational points for all places $v$ of $\QQ$.\end{example*}

If $q=23$, the least two primes satisfying the above are $p = 5$ and $p = 13$. Using a hyperelliptic model of the genus 2 curve $X_0(23)$ \cite{Go91} as above, it can be verified that $C^1(23,5,23)(\QQ)$ is nonempty. Meanwhile, the \texttt{TwoCoverDescent} command in \texttt{MAGMA} determined that $C^1(23,13,23)(\QQ)$ is empty.

Finally if $\Delta <0$ we recall the Hilbert Class Polynomial \cite[p.285]{Cox} which describes the unramified abelian extensions of $\QQ(\sqrt \Delta)$. Only for finitely many $\Delta$ is the splitting of $H_\Delta(X)$ modulo primes completely determined by congruence conditions \cite[Theorem 3.22]{Cox}. The results of this paper allow us to find examples of primes $p$ in which congruence conditions determine the splitting of $H_\Delta(X)$ modulo $p$.


\begin{corollary*}[\ref{SplittingHilbertClassPolynomial}] Let $p\ne 2$ be a prime and let $N$ be a squarefree integer such that $\left(\frac{-N}{p}\right) = -1$. Let $H_{\Delta}(X) \in \ZZ[X]$ denote the Hilbert Class Polynomial of discriminant $\Delta$. It follows that $H_{-4N}(X)$ has a root modulo $p$ if and only if for all odd primes $q\mid N$, $\left(\frac{-p}{q}\right) = 1$.\end{corollary*}

We proceed as follows: After reviewing some quaternion arithmetic, we will prove some Theorems on embeddings of quadratic orders which may be of independent interest. Then after properly defining these Shimura curves and their Atkin-Lehner involutions, we will show how these embedding theorems may be applied to the problem of controlling \emph{superspecial points} on Shimura curves over finite fields. The remaining sections deal with determining $\QQ_p$-rational points.

\section{Quaternion Arithmetic}
\subsection{Basic definitions and theorems}

\begin{definition} A \emph{quaternion algebra} over a field $K$ is a four-dimensional central simple $K$-algebra. \end{definition}

\begin{example} If the characteristic of $K$ is not 2, and $a,b\in K^\times$ then there is a quaternion algebra over $K$ which we denote $\left(\frac{a,b}{K}\right)$. This algebra has a $K$-basis $\langle 1, i, j, k\rangle$ such that $i^2 = a$, $j^2 = b$ and $k = ij = -ji$.\end{example}

\begin{definition} Let $K$ be a number field. We say that a quaternion algebra $B$ is \emph{ramified} at a place $v$ of $K$ if $B\otimes_K K_v$ is a division algebra.\end{definition}

\begin{definition} If $K=\QQ$, we say that a quaternion algebra $B$ is \emph{definite} if $B$ is ramified at $\infty$. Likewise we say that $B$ is \emph{indefinite} if $B$ is unramified at $\infty$.\end{definition}

It is well-known that if $K$ is a number field, the quaternion algebras $B_{/K}$ are determined up to isomorphism by the even number of ramified places of $K$. 
It follows that if $K=\QQ$, $B$ is definite if and only if $B$ is ramified at an odd number of primes. Therefore we make the following definition.

\begin{definition} Let $D>0$ be a squarefree positive integer. Let $B_D$ denote the unique quaternion $\QQ$-algebra such that $B_D$ is ramified at $p$ if and only if $p\mid D$. To any quaternion $\QQ$-algebra, we associate its discriminant $\disc(B)$, the unique positive squarefree number such that $B \cong B_{\disc(B)}$.
\end{definition}

\begin{definition}\label{QuatMainInvolution} Let $B$ be a quaternion $K$-algebra and let $a\mapsto \overline a$ denote the \emph{main involution} of $B$ over $K$. 
Define the trace $a\mapsto \tr(a) = a + \overline a$ and the norm $\N(a) = a\overline a$.\end{definition}

\begin{definition} A $\ZZ$-order $\OO$ in a quaternion $\QQ$-algebra $B$ is a rank four $\ZZ$-subalgebra of $B$ such that for all $\theta\in\OO$, $\tr(\theta) \in \ZZ$ and $\N(\theta)\in \ZZ$.\end{definition}

\begin{definition} The discriminant of a $\ZZ$-order $\OO$ with a $\ZZ$-basis $e_1, \ldots, e_4$, is $\disc(\OO) = \det(\tr(e_ie_j))$.\end{definition}

\begin{lemma}{\cite[Corollaire I.4.8]{Vigneras}}\label{DiscriminantIndexLemma} If $\OO_1 \supset \OO_2$ then $\disc(\OO_1)\mid \disc(\OO_2)$. Moreover, $[\OO_1:\OO_2] = \sqrt{\left\mid \frac{\disc(\OO_2)}{\disc(\OO_1)}\right\mid }$ so if $\disc(\OO_2) = \disc(\OO_1)$ then $\OO_1 = \OO_2$.\end{lemma}

\begin{definition} An order in a quaternion algebra will be called \emph{maximal} if it is maximal with respect to inclusion.\end{definition}

\begin{lemma}{\cite[Corollaire II.5.3]{Vigneras}} An order $\OO$ in a quaternion $\QQ$-algebra $B$ is maximal if and only if $\disc(B) = \sqrt{\mid \disc(\OO)\mid }$.\end{lemma}

If an order $\OO$ is contained in two maximal orders $\OO_1$ and $\OO_2$, then $[\OO_1:\OO] = [\OO_2:\OO]$ by Lemma \ref{DiscriminantIndexLemma}.

\begin{definition} A $\ZZ$-order $\OO\subset B$ is called an \emph{Eichler order} when it is the intersection of two (not necessarily distinct) maximal $\ZZ$-orders. An Eichler $R$-order may be analogously defined for other commutative rings. The \emph{level} of an Eichler order is its index in either maximal order. \end{definition}

\begin{definition} By Lemma \ref{DiscriminantIndexLemma}, if $\OO$ is an Eichler order, $\sqrt{\mid \disc(\OO)\mid }$ is a positive integer, which we may sometimes refer to as the \emph{reduced discriminant}.\end{definition}

\begin{definition}\label{DefnZp2} Let $\ZZ_{p^2}$ denote the unique irreducible unramified degree two ring extension of $\ZZ_p$.\end{definition}

\begin{lemma}{\cite[Corollaire II.1.7]{Vigneras}}\label{pDividesDiscriminantLemma} Let $B$ be a quaternion $\QQ$-algebra ramified at $p$. Then $B\otimes \QQ_p$ has a unique maximal $\ZZ_p$-order $\OO$. Moreover, there exists an element $\pi \in B\otimes \QQ_p$ such that $\pi^2 \OO = p\OO$ and $\OO\cong\ZZ_{p^2} \oplus \pi \ZZ_{p^2}$ . It follows that for $a\in \ZZ_{p^2}$, $\pi a \pi \inv = \sigma(a)$ where $\langle \sigma \rangle = \Aut_{\ZZ_p}(\ZZ_{p^2})$.\end{lemma}


Hereon, we suppress the $\ZZ$ as all of our quaternion algebras will be over $\QQ$ (or be the base change of a quaternion algebra over $\QQ$).

\begin{lemma}{\cite[Lemme II.2.4]{Vigneras}, \cite[Corollaire III.5.2]{Vigneras}} Let $B$ be a quaternion $\QQ$-algebra and $\OO$ an Eichler order of level $N$. If $p\nmid \disc(B)$, then there is an embedding $\OO\otimes \ZZ_p \hookrightarrow M_2(\ZZ_p)$. Moreover there is a unique integer $n$ such that $\OO\otimes \ZZ_p$ is conjugate to an order in $M_2(\ZZ_p)$ of the form $$\left(\begin{array}{cc} \ZZ_p & \ZZ_p \\ p^n \ZZ_p & \ZZ_p\end{array}\right).$$ We may explicitly give $n$ as the non-negative integer such that $p^n\mid N$ but $p^{n+1}\nmid N$.\end{lemma}

\begin{definition} We say that an order $\OO$ is \emph{ramified} at $p$ if $p\mid \disc(\OO)$.\end{definition}

\begin{definition}\cite[p.17]{BasisProblem} Let $B$ be a quaternion algebra over $\QQ$ and $\OO\subset B$ an order. A left $\OO$-ideal is a left $\OO$-module $M$ contained in $B$ such that $\OO M = M$ and for all primes $p$ of $\QQ$, there exist $m_p \in B$ such that $\ZZ_p\otimes M = \ZZ_p\otimes \OO m_p$. If $M$ is a left $\OO$-ideal then we call $\OO_r(M) := \{ x \in B : Mx \subset M\}$ the \emph{right order} of $M$. We say that $M$ is two-sided if $\OO = \OO_r(M)$. We may similarly define right ideals $I$ and their left orders $\OO_l(I)$. \end{definition}

\begin{definition} Let $B$ be a quaternion algebra and $\OO\subset B$ an order. We say that a (left, right or two-sided) $\OO$-ideal $M$ is \emph{integral} if $M\subset \OO$.\end{definition}

\begin{definition} Let $B$ be a quaternion algebra and $\OO\subset B$ an order. We say a left $\OO$-ideal $M$ is \emph{principal} if there is some $m\in B$ such that $M = \OO m$, and similarly for right $\OO$-ideals.\end{definition}

\begin{lemma}{\cite[Corollaire III.5.7(1), Lemme III.5.6]{Vigneras} }\label{QuaternionIndefinite} If $B$ is indefinite and $\OO$ is an Eichler order in $B$ (of any level), then every left (or right) $\OO$-ideal is principal. Therefore, the Eichler orders (of any given level) are conjugate. \end{lemma}


\begin{lemma}{\cite[Proposition V.3.1, Corollaire V.2.3]{Vigneras}} If $B$ is definite and $\OO$ is an Eichler order in $B$ then $\OO^\times$ is finite. The number of left (or right) ideals up to right (or left) multiplication by $B^\times$ is finite.\end{lemma}





\begin{lemma}{\cite[Theorem II.1.1]{BasisProblem}}\label{EichTwoSided} Let $B$ be a quaternion algebra of discriminant $D$. If $\OO$ is an Eichler order of square-free level $N$ in $B$, then the two-sided ideals of $\OO$ form an abelian group under multiplication. For each prime $p\mid DN$, there is a unique two-sided integral ideal $\wp_p$ such that $\wp_p^2 = \OO p = p\OO$. Moreover, any two-sided ideal of $\OO$ is equal to one of the form $$(\prod_{p\mid DN} \wp_p^{\epsilon_{p}}) r$$ where $r\in \QQ$ and $\epsilon_p \in \{0,1\}$.
\end{lemma}

\begin{definition} Let $A$ be a finitely-generated, torsion-free $\ZZ$-algebra, let $A^0$ be $A\otimes_\ZZ \QQ$ and let $\epsilon_A: A\to A^0$ be the natural embedding $a\mapsto a\otimes 1$. Suppose that there exists an embedding $\phi: A_1 \hookrightarrow A_2$ of finitely generated torsion-free $\ZZ$-algebras. Define $\phi^0:A_1^0 \hookrightarrow A_2^0$ to be the induced embedding $a\otimes r/s \mapsto \phi(a) \otimes r/s$. We say that $\phi$ is \emph{optimal} if $\epsilon_{A_1}(A_1) = (\phi^0)\inv ( \epsilon_{A_2}(A_2))$.\end{definition} 

Let $\phi:A_1 \hookrightarrow A_2$ be an embedding of finitely generated torsion-free $\ZZ$-algebras. Define $A_1' := (\phi^0)\inv ( \epsilon_{A_2}(A_2))$ and note that $A_1'$ is a finitely generated torsion-free $\ZZ$-algebra. Note also that $A_1^0 \supset A_1' \supset A_1$ because $\phi^0$ induces an embedding $A_1' \hookrightarrow A_2$. Moreover, this embedding is an optimal embedding $\psi: (\phi^0)\inv(\epsilon_{A_2}(A_2)) \to A_2$.


\begin{example} If $E_{/\CC}$ has $j$-invariant $0$, there is an embedding $\phi$ of $A_1 = \ZZ[\sqrt{-3}]$ into $A_2 = \End(E)\cong \ZZ\left[\frac{1 + \sqrt{-3}}{2}\right]$. This embedding is however not optimal. Note that $\phi^0$ is actually an isomorphism $A_1^0 \cong A_2^0 \cong \QQ(\sqrt{-3})$ and $A_1$ has index 2 in $(\phi^0)\inv(\epsilon_{A_2}(A_2)) \cong A_2$.\end{example}

\begin{definition}\label{QuadraticRingDiscDefn} Let $\Delta$ be an integer which is congruent to zero or one modulo four. We will denote by $R_\Delta$ the unique quadratic order of discriminant $\Delta$. If $\Delta$ is not a square, $R_\Delta\otimes \QQ$ is a quadratic field $K_\Delta = \QQ(\sqrt \Delta)$. In this case, we may define the class number $h(\Delta) := \#\Pic(R_\Delta)$ and the conductor $f(\Delta) := [\ZZ_{K_\Delta} : R_\Delta]$. We also fix $w(\Delta) := \#R_\Delta^\times$.\end{definition}

\begin{definition} Let $p$ be a prime, and let $(\frac{\cdot}{p})$ denote the Kronecker symbol. That is, if $p$ is odd, the Kronecker symbol is the Legendre symbol. If $p=2$ then $(\frac{2}{2}) = 0$ and if $q$ is an odd prime then $(\frac{q}{2}) = (-1)^{(q^2 -1)/8}$. We obtain the Kronecker symbol by extending multiplicatively.

The Eichler symbol may then be defined in terms of the Kronecker symbol as follows: $$ \left\{\frac{\Delta}{p}\right\} = \begin{cases} 1 & p\mid  f(\Delta) \\ \left(\frac{\Delta}{p}\right) & else \end{cases}$$\end{definition}

\begin{definition}\label{DefnEDN}For square-free coprime integers $D$ and $N$ and some integer $\Delta \equiv 0,1 \bmod 4$, we define the quantity $$e_{D,N}(\Delta) := h(\Delta) \prod_{p\mid D} \left(1 - \left\{\frac{\Delta}{p}\right\}\right)\prod_{q\mid N} \left(1 + \left\{\frac{\Delta}{p}\right\}\right).$$\end{definition}

\begin{theorem}[Eichler's embedding theorem]\label{EichlerEmbedding} Let $D$ and $N$ be square-free coprime integers. If $B_D$ is indefinite, then the number of optimal embeddings, up to $\OO^\times$ conjugacy, of a quadratic order $R$ of discriminant $\Delta$ into some Eichler order $\OO$ of level $N$ in $B_D$ is $e_{D,N}(\Delta)$. If $B_D$ is definite, then the number of optimal embeddings, up to $\OO^\times$ conjugacy, of a quadratic order $R$ of discriminant $\Delta$ into some Eichler order $\OO$ of level $N$ in $B_D$ is $e_{D,N}(\Delta)/w(\Delta)$. \end{theorem}

\begin{proof} This is proven separately in the indefinite case \cite[Corollaire III.5.12]{Vigneras} and in the definite case \cite[Proposition 5]{BasisProblem}.\end{proof}



\subsection{Simultaneous embeddings into Eichler orders}\label{SimEmb}


Let $B'$ be a definite quaternion $\QQ$-algebra. Suppose that there exist $\omega_1, \omega_2\in B'$ such that $\omega_1^2 = -q$ and $\omega_2^2 = -d$ for $q,d\in \ZZ$. Thus $\omega_1\omega_2\in B'$ is of norm $qd$. Although $\omega_1$ and $\omega_2$ are integral, it may be the case that $\omega_1\omega_2$ is not integral. We only know that $\tr(\omega_1\omega_2)<4qd$. In order for $\omega_1\omega_2$ to be integral it is necessary and sufficient that $\tr(\omega_1\omega_2) = \omega_1\omega_2 + \omega_2\omega_1 = s \in \ZZ$.

Now let us grant that $\tr(\omega_1\omega_2) \in \ZZ$. 
Since $\omega_1,\omega_2$, and $\omega_1\omega_2$ are integral, any order $\OO'$ that contains $\omega_1$ and $\omega_2$ contains $\omega_1\omega_2$. Note that the $\ZZ$-module generated by 1, $\omega_1$, $\omega_2$ and $\omega_1\omega_2$ is an order of $B'$ if and only if $\langle 1,\omega_1,\omega_2,\omega_1\omega_2\rangle$ is a basis for $B'$ over $\QQ$. In the latter case, we may compute that the reduced discriminant of $\ZZ \oplus \ZZ\omega_1 \oplus \ZZ\omega_2 \oplus \ZZ\omega_1\omega_2$ is $4qd -s^2$. If $q \equiv 3\bmod 4$, $\dfrac{1 + \omega_1}{2}$ is integral and the reduced discriminant of $\ZZ \oplus \ZZ\dfrac{1 + \omega_1}{2} \oplus \ZZ \omega_2 \oplus \ZZ\dfrac{ 1 + \omega_1}{2}\omega_2$ is $dq - \left(\frac{s}{2}\right)^2$.


\begin{theorem}\label{zetafour} Fix square-free positive integers $D',N'$ such that $(D',N') =1$ and $D'$ is the product of an odd number of primes. Fix also $m>1$ such that $m|D'N'$. The following are equivalent.

\begin{enumerate}

\item There is a definite quaternion algebra $B'$ over $\QQ$ of discriminant $D'$, an Eichler order $\OO'$ of level $N'$ in $B'$ and elements $\omega_1$ and $\omega_2$ contained in $\OO'$ such that $\omega_1^2  = -1$ and $\omega_2^2 = -m$. 






\item There are factorizations $D' = \prod_i p_i$ and $N' = \prod_j q_j$ into distinct primes such that
\begin{itemize}

\item  $m = D'N'$ or $2|D'N'$ and $m = D'N'/2$

\item for all $i$ either $p_i =2$ or $p_i \equiv 3 \bmod 4$

\item for all $j$ either $q_j = 2$ or $q_j \equiv 1\bmod 4$

\end{itemize}

\end{enumerate}

\end{theorem}

\begin{proof} 

If $\ZZ[\zeta_4] \hookrightarrow \OO'$, then $p_i =2$ or $p_i\equiv 3\bmod 4$ and $q_j =2$ or $q_j \equiv 1 \bmod 4$ by Theorem \ref{EichlerEmbedding}. 

Since $m>1$, $\ZZ[\zeta_4]\not\hookrightarrow\ZZ[\sqrt{-m}]$ and vice versa. Therefore $\OO' \supset \ZZ \oplus \ZZ\omega_1 \oplus \ZZ\omega_2 \oplus \ZZ\omega_1\omega_2$ and so $m\mid D'N'\mid 4m - s^2$. If $s=0$, we have $m\mid D'N'\mid 2m$ since $D'N'$ is squarefree.

If $s\ne 0$, $m\mid 4m - s^2$ implies that $m\mid s$ and $m \le |s| $. Since $m^2 \le s^2 < 4m$, we have $m<4$. If $m=2$ and $0 < s^2 < 4m= 8$ then $m|s$ implies that $|s| =2$ and thus $2\mid D'N'\mid 4$. Then since $D'N'$ square-free and $D'>1$, $m = D' = D'N' =2$. If $m=3$ and $0 < s^2 < 4m = 12$ then $m\mid s$ implies that $|s|  = 3$ and thus $3\mid D'N'\mid 3$ so $m = D' = D'N' = 3$. We have thus shown $(1)\Rightarrow (2)$.

For $(2) \Rightarrow (1)$, 
it suffices to consider the quaternion algebra $A = \left(\dfrac{-1,-D'N'}{\QQ}\right)$ with $\omega_1= i$ and $\omega_2 = j$. It can be calculated that $A \cong B_{D'}$. 

If $2\mid D'N'$, $\left(\dfrac{1 + \omega_1}{2}\right) \omega_2$ squares to $-D'N'/2$. Set $\omega_2' = \left(\dfrac{1 + \omega_1}{2}\right)\omega_2$ so that the reduced discriminant of $\ZZ \oplus \ZZ \omega_1 \oplus \ZZ \omega_2' \oplus \ZZ \omega_1\omega_2'$ is $4D'N'/2 = 2D'N'$. An explicit order containing $\omega_1$ and $\omega_2$ is the ``Hurwitz quaternions'' $$ \ZZ \oplus \ZZ \omega_1 \oplus \ZZ \omega_2' \oplus \ZZ \dfrac{ 1 + \omega_1 + \omega_2' + \omega_1\omega_2'}{2}$$ which have reduced discriminant $D'N'$.


If $2\nmid D'N'$ then $D'N' \equiv 3 \bmod 4$ and so $\dfrac{ 1 + \omega_2}{2}$ is integral. Therefore $\ZZ \oplus \ZZ\omega_1 \oplus \ZZ \left(\dfrac{1 + \omega_2}{2}\right) \oplus \ZZ\omega_1\left(\dfrac{1 + \omega_2}{2}\right)$ is an order and has reduced discriminant $D'N'$.\end{proof}


We note that we gave very explicit examples of orders satisfying Theorem \ref{zetafour} (1) in the proof above. We note that these orders are unique up to $B^\times$-conjugacy.

\begin{theorem}[Pizer]\label{PizerExtraUnits}

Let $B'$ be a definite $\QQ$-quaternion algebra and suppose that for all $p\mid \disc(B')$, $\left(\dfrac{-4}{p}\right) = -1$. Let $N$ be a squarefree integer such that for all $p\mid N$, $\left(\dfrac{-4}{p}\right) = 1$. Then there is a unique conjugacy class of Eichler orders of level $N$ in $B'$ into which $\ZZ[\zeta_4]$ embeds.

Similarly, suppose that for all $p\mid \disc(B')$, $\left(\dfrac{-3}{p}\right) = -1$, and let $N$ be a squarefree integer such that for all $p\mid N$, $\left(\dfrac{-3}{p}\right) = 1$. Then there is a unique conjugacy class of Eichler orders of level $N$ in $B'$ into which $\ZZ[\zeta_6]$ embeds.

\end{theorem}

\begin{proof} Let $\mathfrak o$ be an order in an imaginary quadratic field. Recall the definition given by Pizer \cite[Definition 11]{Pizer} of $D(\mathfrak{o})$ as the number of $(B')^\times$-conjugacy classes of Eichler orders of level $N$ in $B$ into which $\mathfrak{o}$ is optimally embedded. During the proof of Theorem 16 on page 73 of the same article, it is proven that if $\mathfrak{o} = \ZZ[\zeta_4]$ then $D(\mathfrak{o})$ is zero or one depending on whether or not there is an optimal embedding. Similarly on page 75 of the same article, the same thing is proven for $\ZZ[\zeta_6]$.
\end{proof}

\begin{corollary} Let $B'$ be a definite quaternion algebra of discriminant $D'$, and let $\OO'$ be an Eichler order of $B'$ of squarefree level $N'$ such that $\ZZ[\zeta_4]\hookrightarrow\OO'$. If $m\mid D'N'$ and $m\ne 1$, then $\ZZ[\sqrt{-m}]\hookrightarrow \OO'$ if and only if $m = D'N'$ or $2\mid D'N'$ and $m = D'N'/2$.\end{corollary}

 

\begin{corollary} When the conditions of Theorem \ref{zetafour} are satisfied, $B' \cong \left(\dfrac{-1,-D'N'}{\QQ}\right)$ and $\OO'$ is $(B')^\times$-conjugate to one of the following:

\begin{enumerate}

\item The unique maximal order in $B'$ if $D' = 2$ or $3$.

\item $\ZZ \oplus \ZZ i \oplus \ZZ \frac{1 + j}{2} \oplus \ZZ \frac{i + k}{2}$ if $2 \nmid D'N'$.

\item $\ZZ \oplus \ZZ i \oplus \ZZ \frac{j + k}{2} \oplus \ZZ\left(\frac{1 + i}{2} + \frac{ j + k}{4}\right)$ if $2\mid D'N'$.

\end{enumerate}

Moreover if $2\mid D'N'$ we note that the order in 3. contains $\dfrac{j+k}{2}$, a square root of $-D'N'/2$.

\end{corollary}

We now turn our attention to simultaneous embeddings of $\ZZ[\zeta_6]$ and $\ZZ[\sqrt{-m}]$.

\begin{theorem}\label{zetasix} Fix squarefree positive integers $D',N'$ such that $(D',N') =1$ and $D'$ is the product of an odd number of primes. Fix also $m|D'N'$ such that $m >1$, $m \ne 3$. The following are equivalent

\begin{enumerate}

\item There is a definite quaternion algebra $B'$ of discriminant $D'$, an Eichler order $\OO'$ of level $N'$ in $B'$ and $\frac{1 +\omega_1}{2},\omega_2\in\OO'$ such that $\omega_1^2  = -3$ and $\omega_2^2 = -m$. 






\item  There are factorizations $D' = \prod_i p_i$, $N' = \prod_j q_j$ into distinct primes such that

\begin{itemize}

\item $m = D'N'$, or $3\mid D'N'$ and $m = D'N'/3$

\item for all $i$ either $p_i =3$ or $p_i \equiv 2 \bmod 3$

\item for all $j$ either $q_j = 3$ or $q_j \equiv 1\bmod 3$

\end{itemize}

\end{enumerate}

\end{theorem}
\begin{proof}

If $\ZZ[\zeta_6] \hookrightarrow \OO'$ then $p_i =3$ or $p_i\equiv 2\bmod 3$ and $q_j =3$ or $q_j \equiv 1 \bmod 3$ by Theorem \ref{EichlerEmbedding}.

Note that since $m>1$ and $m\ne 3$, $\ZZ[\zeta_6]\not\hookrightarrow\ZZ[\sqrt{-m}]$ and vice versa. We know that since $\OO' \supset \ZZ \oplus \ZZ\left(\frac{1 + \omega_1}{2}\right) \oplus \ZZ\omega_2 \oplus \ZZ\left(\frac{1 + \omega_1}{2}\right)\omega_2$, $m\mid D'N'\mid 3m - (s/2)^2$. If $s=0$, we have the result that $m\mid D'N'\mid 3m$.

If $s\ne 0$, $m\mid 3m - (s/2)^2$ implies that $m\mid (s/2)$. Thus $m^2 \le (s/2)^2 <3m$ so $m=2$. If $0<(s/2)^2 < 6$ and $2\mid (s/2)$ then $s=4$ so $m = D' = D'N' = 2$.

To show that $(2)$ implies $(1)$, 
it suffices to consider the quaternion algebra $A = \left(\dfrac{-3,-D'N'}{\QQ}\right)$ with $\omega_1 = i$ and $\omega_2 = j$. It can be calculated that $A\cong B_{D'}$. 

If $3\mid D'N'$, then $(\omega_1\omega_2)^2 = -3D'N'$ so $(1/3)\omega_1\omega_2$ squares to $-D'N'/3$. We set $\omega_2' = \dfrac{\omega_1}{3}\omega_2$ so the reduced discriminant of $\ZZ \oplus \ZZ \dfrac{1 + \omega_1}{2} \oplus \ZZ \omega_2' \oplus \ZZ \dfrac{1 + \omega_1}{2}\omega_2'$ is $3(D'N'/3) = D'N'$. 

If $3\nmid D'N'$, then $D'N' \equiv -1 \bmod 3$. 
It can thus be calculated that $\ZZ \oplus \ZZ\dfrac{1 + \omega_1}{2} \oplus \ZZ\left(\dfrac{ 1 + \omega_2}{2} + \dfrac{\omega_1 + \omega_1\omega_2}{6}\right) \oplus \ZZ \dfrac{-3 + \omega_1 - 2\omega_1\omega_2}{6}$ is an Eichler order of level $N'$ in $A$.\end{proof}
%
%
%
%

\begin{corollary} Let $B'$ be a definite quaternion algebra of discriminant $D'$ and let $\OO'$ be an Eichler order of $B'$ of squarefree level $N'$ such that $\ZZ[\zeta_6]\hookrightarrow\OO'$. If $m\mid D'N'$ and $m\ne 1,3$, then $\ZZ[\sqrt{-m}]\hookrightarrow \OO'$ if and only if $m = D'N'$ or $D'N'/3$.\end{corollary}

 

\begin{corollary}

When the conditions of Theorem \ref{zetasix} are satisfied, $B' \cong \left(\dfrac{-3,-D'N'}{\QQ}\right)$ and $\OO'$ is $B^\times$-conjugate to one of the following:

\begin{enumerate}

\item The unique maximal order in $B'$ if $D' = 2$

\item $\ZZ \oplus \ZZ \frac{1 + i}{2} \oplus \ZZ \left(\frac{1 + j}{2}+ \frac{ i + k}{6}\right) \oplus \ZZ \frac{-3 + i -2k}{6}$ if $3 \nmid D'N'$

\item $\ZZ \oplus \ZZ \frac{1 + i}{2} \oplus \ZZ \frac{k}{3} \oplus \ZZ\frac{k-j}{6}$ if $3\mid D'N'$

Moreover if $3\mid D'N'$, the order in (3) contains $k/3$, a square root of $-D'N'/3$.

\end{enumerate}
\end{corollary}

We prove one final theorem on simultaneous embeddings. For the remainder of the section, let $D$ be the squarefree product of an even number of primes, $N$ a squarefree integer coprime to $D$, and $p$ a prime not dividing $DN$. We shall also set $B' := B_{Dp}$ and let $m\mid DN$ be an integer greater than one. The following lemma is an easy calculation.

\begin{lemma}\label{ExplicitQuaternionAlgebras} We have the following isomorphisms of $\QQ$-algebras.

\begin{enumerate}

\item If $2\nmid DNp$, $\left(\frac{-p}{q}\right) = -1$ for all primes $q\mid D$, $\left(\frac{-p}{q}\right) = 1$ for all primes $q\mid N$, and $\left(\frac{-DN}{p}\right) = -1$, then $B' \cong \left(\frac{-p,-DN}{\QQ}\right)$.

\item If $2\mid N$, $\left(\frac{-p}{q}\right) = -1$ for all primes $q\mid D$, $\left(\frac{-p}{q}\right) = 1$ for all primes $q\mid (N/2)$, and $\left(\frac{-DN}{p}\right) = -1$, then $B' \cong \left(\frac{-p,-DN}{\QQ}\right)\cong \left(\frac{-p,-DN/2}{\QQ}\right)$.

\item If $2\mid D$, $\left(\frac{-p}{q}\right) = -1$ for all primes $q\mid D$, $\left(\frac{-p}{q}\right) = 1$ for all primes $q\mid N$, and $\left(\frac{-DN}{p}\right) = -1$, then $B' \cong \left(\frac{-p,-DN}{\QQ}\right)$.

\item If $2\mid D$, $\left(\frac{-p}{q}\right) = -1$ for all primes $q\mid D$, $\left(\frac{-p}{q}\right) = 1$ for all primes $q\mid N$, and $\left(\frac{-DN/2}{p}\right) = -1$, then $B' \cong \left(\frac{-p,-DN/2}{\QQ}\right)$.

\item If $p=2$, $\left(\frac{-2}{q}\right) = -1$ for all primes $q\mid D$, and $\left(\frac{-2}{q}\right) = 1$ for all primes $q\mid N$ then $B' \cong \left(\frac{-2,-DN}{\QQ}\right)$.

\end{enumerate}

\end{lemma}

\begin{theorem}\label{GoodReductionSimultaneousEmbeddings} Recall that $D$ is the squarefree product of an even number of primes, $N$ a squarefree integer coprime to $D$, and $p$ a prime not dividing $DN$. Recall further that $B' = B_{Dp}$ and let $m\mid DN$ be an integer greater than one. We have the following equivalences.

\begin{enumerate}

\item Suppose that $2\nmid DNp$. There is an Eichler order $\OO'$ of level $N$ in $B'$ and embeddings $\psi_1: \ZZ[\sqrt{-p}]\hookrightarrow \OO'$ and $\psi_2:\ZZ[\sqrt{-m}] \hookrightarrow \OO'$ if and only if $m = DN$, $\left(\frac{-p}{q}\right) = -1$ for all primes $q\mid D$, $\left(\frac{-p}{q}\right) = 1$ for all primes $q\mid N$, and $\left(\frac{-DN}{p}\right) = -1$.

\item Suppose that $2\mid N$. There is an Eichler order $\OO'$ of level $N$ in $B'$ and embeddings $\psi_1: \ZZ[\sqrt{-p}]\hookrightarrow \OO'$ and $\psi_2:\ZZ[\sqrt{-m}] \hookrightarrow \OO'$ if and only if one of the following two cases occurs.

\begin{itemize} 

\item $m = DN$, $\left(\frac{-p}{q}\right) = -1$ for all primes $q\mid D$, $\left(\frac{-p}{q}\right) = 1$ for all primes $q\mid (N/2)$, and $\left(\frac{-DN}{p}\right) = -1$

\item $m = DN/2$, $\left(\frac{-p}{q}\right) = -1$ for all primes $q\mid D$, $\left(\frac{-p}{q}\right) = 1$ for all primes $q\mid (N/2)$, and $\left(\frac{-DN/2}{p}\right) = -1$

\end{itemize}

\item Suppose $2\mid D$ and $\left(\frac{-DN}{p}\right) = -1$. There is an Eichler order $\OO'$ of level $N$ in $B'$ and embeddings $\psi_1: \ZZ[\sqrt{-p}]\hookrightarrow \OO'$ and $\psi_2:\ZZ[\sqrt{-m}] \hookrightarrow \OO'$ if and only if $m = DN$, $\left(\frac{-p}{q}\right) = -1$ for all primes $q\mid (D/2)$, $p\not\equiv 7\bmod 8$, and $\left(\frac{-p}{q}\right) = 1$ for all primes $q\mid N$.

\item Suppose $2\mid D$ and $\left(\frac{-DN}{p}\right) = 1$. There is an Eichler order $\OO'$ of level $N$ in $B'$ and embeddings $\psi_1: \ZZ[\sqrt{-p}]\hookrightarrow \OO'$ and $\psi_2:\ZZ[\sqrt{-m}] \hookrightarrow \OO'$ if and only if $m = DN/2$, $DN \equiv 2,6$, or $10 \bmod 16$, $\left(\frac{-p}{q}\right) = -1$ for all primes $q\mid (D/2)$, $p\not\equiv 7\bmod 8$, and $\left(\frac{-p}{q}\right) = 1$ for all primes $q\mid N$.

\item Suppose that $p= 2$. There is an Eichler order $\OO'$ of level $N$ in $B'$ and embeddings $\psi_1: \ZZ[\sqrt{-p}]\hookrightarrow \OO'$ and $\psi_2:\ZZ[\sqrt{-m}] \hookrightarrow \OO'$ if and only if $m = DN \equiv \pm 3\bmod 8$, $\left(\frac{-2}{q}\right) = -1$ for all primes $q\mid D$, and $\left(\frac{-2}{q}\right) = 1$ for all primes $q\mid N$.

\end{enumerate}

\end{theorem}

\begin{proof}

Suppose first that there exist embeddings $\psi_1: \ZZ[\sqrt{-p}]\hookrightarrow \OO'$ and $\psi_2:\ZZ[\sqrt{-m}] \hookrightarrow \OO'$ into some Eichler order $\OO'$ of level $N$ in $B_{Dp}$. Let $\omega_1 = \psi_1(\sqrt{-p})$ and $\omega_2 = \psi_2(\sqrt{-m})$, so $\OO' \supset \{1,\omega_1,\omega_2,\omega_1\omega_2\}$. Since $(p,m) =1$, $\ZZ[\sqrt{-p}]\not\subset \ZZ[\sqrt{-m}]$ and $\ZZ[\sqrt{-m}\not\subset \ZZ[\sqrt{-p}]$. Thus, $$\OO'\supset \ZZ \oplus \ZZ\omega_1 \oplus \ZZ\omega_2 \oplus\ZZ\omega_1\omega_2,$$ an order of reduced discriminant $4mp - s^2$ where $s$ is the trace of $\omega_1\omega_2$. Therefore $DNp\mid 4mp - s^2$, and since $mp\mid DNp$, we must have $mp \mid s^2$. Since $mp$ is squarefree, $mp \mid s$ and so either $mp \le \mid s\mid $ or $s = 0$.

If $s \ne 0$ then $m^2 p^2 \le s^2 < 4mp$ and thus $mp <4$. However, recall that $m$ is an integer greater than one and $p$ is a prime, so $mp \ge 4$. Therefore $s = 0$ and $mp \mid  DNp \mid  4mp$. In fact, since $DNp$ is squarefree, it divides the squarefree part of $4mp$. If $2\nmid mp$ then $mp \mid  DNp \mid  2mp$ and either $2 \nmid DN$ and $m = DN$ or $2\mid DN$ and $m = DN/2$. If $2\mid m$ then $mp \mid  DNp \mid  mp$ and so $m = DN$. If $p =2$ then once more $mp \mid DNp\mid  mp$ and so $m = DN$. Recall now that if $q$ is an odd prime then $\left\{\frac{4\Delta}{q}\right\} = \left(\frac{\Delta}{q}\right)$. Therefore Theorem \ref{EichlerEmbedding} gives us the following congruence conditions.

\begin{itemize}

\item If $2\nmid DNp$ then $\left(\frac{-p}{q}\right) = -1$ for all primes $q\mid D$, $\left(\frac{-p}{q}\right) = 1$ for all primes $q\mid N$, and $\left(\frac{-DN}{p}\right) = -1$.

\item If $2\mid N$ and $m = DN$ then $\left(\frac{-p}{q}\right) = -1$ for all primes $q\mid D$, $\left(\frac{-p}{q}\right) = 1$ for all primes $q\mid (N/2)$, and $\left(\frac{-DN}{p}\right) = -1$

\item If $2\mid N$ and $m = DN/2$ then $\left(\frac{-p}{q}\right) = -1$ for all primes $q\mid D$, $\left(\frac{-p}{q}\right) = 1$ for all primes $q\mid (N/2)$, and $\left(\frac{-DN/2}{p}\right) = -1$

\item If $2\mid D$ and $m = DN$, then $\left(\frac{-p}{q}\right) = -1$ for all primes $q\mid (D/2)$, $p\equiv \pm 3\bmod 8$, and $\left(\frac{-p}{q}\right) = 1$ for all primes $q\mid N$.

\item If $2\mid D$ and $m = DN/2$, then $DN \equiv 2,6,10 \bmod 16$, $\left(\frac{-p}{q}\right) = -1$ for all primes $q\mid (D/2)$, $p\equiv \pm 3 \bmod 8$, and $\left(\frac{-p}{q}\right) = 1$ for all primes $q\mid N$.

\item If $p=2$ then $DN \equiv \pm 3\bmod 8$, $\left(\frac{-2}{q}\right) = -1$ for all primes $q\mid D$, and $\left(\frac{-2}{q}\right) = 1$ for all primes $q\mid N$.

\end{itemize}


We now prove the converse. If $2\nmid DNp$, then by Lemma \ref{ExplicitQuaternionAlgebras}(1), $B_{Dp} \cong B' = \left(\frac{-p,-DN}{\QQ}\right)$. Contained in $B'$ is the order $\ZZ\oplus\ZZ i \oplus \ZZ j \oplus \ZZ ij$ of reduced discriminant $4DNp$. If $p\equiv 3\bmod 4$ then $\ZZ\oplus \ZZ \frac{1 +i}{2} \oplus \ZZ j \oplus \ZZ \left(\frac{1 + i}{2}\right)j$ is an order of reduced discriminant $DNp$, and is thus an Eichler order of level $N$. Likewise if $DN \equiv 3\bmod 4$ there is an Eichler order of level $N$. 
Assume now that $p\equiv 1\bmod 4$. Then $$\left(\frac{-DN}{p}\right) = \left(\frac{DN}{p}\right) = \prod_{q\mid DN} \left(\frac{q}{p}\right) = \prod_{q\mid DN} \left(\frac{p}{q}\right) = \prod_{q\mid DN} \left(\frac{-p}{q}\right) (-1)^r$$ where $r$ is the number of primes $q\mid DN$ such that $q\equiv 3\bmod 4$. 
Moreover, since $D$ is the product of an even number of primes, $\left(\frac{-p}{q}\right) = -1$ if $q\mid D$, and $\left(\frac{-p}{q}\right) =1 $ if $q\mid N$, it follows that $\prod_{q\mid DN} \left(\frac{-p}{q}\right) = 1$. Putting this all together we have shown that if $p\equiv 1\bmod 4$, then $$ -1 = \left(\frac{-DN}{p}\right) = \begin{cases} 1 & DN \equiv 1\bmod 4 \\ -1 & DN \equiv 3\bmod 4\end{cases}.$$

If $2\mid N$ and $m= DN$, then by Lemma \ref{ExplicitQuaternionAlgebras}(2), $B_{Dp} \cong B' = \left(\frac{-p,-DN}{\QQ}\right)$. Contained in $B'$ is the order $\ZZ\oplus\ZZ i \oplus \ZZ j \oplus \ZZ ij$ of reduced discriminant $4DNp$. If $p\equiv 3\bmod 4$ then $\ZZ\oplus \ZZ \frac{1 +i}{2} \oplus \ZZ j \oplus \ZZ \left(\frac{1 + i}{2}\right)j$ is an order of reduced discriminant $DNp$, and is thus an Eichler order of level $N$. If $p\equiv 1\bmod 4$ then $\ZZ \oplus \ZZ i \oplus \ZZ \left(\frac{1 + i + j}{2}\right) \oplus \ZZ \left(\frac{-1 -i + ij}{2}\right)$ is an order of reduced discriminant $DNp$.

If $2\mid N$ and $m = DN/2$, then by Lemma \ref{ExplicitQuaternionAlgebras}(2), $B_{Dp} \cong B' = \left(\frac{-p,-DN/2}{\QQ}\right)$. Contained in $B'$ is the order $\ZZ\oplus\ZZ i \oplus \ZZ j \oplus \ZZ ij$ of reduced discriminant $2DNp$. It follows that the ``Hurwitz quaternions'' $\ZZ \oplus \ZZ i \oplus \ZZ j \oplus \ZZ \left(\frac{1 + i + j + ij}{2}\right)$ are an Eichler order of reduced discriminant $DNp$.

If $2\mid D$ and $m = DN$, then by Lemma \ref{ExplicitQuaternionAlgebras}(3), $B_{Dp} \cong B' = \left(\frac{-p,-DN}{\QQ}\right)$. Contained in $B'$ is the order $\ZZ\oplus\ZZ i \oplus \ZZ j \oplus \ZZ ij$ of reduced discriminant $4DNp$. If $p\equiv 3\bmod 8$ then $\ZZ\oplus \ZZ \frac{1 +i}{2} \oplus \ZZ j \oplus \ZZ \left(\frac{1 + i}{2}\right)j$ is an order of reduced discriminant $DNp$, and is thus an Eichler order of level $N$. If $p\equiv 5\bmod 8$ then $\ZZ \oplus \ZZ i \oplus \ZZ \left(\frac{1 + i + j}{2}\right) \oplus \ZZ \left(\frac{-1 -i + ij}{2}\right)$ is an order of reduced discriminant $DNp$.

If $2\mid D$ and $m = DN/2$, then by Lemma \ref{ExplicitQuaternionAlgebras}(4), $B_{Dp} \cong B' = \left(\frac{-p,-DN/2}{\QQ}\right)$. Contained in $B'$ is the order $\ZZ\oplus\ZZ i \oplus \ZZ j \oplus \ZZ ij$ of reduced discriminant $2DNp$. It follows that $\ZZ \oplus \ZZ i \oplus \ZZ j \oplus \ZZ \left(\frac{1 + i + j + ij}{2}\right)$ is an Eichler order of reduced discriminant $DNp$.

If $p=2$, then by Lemma \ref{ExplicitQuaternionAlgebras}(5), $B_{Dp} \cong B' = \left(\frac{-2,-DN}{\QQ}\right)$. Contained in $B'$ is the order $\ZZ\oplus\ZZ i \oplus \ZZ j \oplus \ZZ ij$ of reduced discriminant $4DNp$. If $DN\equiv 3\bmod 8$ then $\ZZ\oplus \ZZ \frac{1 +j}{2} \oplus \ZZ i \oplus \ZZ \left(\frac{1 + j}{2}\right)i$ is an order of reduced discriminant $DNp$, and is thus an Eichler order of level $N$. If $DN\equiv 5\bmod 8$ then $\ZZ \oplus \ZZ j \oplus \ZZ \left(\frac{1 + i + j}{2}\right) \oplus \ZZ \left(\frac{-1 -j + ij}{2}\right)$ is an order of reduced discriminant $DNp$.\end{proof}

\section{Shimura Curves}

We begin with the definition of a Shimura curve as a coarse moduli scheme, presuming some familiarity with abelian schemes and moduli spaces. As always, we will assume that $D$ is the squarefree product of an even number of primes and that $N$ is squarefree.

\begin{definition}\label{DefnCoarseModuliScheme} Fix a scheme $S$ and an Eichler order $\OO$ of level $N$ in $B_D$. By $X^D_0(N)_S$ we will denote the coarse moduli scheme parametrizing pairs $(A,\iota)$ over $S$-schemes $T$, where $A_{/T}$ is an abelian scheme and $\iota: \OO \hookrightarrow \End_T(A)$ is an optimal embedding such that the pair $(A,\iota)$ is \emph{mixed} in the sense of Ribet \cite{RibetBimodules} or \emph{special} in the sense of Drinfeld \cite{Dr76}.\end{definition}

\begin{definition}\label{AtkinLehnerOO}
Let $\beta_m$ denote a generator of the unique two-sided integral ideal of $\OO$ of norm $m|DN$. There is an automorphism $w_m$ of $X^D_0(N)_S$ induced by the bijection $ [(A,\iota)] \mapsto [(A,(\iota(\beta_m))\inv \iota(\cdot) \iota(\beta_m))]$.
\end{definition}

Note that the above makes sense because any generator of the unique two-sided integral ideal of $\OO$ of norm $m$ is of the form $\beta_m u$ where $u$ is a unit of $\OO$. Note that the group of all such $w_m$ is abelian because the group of two-sided integral ideals is abelian. We call the group of all such $w_m$ the \emph{Atkin-Lehner group} $W$ and note there is an isomorphism $(\ZZ/2\ZZ)^{\{p|DN\textrm{ prime}\}} \cong W$ by $m|DN \leftrightarrow \{p\mid m\} \mapsto w_m$.

\begin{definition}\label{ALOOFixed} We say that $(A,\iota)$ is fixed by $w_m$ if $[(A,\iota)] = [(A,(\iota(\beta_m))\inv \iota(\cdot) \iota(\beta_m))]$, where $\beta_m$ is a generator of the unique integral two-sided ideal of $\OO$ of norm $m$.\end{definition}
 

We now state some theorems on explicit descriptions of $X^D_0(N)_S$ over various schemes $S$.



\begin{definition} Let $D,N$ be positive square-free integers and let $\OO$ be an Eichler order of level $N$ in $B_{D}$. Define $\Pic(D,N)$ to be the set of isomorphism classes of right $\OO$-ideals. \end{definition}

Lemma \ref{QuaternionIndefinite} shows that $\Pic(D,N) = \{[\OO]\}$ when $B_D$ is indefinite. When $B_D$ is definite, there exist formulas for the size of $\Pic(D,N)$ \cite[Theorem 16]{Pizer}.

\begin{definition}\label{PicardLength} The \emph{length} of an element $[I]$ of $\Pic(D,N)$ is $\ell([I]) := \#(\OO_l(I)^\times/\pm 1)$.\end{definition}

We shall use this to make sense of the reduction $X^D_0(N)_{\FF_p}$ when $p|D$. We say a normal, proper, flat relative curve $M_{/\ZZ_p}$ is a \emph{Mumford curve} if each component of the special fiber is isomorphic over $\FF_p$ to $\PP^1_{\FF_p}$ and the intersection points are all $\FF_p$-rational double points. 
\begin{theorem}{\cite[Corollary 78]{Cl03}}\label{CerednikDrinfeldModelTheorem} Let $p \mid D$. There is a Mumford curve $M_{(D,N)/\ZZ_p}$ whose components over $\FF_p$ are in bijection with two copies of $\Pic(D/p,N)$ interchanged by an involution $a_p$ of $M_{(D,N)}$, whose intersection points are in bijection with $\Pic(D/p,Np)$, and whose dual graph is bipartite. Moreover let $x$ be an intersection point between two components of $(M_{(D,N)})_{\FF_p}$ corresponding to $[I]\in \Pic(D/p,Np)$. Then the following holds: $$\widehat{ \OO_{M_{(D,N)},x} } \cong \ZZ_p[[X,Y]]/(XY - p^{\ell([I])}).$$ 

Most importantly, there is an isomorphism $\phi:X^D_0(N)_{\ZZ_{p^2}}\stackrel{\sim}{\to} (M_{(D,N)})_{\ZZ_{p^2}}$ such that $\phi w_p = a_p\phi$. If $\langle \sigma \rangle =\Aut_{\ZZ_p}(\ZZ_{p^2})$, this isomorphism realizes $X^D_0(N)_{\ZZ_p}$ as the \'etale quotient of $(M_{(D,N)})_{\ZZ_{p^2}}$ by the action of $\sigma a_p$.\end{theorem} 


\begin{remark} Thinking of the dual graph in this way yields an algorithm to compute dual graphs which the author has implemented in \texttt{MAGMA}\cite{Magma}. If we fix $\OO \subset \OO^D$, it is possible to effectively compute representatives $\{I_1,\ldots, I_a\}$ for $\Pic(\OO^D)$ and $\{J_1, \ldots, J_b\}$ for $\Pic(\OO)$.
The \emph{origin} of $J_j$ is the unique $I_i$ such that $J_j\OO^D \cong I_i$. Also, via the \texttt{PrimeIdeal} command, we may compute the unique two-sided integral ideal $\wp$ of $\OO$. Therefore we may compute $w_p[J_j] = [J_j \wp]$ as in Theorem \ref{ALEquivariance}. The terminus of $J_j$ is then the origin of $[J_j \wp]$, determining the graph. It is also possible to determine the action of $w_m$ on this graph for all $m\mid DN$. \end{remark}

For a ring $A$ of characteristic $p$, let $W(A)$ denote the Witt vectors of $A$ \cite[\S II.4]{Neukirch}. Recall that $N$ is always assumed to be square-free.

\begin{theorem}\label{DeligneRapoportModelTheorem} If $p\nmid DN$ and $T$ is an $\FF_p$-scheme then $X^D_0(N)_{/T}$ is smooth. If $p \mid N$ and $T$ is an $\FF_p$-scheme then there is a closed embedding $c: X^D_0(N/p)_T \to X^D_0(N)_T$. Let $S = \Spec(R)$ be a flat $\ZZ_{(p)}$-scheme. If $T$ is an $S$-scheme and if $\Phi: X^D_0(N)_T \to X^D_0(N/p)_T$ is the forgetful map $X^D_0(N) \to X^D_0(N/p)$, then $\Phi c$ is the identity and $\Phi w_p c$ is the Frobenius map $(A,\iota) \mapsto (A^{(p)},\Frob_{p,*}\iota)$ (see Definition \ref{DefnFrob}). Moreover, $X^D_0(N)_T$ fits into the following diagram

\[
\xymatrix{
X^D_0(N/p)_T\ar@{->}[rd]^c \ar@{->}[dd]_{id} & & X^D_0(N/p)_T\ar@{->}[ld]_{w_p c}\ar@{->}[dd]^{id}\\
& X^D_0(N)_T\ar@{->}[ld]_{\Phi} \ar@{->}[rd]^{\Phi w_p} & \\
X^D_0(N/p)_T & & X^D_0(N/p)_T
}
\]


If $t$ is a closed point of $T$ such that $k(t) = \overline{k(t)}$, the intersection of $c(X^D_0(N/p)(k(t)))$ and $w_p c(X^D_0(N/p)(k(t)))$ is precisely the set of superspecial points (in the sense of Definition \ref{DefnSuperspecial}), which are in bijection with $\Pic(Dp,N/p)$. Moreover, for each superspecial point $x$ over $t$ corresponding to $[I]\in\Pic(Dp,N/p)$, the completion of the strict henselization of the local ring of $X^D_0(N)$ at $x$ is isomorphic to $R\otimes W(\overline \FF_p)[[ X,Y]]/(XY - p^{\ell([I])})$.

\end{theorem}

\begin{proof} The bijection between superspecial points and $\Pic(Dp,N/p)$ is Theorem \ref{RibetBimodulesTheorem}. The remainder of the result in the case of $\ZZ_{(p)}$ was first written down by Helm \cite[Theorem 10.3]{HelmPublished}.\end{proof} 

\begin{lemma}{\cite[Theorem 1.1]{MolinaBimodules}}\label{ComponentsIntersectionPointsLemma} The components and singular points of the $\overline\FF_p$ special fiber can be put into the following $W$-equivariant bijections.

\begin{center}\begin{tabular}{|c|c|c|}
              \hline & Components & Intersection Points \\
              \hline $p \mid D$ & $\Pic(D/p,N) \coprod \Pic(D/p,N)$ & $\Pic(D/p,Np)$ \\
             \hline $p \mid N$ & $\Pic(D,N/p) \coprod \Pic(D,N/p)$ & $\Pic(Dp,N/p)$ \\
              \hline
             \end{tabular}\end{center}
             
If $p\mid D$, the bijection of a set of components with $\Pic(D/p,N)$ is $W/\langle w_p\rangle$-equivariant with $w_p$ interchanging each. If $p\nmid DN$, the superspecial points of $X^D_0(N)_{\overline\FF_p}$ can be put into $W$-equivariant bijection with $\Pic(Dp,N)$ via the embedding $c:X^D_0(N)_{\overline\FF_p} \to X^D_0(Np)_{\overline\FF_p}$.
\end{lemma}


\subsection{Superspecial surfaces}\label{superspecialsection}

Fix a prime number $p$ and a maximal order $\mathcal S$ in the quaternion algebra $B_p$ over $\QQ$ ramified precisely at $p$ and $\infty$. By a theorem of Deuring, there is a supersingular elliptic curve $E$ over the algebraic closure $\FF$ of $\FF_{p}$ such that $\End_\FF(E) \cong \mathcal S$ \cite[p.23]{RibetBimodules}. 

\begin{definition}\label{DefnSupersingular} Fix $E_{/\FF}$, a \emph{supersingular} elliptic curve with $\End_\FF(E) \cong \mathcal S$. We say that an abelian variety $A_{/\FF}$ is supersingular when there is an isogeny $A \to E^{\dim(A)}$.\end{definition}

Note that if $E'_\FF$ is supersingular then $E$ is isogenous to $E'$ so the above definition does not depend on the choice of $E$.

\begin{lemma}{\cite[Theorem 68]{Cl03}, \cite[Lemma 4.1]{RibetBimodules}}\label{ordinaryorsupersingular} If $A_{/\FF_q}$ is an abelian surface over a finite field and $B_D \hookrightarrow \End^0_{\FF_q}(A)$, $A$ is isogenous over $\FF_q$ to the square of an elliptic curve $(E_0)_{\FF_q}$. Moreover if $p \mid D$ this elliptic curve must be supersingular.\end{lemma}

\begin{definition}\label{DefnSuperspecial} We say that an abelian surface $A_{/\FF}$ is \emph{superspecial} if $A\cong E_i \times E_j$ with $E_i, E_j$ supersingular elliptic curves over $\FF$.\end{definition}

\begin{lemma}{\cite[p. 21-22]{RibetBimodules}}\label{pnmidDsupersingular} Suppose that $A$ is a supersingular abelian $\OO$-surface over $\FF$ with $p\nmid D$. Then $A$ is superspecial.\end{lemma}

Note that if $A$ is supersingular, it need not be superspecial. When $A$ is ordinary, we have the following.

\begin{theorem}\label{OrdinaryCM} If $(A_{/k},\iota)$ is an ordinary QM-abelian surface over a finite field $k$, then there exist ordinary elliptic curves $E_0, E_0'$ over $k$ such that $A \cong E_0 \times E_0'$. If $m>1$ then $(A,\iota)$ is $w_m$-fixed (see Definition \ref{ALOOFixed}) if and only if $\End_k(E_0)\cong_k \End_k (E'_0)$ and $\End_k(E_0)$ is isomorphic to one of $\ZZ[\sqrt{-m}]$ or $\ZZ[\frac{1 + \sqrt{-m}}{2}]$. \end{theorem}

\begin{proof} The first part of the statement is part of a more general theorem of Kani \cite[Theorem 2]{Kani}, who calls ordinary elliptic curves CM. For the second part, note that $(A_{/S},\iota)$ is $w_m$-fixed if and only if $R = \ZZ[\sqrt{-m}]$ (or $\ZZ[\zeta_4]$ if $m=2$) embeds into the commutant of $\iota(\OO)$ in $\End_S(A)$. 

Let $k$ be a finite field, $A_{/k}$ be ordinary, and $(A,\iota)$ be $w_m$-fixed. Also let $W(k)$ denote the Witt vectors of $k$ \cite[\S II.4]{Neukirch}, which in this case are just a finite \'etale extension of $\ZZ_p$. Then there is a canonical choice of an abelian scheme $\CA_{W(k)}$ with an isomorphism $f: \End_k(A) \stackrel{\sim}{\to} \End_{W(k)}(\CA)$ \cite[Theorem V.3.3]{Messing}. Therefore the Serre-Tate canonical lift $(\CA, f\circ \iota)$ is a QM-abelian surface. Therefore so is $\CA_{\CC}$ (the choice of embedding $W(k)\hookrightarrow \CC$ does not change the isomorphism class of $\CA_\CC$ \cite[7.Th\'eor\`eme]{VAO}), and there is an embedding of $R$ into $\End_{f(\iota(\OO))}(\CA_\CC)$. Then we may find both an optimal embedding $\varphi: R'\hookrightarrow \OO$ for some imaginary quadratic order $R'\supset R$ and an isomorphism $\CA_\CC \cong E_1 \times E_2$ where the $E_i$'s have CM by $R'$ and $f\circ \iota$ is given by $\varphi$ \cite[p. 6]{MolinaBimodules}.

Now let $K := W(k) \otimes \QQ$, which must therefore be a finite unramified extension of $\QQ_p$. We can then show that $\CA_K \cong E_1' \times E_2'$ where $E_i' \otimes \CC \cong E_i$ \cite[Lemma 60]{Kani}. Moreover, each $E_i'$ has CM by $R'$ since $\OO \hookrightarrow \End_K(\CA_K)$ and we have $\varphi : R'\hookrightarrow \OO$. Now, if $V$ is an abelian variety over $K$, let $\mathfrak{N}(V)$ denote its N\'eron model over $W(k)$ \cite[Definition I.2.1]{BLR}. It follows that since $\CA$ is an abelian scheme, it is the N\'eron model of its generic fiber \cite[Proposition I.2.8]{BLR}, and thus $\CA  \cong  \mathfrak{N}(\CA_K)  \cong  \mathfrak{N}(E_1' \times E_2') \cong \mathfrak{N}(E_1') \times \mathfrak{N}(E_2').$\end{proof} 


\begin{theorem}{\cite[Theorem 3.5]{Shioda}}Let  $E_{/\FF}$ be as in Definition \ref{DefnSupersingular} and let $A_{/\FF}$ be an abelian surface isomorphic to the product of any two supersingular elliptic curves. Then $A\cong E\times E$.\end{theorem}



Recall that $\mathcal S$ is a maximal order in $B_p$ and $p\mid D$. Recall also that an $(\OO,\mathcal S)$-bi-module is a left $\OO$-module $M$ which is also a right $\mathcal S$-module such that if $x\in \OO$, $y\in \mathcal S$, and $m\in M$, then $(xm)y = x(my)$. This implies that we have homomorphisms $\OO \to \End_{\mathcal S}(M)$ and $\mathcal S^{\mathrm{op}} \to \End_\OO(M)$. If both of these homomorphisms are optimal embeddings we say that $M$ is an \emph{optimal} $(\OO,\mathcal S)$ bi-module.

\begin{theorem}{\cite[p.38]{RibetBimodules}}\label{RibetBimodulesTheorem} Suppose that $\OO$ is an Eichler order of square-free level $N$ in an indefinite quaternion algebra $B$ of discriminant $D$ with $(D,N)=1$. There is a bijection between the following sets.

\begin{itemize}

\item superspecial $\OO$-abelian surfaces $(A,\iota)_{/\FF}$ up to isomorphism

\item $\ZZ$-rank 8 optimal $(\OO,\mathcal S)$ bi-modules up to isomorphism

\end{itemize}

\end{theorem}


\begin{lemma}\label{ALRibetLemma} Let $q|DN$ and let $\mathfrak Q$ denote the unique two-sided integral ideal of norm $q$ in $\OO$. Under the bijection in Theorem \ref{RibetBimodulesTheorem}, the action of $w_q$ described in Definition \ref{AtkinLehnerOO} corresponds to the action $M \mapsto \mathfrak Q \otimes _\OO M$.\end{lemma}

\begin{proof}The bi-module $ \mathfrak Q \otimes _\OO M$ is isomorphic to $\beta_q M$ as an $(\OO, \mathcal S)$-bi-module since $\mathfrak Q = \beta_q\OO = \OO\beta_q$. Therefore to get an action of $\OO$ on $\beta_q M$, we must pre-compose by $\beta_q\inv$ and post-compose by $\beta_q$.\end{proof}

\begin{definition}Let $\OO,\mathcal S$ be Eichler orders in a quaternion algebra over a number field $K$. We say that two $(\OO,\mathcal S)$-bi-modules $M,N$ are \emph{locally isomorphic} if for all places $v$ of $K$, $M_v \cong N_v$ as $(\OO_v,\mathcal S_v)$-bi-modules.\end{definition}

\begin{remark} It is choosing a local isomorphism class of a bi-module that allows us to keep track of whether or not a surface $(A,\iota)$ is mixed \cite[p.39]{RibetBimodules}.\end{remark}

\begin{theorem}\label{ALEquivariance}Let $\OO,\mathcal S$ be as in Theorem \ref{RibetBimodulesTheorem} and fix an $(\OO,\mathcal S)$-bi-module $M$. Then $\Lambda := \End_{\OO,\mathcal S}(M)$ is an Eichler order in either $B_{Dp}$ if $p\nmid D$ or $B_{D/p}$ if $p\mid D$. Moreover, if we fix a bi-module $M$, there is a bijection between the following two sets 

\begin{itemize}

\item $(\OO,\mathcal S)$-bi-modules $N$ locally isomorphic to $M$ up to isomorphism and

\item Rank one projective right $\Lambda$ modules up to isomorphism.

\end{itemize}

Let $q\ne p$ be prime. This bijection sends the action described in Lemma \ref{ALRibetLemma} to the action $[I]\mapsto [I\mathfrak Q_\Lambda]$, where $\mathfrak Q_\Lambda$ is the unique two-sided ideal of norm $q$ of $\Lambda$.\end{theorem}

\begin{proof} The bijection in the case where $\OO$ is a maximal order is a theorem of Ribet \cite[Theorem 2.3]{RibetBimodules}. The extension to Eichler orders is due to Molina \cite[Remark 4.11]{MolinaBimodules}. His proof depends on showing that $\Hom_{\OO,\mathcal S}(N,\mathfrak Q_\OO \otimes N)$ is $\mathfrak Q_\Lambda$.\end{proof}

\begin{definition} Retaining the notation of Theorem \ref{ALEquivariance}, the action $[I]\mapsto [I\mathfrak{Q}_\Lambda]$ will be referred to as $w_q$. Moreover if $m$ is the product of primes ramified in $\Lambda$, we define $w_m$ as the composition of all $w_q$ ranging over $q \mid m$.\end{definition}

\begin{corollary}\label{EmbeddingFixedPoint} Let $m > 1$. A superspecial $\OO$-abelian surface $(A,\iota)$ with corresponding bi-module $M$ is fixed under the action of $w_m$ if and only if there is an embedding of $\ZZ[\sqrt{-m}]$ $($or $\ZZ[\zeta_4]$ if $m=2)$ into $\Lambda = \End_{\OO,\mathcal S}(M)$.\end{corollary}

\begin{proof}By Theorem \ref{ALEquivariance}, $(A,\iota)$ is fixed by the action of $w_m$ if and only if $[\prod_{q \mid m} \mathfrak{Q}_\Lambda] = [1]$, which is to say if and only if the unique two-sided ideal of norm $m$ is principal. Therefore there is a fixed point if and only if there is an element $\gamma$ of $\End_{\OO,\mathcal S}(M)$ which can serve as the principal generator. That is, $\gamma^2\Lambda = m\Lambda$ so there is a unit $u$ of $\Lambda$ such that $\gamma^2 = um$. Therefore, $u \in \ZZ_F$ where $F = \QQ(\gamma)$, an imaginary quadratic extension of $\QQ$. Following Kurihara \cite[Proposition 4-4]{Kurihara}, $u \ne 1$ since $\Lambda$ is definite, $u^2 +1 =0$ can only happen if $m=2$, and $u^2 \pm u +1 =0 $ can only happen if $m=3$. This exhausts all possibilities since $\QQ(u) \subset F$. If $u^2 +1 = 0$ then $\ZZ[u] \cong \ZZ[\gamma]$ with $u \mapsto \gamma + 1$. If $u^2 \pm u +1 = 0$ then $\ZZ[u] \cong \ZZ[\gamma]$ with $u \mapsto \gamma \pm 1$.\end{proof}

This is of particular interest to us because of the following lemma.

\begin{lemma}\label{ALGaloisLemma} If $(A,\iota)$ is a superspecial abelian $\OO$-surface over $\FF$, then $w_p(A,\iota)$ (in the sense of Theorem \ref{ALEquivariance}) is its $\FF_{p^2}/\FF_p$-Galois conjugate. Equivalently, if $P:\Spec(\FF) \to X^D_0(N)$ corresponds to a superspecial abelian $\OO$-surface $(A,\iota)$ over $\FF$ and $\phi_1: \FF \to \FF$ is the $p$-th power map, the following diagram commutes. \[\xymatrix{
\Spec(\FF) \ar@{->}[r]^P \ar@{->}[d]^{\phi_1^*} & X^D_0(N) \ar@{->}[d]^{w_p} \\
\Spec(\FF) \ar@{->}[r]^P & X^D_0(N)
}\]

\end{lemma}

\begin{proof} If $p\mid D$, then for \emph{all} points $P:\Spec(\FF) \to X^D_0(N)$, the square of this Lemma commutes by Theorem \ref{CerednikDrinfeldModelTheorem}. If $p\mid N$, and $P:\Spec(\FF) \to X^D_0(N)$ corresponds to an abelian $\OO$-surface $(A_\FF,\iota)$ then by Theorem \ref{DeligneRapoportModelTheorem}, $w_p P$ corresponds to $(A^{(p)},\Frob_{p,*}\iota)$. By Lemma \ref{FrobeniusGaloisLemma}, this corresponds to the point $P\phi_1^*$. If $p\nmid DN$, we can reduce to the case $p \mid N$ via the embedding $c: X^D_0(N)_{\FF} \to X^D_0(Np)_{\FF}$.\end{proof}


\begin{definition} Let $(A,\iota)$ be a superspecial $\OO$-abelian surface over $\FF$ with corresponding bi-module $M$. The \emph{length} of $(A,\iota)$ is $\#(\End_{(\OO,\mathcal S)}(M)^\times/\pm 1)$. \end{definition}

Note that $\End_{(\OO,\mathcal S)}(M) \cong \End_\FF(A,\iota)$ \cite[Equation 3.5]{MolinaBimodules}. Therefore if $(A,\iota)$ corresponds to a point of $X^D_0(N)(\FF)$ then this definition agrees with Definition \ref{PicardLength}.

\begin{corollary}\label{SuperspecialFixedZetaSix} Let $(A,\iota)$ be a mixed superspecial $\OO$-abelian surface with corresponding bi-module $M$ and whose length is divisible by three. Let $N'$ be the level of $\OO' = \End_{(\OO,\mathcal S)}(M)$ and $D'$ the discriminant of $\OO' \otimes \QQ$. Then for all $p \mid D'$, $p=3$ or $p\equiv 2\bmod 3$, and for all $q \mid N'$, $q=3$ or $q\equiv 1\bmod 3$. Moreover, $(A,\iota)$ is fixed by $w_m$ if and only if $m=1,3,D'N'$ or $D'N'/3$ if $3\mid D'N'$. \end{corollary}

\begin{proof} Unless $D' = 2,3$ and $N'=1$, the only possible such length is three \cite[Proposition V.3.1]{Vigneras}. In each of those cases, if $p\mid D'$ then $p=2$ or $p=3$. If $(D',N') \ne (2,1),(3,1)$, the length of $(A,\iota)$ is three if and only if $\ZZ[\zeta_6]\hookrightarrow \OO'$ and the first part of our statement holds by Theorem \ref{EichlerEmbedding}.

Recall now that any $(A,\iota)$ is fixed by $w_1$. If $\ZZ[\zeta_6]$ embeds into $\OO'$ note that $\ZZ[\sqrt{-3}] \subset \ZZ[\zeta_6]\hookrightarrow \OO'$ so $(A,\iota)$ is fixed by $w_3$ if $3 \mid D'N'$. Now suppose that $m\ne 3$ so  $\ZZ[\zeta_6]$ does not contain $\ZZ[\sqrt{-m}]$ and vice versa. In that case we have simultaneous embeddings if and only if $m= D'N'$ or if $3 \mid D'N'$ and $m= D'N'/3$ by Theorem \ref{zetasix}.\end{proof}

\begin{corollary}\label{SuperspecialFixedZetaFour}Let $(A,\iota)$ be a mixed superspecial $\OO$-abelian surface with corresponding bi-module $M$ and whose length is even. Let $N'$ be the level of $\OO' = \End_{(\OO,\mathcal S)}(M)$ and $D'$ the discriminant of $\OO' \otimes \QQ$. Then for all $p \mid D'$, $p=2$ or $p\equiv 3\bmod 4$, and for all $q\mid N'$, $q=2$ or $q\equiv 1\bmod 4$. Moreover, $(A',\iota')$ is fixed by $w_m$ if and only if $m=1,2, D'N'$ or $D'N'/2$ if $2\mid D'N'$. \end{corollary}

\begin{proof} Recall that unless $D' = 2,3$ and $N'=1$, the only possible even length is two \cite[Proposition V.3.1]{Vigneras}. In each of those cases our conditions hold. If $(D',N') \ne (2,1),(3,1)$, the length of $(A,\iota)$ is two if and only if $\ZZ[\zeta_4]\hookrightarrow \OO'$ and the first part of our statement holds by Theorem \ref{EichlerEmbedding}.

Recall now that any $(A,\iota)$ is fixed by $w_1$. If we have $\ZZ[\zeta_4]\hookrightarrow \OO'$ then $(A,\iota)$ is fixed by $w_2$ if $2 \mid D'N'$ by Corollary \ref{EmbeddingFixedPoint}. Now suppose that $m>2$ so $\ZZ[\zeta_4]$ does not contain $\ZZ[\sqrt{-m}]$ and vice versa. In that case we have simultaneous embeddings if and only if $m= D'N'$ or if $2 \mid D'N'$ and $m= D'N'/2$ by Theorem \ref{zetafour}.\end{proof}


\begin{corollary}\label{SuperspecialFpRationalALFixed} Let $\OO$ be an Eichler order of square-free level $N$ in $B_D$ where $D$ is the square-free product of an even number of primes and $N$ is coprime to $D$. Let $m \mid DN$ and let $p$ be a prime not dividing $DN$. If $p=2$ there is a mixed superspecial abelian $\OO$ surface $(A_{\FF_p},\iota)$ fixed by $w_m$ if and only if one of the following occurs.

\begin{enumerate}

\item $m = DN$, $q\equiv 3\bmod 4$ for all $q \mid D$, and $q\equiv 1\bmod 4$ for all $q \mid N$.

\item $m = DN \equiv \pm 3\bmod 8$, $\left(\frac{-2}{q}\right) = -1$ for all primes $q \mid D$, and $\left(\frac{-2}{q}\right) = 1$ for all primes $q \mid N$.

\end{enumerate}

If $p \ne 2$, there is a mixed superspecial abelian $\OO$ surface $(A_{\FF_p},\iota)$ fixed by $w_m$ if and only if one of the following occurs.

\begin{enumerate}


\item $2\nmid D$, $m= DN$, $\left(\frac{-DN}{p}\right) = -1$, $\left(\frac{-p}{q}\right) = -1$ for all $q \mid D$, and $\left(\frac{-p}{q}\right) = 1$ for all $q \mid N$ such that $q\ne 2$.

\item $2 \mid N$, $m = DN/2$, $\left(\frac{-DN/2}{p}\right) = -1$, $\left(\frac{-p}{q}\right) = -1$ for all $q \mid D$, and $\left(\frac{-p}{q}\right) = 1$ for all $q \mid N$ such that $q\ne 2$.

\item $2 \mid D$, $m= DN$, $p\equiv \pm 3 \bmod 8$, $\left(\frac{-DN}{p}\right) = -1$, $\left(\frac{-p}{q}\right)= -1$ for all $q \mid (D/2)$, and $\left(\frac{-p}{q}\right) = 1$ for all $q \mid N$.

\item $2 \mid D$, $m = DN/2$, $DN \equiv 2,6,10 \bmod 16$, $p\equiv \pm 3\bmod 8$, $\left(\frac{-DN/2}{p}\right) = -1$, $\left(\frac{-p}{q}\right) = -1$ for all $q \mid D$, and $\left(\frac{-p}{q}\right) = 1$ for all $q \mid N$.

\end{enumerate}
\end{corollary}


\begin{proof} By Lemma \ref{ALGaloisLemma}, a superspecial abelian surface $(A_{/\FF_{p^2}},\iota)$ is defined over $\FF_p$ if and only if it is $w_p$-fixed. Therefore there is some $(A_{\FF_p},\iota)$ fixed by $w_m$ if and only if there is some Eichler order of level $N$ in $B_{Dp}$ which admits an embedding of both $\ZZ[\sqrt{-m}]$ (or $\ZZ[\zeta_4]$ if $m=2$) and $\ZZ[\sqrt{-p}]$ (or $\ZZ[\zeta_4]$ if $p=2$).

Let us first assume $p=2$. Condition 1 is precisely Corollary \ref{SuperspecialFixedZetaFour} applied to the situation where $(m,2) =1$. Condition 2 is Theorem \ref{GoodReductionSimultaneousEmbeddings}(5).

Now let us assume $p\ne 2$. Conditions 1 and 2 are Theorem \ref{GoodReductionSimultaneousEmbeddings}(1-2) and conditions 3 and 4 are Theorem \ref{GoodReductionSimultaneousEmbeddings}(3-4).\end{proof}

\section{Local points at good primes}

Throughout this section we will fix $D$ the discriminant of an indefinite quaternion $\QQ$-algebra, $N$ a square-free integer coprime to $D$, an integer $m\mid DN$ and a prime $p\nmid DN$. Recall that $X^D_0(N)_{/\ZZ_p}$ has a smooth special fiber by Theorem \ref{DeligneRapoportModelTheorem}. Let $w_m$ be as in Definition \ref{AtkinLehnerOO}. Let $\ZZ_{p^2}$ be as in Definition \ref{DefnZp2} with $\langle \sigma\rangle = \Aut_{\ZZ_p}(\ZZ_{p^2})$ and let $\CZ_{/\ZZ_p}$ denote the quotient of $X^D_0(N)_{\ZZ_{p^2}}$ by the action of $w_m \sigma$. 

If $p$ is split in $\QQ(\sqrt d)$, then $X^D_0(N)$ is isomorphic to $C^D(N,d,m)$ over $\QQ_p$. We may then obtain results on local points without appealing to $\CZ$.

If $p$ is inert in $\QQ(\sqrt d)$ and $C^D(N,d,m)_{/\QQ}$ is the twist of $X^D_0(N)_{/\QQ}$ by $w_m$ and $\QQ(\sqrt d)$ then $\CZ$ is a $\ZZ_p$-model for $C^D(N,d,m)_{\QQ_p}$. This is because it follows from applying the theorem on \'etale base change \cite[Proposition 10.1.21(c)]{Liu} to the map $X^D_0(N)_{\ZZ_{p^2}}$ that $\CZ_{\FF_p}$ is also smooth. 

Some easy results present themselves. For instance we may use Weil's bounds to show that we have $p$-adic points for all but finitely many primes $p$. Unless otherwise noted, assume that $g$ is the genus of $X^D_0(N)_{/\QQ}$.

\begin{theorem}\label{WeilBoundsTwist}Suppose that $p$ is unramified in $\QQ(\sqrt d)$ and $p>4g^2$. It follows that  $C^D(N,d,m)(\QQ_p)\ne\emptyset$.\end{theorem}

\begin{proof} Recall that Weil's bounds \cite[Exercise 9.1.15]{Liu} tell us that if $X$ is smooth of genus $g$ over $\FF_p$ then $$\mid \#X(\FF_p) - (p+1)\mid  \le 2g\sqrt p,$$ and thus $\#X(\FF_p) \ge p+1 - 2g\sqrt p > 4g^2 - 4g^2 +1 = 1$. Hensel's Lemma tells us that if we let $\CZ_{/\ZZ_p}$ be a regular model of $C^D(N,d,m)_{\QQ_p}$ and set $X = \CZ_{\FF_p}$ then $C^D(N,d,m)(\QQ_p) = \CZ(\QQ_p)$ is nonempty since $g = g(C^D(N,d,m)_{\FF_p})$.\end{proof}



For $p< 4g^2$, we must use another technique. In the split case we use Shimura's construction of the zeta function of $X^D_0(N)_{\FF_p}$ using Hecke operators to give an exact formula for the size of $X^D_0(N)(\FF_p)$. In the inert case, we give a partial answer in terms of superspecial points.

\begin{definition} Let $S$ be an $\FF_p$-scheme and let $A_{/S}$ be an abelian scheme. Let $\Frob_{p^r}: A \to A^{(p^r)}$ and $\Ver_{p^r}: A^{(p^r)} \to A$ be the Frobenius and Verschiebung isogenies, so that $\Frob_{p^r} \Ver_{p^r} = \Ver_{p^r}\Frob_{p^r} = [p^r]$ on $A^{(p^r)}$ and $A$ respectively.\end{definition}

\begin{definition}\label{DefnFrob} Let $S$ be an $\FF_p$-scheme and let $(A,\iota)$ be an abelian $\OO$-surface. By $\Frob_{p^r,*}\iota$ we denote the unique optimal embedding $\OO \hookrightarrow \End_S(A^{(p^r)})$ such that for all $\alpha \in \OO$ the following commutes.
\[\xymatrixcolsep{5pc}\xymatrix{ A \ar[r]^{\iota(\alpha)}\ar[d]_{\Frob_{p^r}} & A \ar[d]^{\Frob_{p^r}} \\ A^{(p^r)} \ar[r]^{\Frob_{p^r,*}\iota(\alpha)} & A^{(p^r)}}\] 
\end{definition}

\begin{lemma}\label{FrobeniusGaloisLemma} Let $S = \Spec(\overline \FF_p)$ and $\phi_r:S\to S$ be the morphism given by the $p^r$-th power map. Let $(A_{/S},\iota)$ be a QM-abelian surface corresponding to a point $P: S \to X^D_0(N)_S$. Let $P\circ \phi_r: S \to S \to X^D_0(N)_S$ denote the Galois conjugate point. Then the QM-abelian surface corresponding to $P\circ\phi_r$ is $(A^{(p^r)},\Frob_{p^r,*}\iota)$. 
\end{lemma}


\begin{proof} Fix an Eichler order $\OO$ of level $N$ in $B_D$. Note that $(A^{(p^r)},\Frob_{p^r,*}\iota) = (A^{(p^r)},\Ver_{p^r}^*\iota)$ where $\Ver_{p^r}^*\iota$ is defined in the obvious way. Since $\Ver_{p^r}$ itself is the pullback of $\phi_r$ along $A\to S$ \cite[p.94]{Liu} we obtain our result.\end{proof} 

\subsection{Split Primes and the Eichler-Selberg Trace Formula}

\begin{definition}\label{HeckeOperators} Let $S$ be a $\ZZ_p$-scheme with $p\nmid DN$. Let $X^D_0(N)$ be defined over $S$. If $(n, DN)=1$, $T_n$ is the correspondence 

\[
\xymatrix{ 
 & X^D_0(Nn)_S \ar@{->}[ld]^{\Phi_1} \ar@{->}[rd]_{\Phi_2} &\\
X^D_0(N)_S & & X^D_0(N)_S
}
\]

where $\Phi_1$ is the modular forgetful map and $\Phi_2 = \Phi_1 \circ w_n$. \end{definition}

The correspondences $T_n$ are commonly known as \emph{Hecke correspondences}. Let $s$ be a closed point of $S$ with $k(s) = \overline{k(s)}$ so that $X^D_0(N)_s$ has a $k(s)$-rational point and thus correspondences on $X^D_0(N)$ are in bijection with endomorphisms of $J^D_0(N)_s$ \cite[Corollary 6.3]{MilneJac}. We may also use $T_n$ to denote the endomorphism of $J^D_0(N)_s \cong J(X^D_0(N)_s)$ induced by the map of sets $X^D_0(N)_s \to {\rm Div}(X^D_0(N)_s)$ $P\mapsto (\Phi_{2,*}\Phi_1^*) P$. This operator on $J^D_0(N)_s$ is commonly referred to as a \emph{Hecke operator}. We will explore the case $(n,DN)>1$ in section \ref{sectionRSY}.

\begin{theorem}[Eichler-Shimura]\label{EichlerShimuraTheorem} There is an equality of endomorphisms of $J^D_0(N)_{s}$ between $T_p$ and $\Frob_p + \Ver_p$.\end{theorem}

\begin{proof} The particularly simple proof given below was sketched by Stein in the case of the elliptic modular curve $X^1_0(N)$ \cite[Theorem 12.6.4]{RibetStein}, and the same proof also works for Shimura Curves.\end{proof} 





\begin{definition} If $C_{\FF_p}$ is a smooth, projective curve, we may define the zeta function of $C$ as $$ Z(C,x) := \mathrm{exp}\left(\sum_{r=1}^\infty \#C(\FF_{p^r}) \frac{x^r}{r}\right).$$\end{definition}

Shimura \cite{Sh67} proved the following explicit formula for the zeta function.

\begin{theorem}\label{ShimuraZetaFunction} If $\Omega$ denotes the canonical sheaf on $X^D_0(N)$, then \begin{equation} Z(X^D_0(N)_{\FF_p},x) = \frac{\det_{H^0(X^D_0(N),\Omega)}(I_g - T_p x + px^2 I_g)}{(1-x)(1-px)}.\end{equation}\end{theorem}

\begin{proof} First we note that if $R$ is a characteristic $(0,p)$ dvr with separably closed residue field $k$ and $\ell \ne p$ is a prime of good reduction for a relative curve $\CX_{/R}$ then,  by smooth and proper base change \cite[Corollary VI.4.2]{LEC}, $H^1(\CX_k,\QQ_\ell)  \cong H^0(\CX_k,\Omega)\oplus H^0(\CX_k,\Omega)^\vee$.




Now we invoke the Weil Conjectures for curves \cite[Corollary V.2.6]{LEC}. That is, $$Z_p(X^D_0(N)_s, x) = \prod_{i=0}^{2} \left(\det\left(I- x\Frob_p\right)\mid _{H^i(X^D_0(N)_s, \QQ_\ell)}\right)^{(-1)^{(1+i)}}.$$

Moreover, since $\dim X^D_0(N)_s =1$, $(I- x\Frob_p)\mid_{H^0(X^D_0(N)_s,\QQ_\ell)} = (1-x)$ and $(I- x\Frob_p)\mid_{H^2(X^D_0(N)_s,\QQ_\ell)} = (1-px)$. However, since $H^1(X^D_0(N)_s,\QQ_\ell) \cong H^0(X^D_0(N)_s,\Omega) \oplus H^0(X^D_0(N)_s,\Omega)^\vee \cong H^0(J^D_0(N)_s,\Omega) \oplus H^0(J^D_0(N)_s,\Omega)^\vee$\cite[Proposition 2.2]{MilneJac}, we have $(I_{2g}- x\Frob_p)\mid_{H^1(X^D_0(N)_s,\QQ_\ell)}$ equal to \begin{eqnarray*} & = & (I_g - x\Frob_p)(I_g-x\Frob_p^\vee)\mid _{H^0(J^D_0(N)_s,\Omega)}\\ & = & (I_g - x(\Frob_p + \Ver_p ) + x^2\Frob_p\Ver_p )\mid_{H^0(J^D_0(N),\Omega)}  \\ & =& (I_g - T_p x + px^2I_g)\mid_{H^0(X^D_0(N),\Omega)}.\end{eqnarray*} \end{proof}

\begin{corollary}{\cite[Proposition 2.1]{JoLi}} If $r>1$ then \begin{equation} \# X^D_0(N)(\FF_{p^r})) =  p^r +1 - \tr(T_{p^r}) + p\tr(T_{p^{r-2}})\end{equation} and if $r=1$,\begin{equation} \# X^D_0(N)(\FF_{p})) =  p +1 - \tr(T_{p})\end{equation}  \end{corollary}


Let $\sigma_1$ as the usual divisor sum function. Let $w,f$ be as in Definition \ref{QuadraticRingDiscDefn} and $e_{D,N}$ be as in Definition \ref{DefnEDN}.

\begin{theorem}\label{EichlerSelbergTraceFormula}[Eichler's Trace Formula, \cite[\S 4]{Eichler}] Let $D$ be the discriminant of an indefinite rational quaternion algebra, $N$ a square-free integer coprime to $D$ and $\ell$ a prime not dividing $DN$. Let $\tr(T_n)$ denote the trace of $T_n$ on $H^0(X^D_0(N)_\CC,\Omega)$.

If $n$ is not a square and $(n,DN) =1$, then 

\begin{equation} \tr(T_n) = \sigma_1(n) - \sum_{s = -\lfloor 2 \sqrt{n}\rfloor}^{\lfloor 2\sqrt{n}\rfloor}\sum_{f\mid f(s^2 - 4n)} \frac{e_{D,N}\left(\frac{s^2 - 4n}{f^2}\right)}{w\left(\frac{s^2 - 4n}{f^2}\right)}.\end{equation} 

\end{theorem}

\begin{corollary} $$\#X^D_0(N)(\FF_p) =  \sum_{s = -\lfloor 2\sqrt{p}\rfloor}^{\lfloor 2\sqrt{p}\rfloor}\sum_{f\mid f(s^2 - 4p)} \frac{e_{D,N}\left(\frac{s^2 - 4p}{f^2}\right)}{w\left(\frac{s^2 - 4p}{f^2}\right)}$$ 
\end{corollary}



\subsection{Inert primes and the Eichler-Selberg trace formula}\label{sectionRSY}

We shall briefly follow Rotger, Skorobogatov and Yafaev \cite[\S 2]{RSY} to obtain a formula for the number of points of $C^D(N,d,m)(\FF_p)$. This will not give a strict numerical criterion for the presence or absence of points, but it will give an exact formula as we will see in Theorem \ref{NumberPointsGoodReductionTwist}. In certain cases however, such as when $m=DN$, we will be able to use the properties of superspecial points to get a numerical criterion, as in Corollary \ref{SuperspecialFullAL}. We begin by extending the definition of Hecke operators $T_n$.

Suppose that $(DN, \frac{n}{(n,DN)}) = 1$, $m = (n,DN)|DN$ and $n' = \frac{n}{(n,DN)}$. Let $S$ be a $\ZZ_p$-scheme and $\Phi_1: X^D_0(Nn')_S \to X^D_0(N)_S$ be the forgetful map. By abuse of notation, let $w_m$ denote the Atkin-Lehner involution on either $X^D_0(Nn')_S$ or $X^D_0(N)_S$. Note that $\Phi_1 w_m = w_m \Phi_1$, so if $s$ is a closed point of $S$ with $k(s) = \overline{k(s)}$, $T_{n'}w_m = w_mT_{n'}: X^D_0(N)_s \to \Div(X^D_0(N)_s)$.

\begin{definition}Suppose that $(DN, \frac{n}{(n,DN)}) = 1$, $m = (n,DN)|DN$ and $n' = \frac{n}{(n,DN)}$. Then define $T_n = w_mT_{n'}$.\end{definition}

Let $m\mid DN$ and consider the quotient $(X^D_0(N)/w_m)_s$. Let $\Omega$ denote the canonical sheaf of $(X^D_0(N)_s$. Since $w_m$ is an involution, $H^0(X^D_0(N)_s, \Omega)$ decomposes into the direct sum of the $+1$ and $-1$ eigenspaces under its action. Note that $H^0((X^D_0(N)/w_m)_s, \Omega)$ is the $+1$ eigenspace.

Suppose that $v\in H^0(X^D_0(N)_s, \Omega)$ such that $w_m v = v$. Then $w_m T_p v = T_p w_m v = T_p v$ and therefore $T_p$ acts on $H^0((X^D_0(N)/w_m)_s, \Omega)$.

\begin{definition}If $p\nmid DN$ and $m|DN$, then by $T^{(m)}_p$ we denote the restriction of $T_p$ to $H^0((X^D_0(N)/w_m)_s, \Omega)$.\end{definition}

Note that since $T^{(m)}_p$ is just $T_p$ on a smaller vector space, $T^{(m)}_p = \Frob_p + \Ver_p$ on ${\rm Jac}((X^D_0(N)/w_m)_s)$ by Theorem \ref{EichlerShimuraTheorem}.



\begin{corollary} Let $g'$ be the genus of $(X^D_0(N))/w_m)_{\FF_p}$. The zeta function of the quotient curve is $$ Z_p(X^D_0(N)/w_m, x) = \frac{ \det_{H^0(X^D_0(N)/w_m,\Omega)}(I_{g'}  - T^{(m)}_p x + px^2 I_{g'})}{(1-x)(1-px)}.$$ 
\end{corollary}

\begin{proof} Since $T^{(m)}_p = \Frob_p + \Ver_p$ on ${\rm Jac}((X^D_0(N)/w_m)_s)$, 
we may reuse the proof of Theorem \ref{ShimuraZetaFunction}.\end{proof}

We may thus see that if $r>1$ then $$\#(X^D_0(N)/w_m)(\FF_{p^r})=  p^r +1 - \tr(T^{(m)}_{p^r}) + p\tr(T^{(m)}_{p^{r-2}}),$$ and $$\#(X^D_0(N)/w_m)(\FF_p) = p +1 - \tr(T^{(m)}_p).$$



Now note that we may compute that $\tr(T_{p^r}^{(m)})$ (on $(X^D_0(N)/w_m)$) is equal to $\frac{1}{2}( \tr(T_{p^r}) + \tr(T_{p^rm}))$ (on $X^D_0(N)$) via the identification of $H^0(X^D_0(N)/w_m,\Omega)$ with the $+1$ eigenspace of $H^0(X^D_0(N),\Omega)$ with respect to $w_m$. We may thus explicitly compute the traces on the quotient curve using Eichler's Trace Formula to obtain the following. 

\begin{theorem}\label{QuotientCurvePoints} If $r>1$ then \begin{equation}\#(X^D_0(N)/w_m)(\FF_{p^r})=  p^r +1 - \frac{\tr(T_{p^r}) + \tr(T_{p^rm})}{2} + \frac{p(\tr(T_{p^{r-2}}) + \tr(T_{p^{r-2}m}))}{2}\end{equation} and if $r=1$ then \begin{equation}\#(X^D_0(N)/w_m)(\FF_{p})=  p +1 - \frac{\tr(T_{p}) + \tr(T_{pm})}{2} \end{equation}\end{theorem}


If $\left(\dfrac{d}{p}\right) = 1$, then $C^D(N,d,m) \cong_{\QQ_p} X^D_0(N)$. If $\left(\dfrac{d}{p}\right) = -1$, then consider the following:

\begin{lemma}\label{GoodPointCount} \begin{equation} 2\#X^D_0(N)/w_m (\FF_{p^r}) = \#X^D_0(N)(\FF_{p^r}) + \# C^D(N,d,m)(\FF_{p^r})\end{equation}\end{lemma}

\begin{proof} Consider the following quotients by the action of $w_m$.

\[\xymatrix{ X^D_0(N)(\FF_{p^r}) \ar[rd] & & C^D(N,d,m)(\FF_{p^r})\ar[ld] \\
& X^D_0(N)/w_m(\FF_{p^r}) & }\]

Consider that $X^D_0(N)/w_m(\FF_{p^r})$ is made up of the set of equivalence classes $[P,Q]$ such that $P,Q\in X^D_0(N)(\overline \FF_{p^r})$, $w_m(P) = Q$ and for all $\sigma \in \Gal(\overline \FF_{p^r}/\FF_{p^r})$ either $\sigma P = Q$ and $\sigma Q = P$ or $\sigma P = P$ and $\sigma Q = Q$. In either case, $P,Q \in \FF_{p^{2r}}$ and we may fix $\sigma$ as the generator of $\Gal(\FF_{p^{2r}}/\FF_{p^r})$. The former case indicates that $w_m\sigma P = w_m Q = P$ and thus $P,Q \in C^D(N,d,m)(\FF_{p^r})$ while the latter case indicates that $P,Q \in X^D_0(N)(\FF_{p^r})$.

If $P \ne Q$ then $[P,Q]$ is a point over which the (geometric) map $X^D_0(N) \to X^D_0(N)/w_m$ is unramified, and so gives rise to two points in either $X^D_0(N)(\FF_{p^r})$ or $C^D(N,d,m)(\FF_{p^r})$. 
If $P = Q$ then $[P,Q] = [P,P]$ is a ramification point for the map of schemes $X^D_0(N) \to X^D_0(N)/w_m$. Note however that we have both $w_m \sigma P = P$ and $\sigma P = P$ so $P$ lies both on $X^D_0(N)(\FF_{p^r})$ and $C^D(N,d,m)(\FF_{p^r})$.\end{proof}

We may then combine Lemma \ref{GoodPointCount} with Theorem \ref{QuotientCurvePoints} to obtain the following explicit formula.

\begin{theorem}\label{NumberPointsGoodReductionTwist} Let $p$ be inert in $\QQ(\sqrt d)$ and let $m|DN$. If $r>1$ then \begin{equation} \# C^D(N,d,m)(\FF_{p^r}) =p^r + 1 - \tr(T_{p^rm}) + p \tr(T_{p^{r-2}m})\end{equation} and if $r=1$ then \begin{equation} \# C^D(N,d,m)(\FF_{p}) =p + 1 - \tr(T_{pm})\end{equation} \end{theorem}


 
In light of Theorem \ref{NumberPointsGoodReductionTwist}, we make the following definition.

\begin{definition}\label{newTF} For squarefree coprime integers $D$ and $N$, for $m|DN$ and for $p\nmid DN$, let $TF(D,N,m,p) := p + 1 - \tr_{H^0(X^D_0(N),\Omega)}(T_{pm})$.\end{definition}

\subsection{Inert primes and superspecial points}

We now use the theory of superspecial points to gain explicit criteria for the presence of rational points. 
Recall that the superspecial points of $X^D_0(N)(\overline\FF_p)$ are in bijection with $\Pic(Dp,N)$ via the embedding $c:X^D_0(N)_{\FF_p} \to X^D_0(Np)_{\FF_p}$ by Lemma \ref{ComponentsIntersectionPointsLemma}. Recall also that the action of $\Frob_p \in\Gal(\overline\FF_p/\FF_p)$ on the superspecial points in $X^D_0(N)(\overline\FF_p)$ is given by $w_p$ by Lemma \ref{ALGaloisLemma}. 

\begin{theorem}\label{SuperspecialGood}

If $p\nmid DN$ is inert in $\QQ(\sqrt d)$, then $C^D(N,d,m)(\QQ_p)$ is nonempty if either

\begin{itemize}
 \item  $mp\not\equiv 3 \bmod 4$ and $e_{Dp,N}(-4mp) \ne 0$, or

 \item  $mp \equiv 3 \bmod 4$ and one of $e_{Dp,N}(-4mp)$ or $e_{Dp,N}(-mp)$ is nonzero, or

 \item  $p =2$, $m=1$, and one of $e_{Dp,N}(-4)$ or $e_{Dp,N}(-8)$ is nonzero.
\end{itemize}
\end{theorem}

\begin{proof} We wish to determine if $\CZ(\FF_p)$ contains a superspecial point. This occurs if and only if there is a superspecial point $P\in X^D_0(N)(\overline\FF_p)$ such that $P = w_{mp}P$. By Corollary \ref{EmbeddingFixedPoint}, there is a superspecial $w_{mp}$-fixed point $P$ if and only if there is an embedding of $\ZZ[\sqrt{-mp}]$ into $\End_{\iota(\OO)}(A)$ where $(A,\iota)$ corresponds to $P$, or possibly $\ZZ[\zeta_4]$ if $mp =2$. 

If $mp =2$ then both $\ZZ[\zeta_4]$ and $\ZZ[\sqrt{-2}]$ are maximal orders, of discriminants $-4$ and $-8$ respectively. If $mp \equiv 1 \bmod 4$, then $\ZZ[\sqrt{-mp}]$ is maximal and of discriminant $-4mp$. If $mp \equiv 3\bmod 4$ then $\ZZ[\sqrt{-mp}]$ again has discriminant $-4mp$ but is no longer maximal. It is contained in $\ZZ[\frac{1+ \sqrt{-mp}}{2}]$, which is maximal and has discriminant $-mp$. Since there are no intermediate orders, this completes the proof.\end{proof}

\begin{corollary}\label{SuperspecialFullAL} If $p\nmid DN$ is inert in $\QQ(\sqrt d)$, $C^D(N,d,m)(\QQ_p)$ is nonempty when $m=DN$. Moreover, $\CZ(\FF_p)$ contains a point corresponding to a superspecial surface.\end{corollary}



\section{Local Points at Ramified Primes}

Throughout this section we will fix $D$ the discriminant of an indefinite quaternion $\QQ$-algebra, $N$ a squarefree integer coprime to $D$, a squarefree integer $d$, an integer $m\mid DN$ and a prime $p\nmid DN$ ramified in $\QQ(\sqrt d)$. Let $X^D_0(N)_{/\QQ}$ be given by Corollary \ref{DefnCoarseModuliScheme}. Let $w_m$ be as in Definition \ref{AtkinLehnerOO}. Let $C^D(N,d,m)_{/\QQ}$ the twist of $X^D_0(N)$ by $\QQ(\sqrt{d})$ and $w_m$. If $\Delta<0$, let $H_{\Delta}(X) \in \ZZ[X]$ \cite[p.285]{Cox} denote the Hilbert Class Polynomial of discriminant $\Delta$, and recall that this is simply the polynomial whose roots are the $j$-invariants of elliptic curves with complex multiplication by $R_\Delta$.

\begin{theorem}\label{RamifiedMainTheorem} Suppose that $p\nmid 2DN$ is a prime which is ramified in $\QQ(\sqrt d)$ and $m|DN$. Then $C^D(N,d,m)(\QQ_p)\ne \emptyset$ if and only if one of the following occurs.

\begin{enumerate}

\item $\left(\frac{-m}{p}\right) = 1$, $e_{D,N}(-4m) \ne 0$, and $H_{-4m}(X) =0$ has a root modulo $p$

\item $\left(\frac{-m}{p}\right) = 1$, $m \equiv 3\bmod 4$, $e_{D,N}(-m) \ne 0$, and $H_{-m}(X) = 0$ has a root modulo $p$

\item  $\left(\frac{-DN}{p}\right) = -1$,$m= DN$, $2\nmid D$, $\left(\frac{-p}{q}\right) = -1$ for all primes $q\mid D$, and $\left(\frac{-p}{q}\right) = 1$ for all primes $q\mid N$ such that $q\ne 2$

\item  $2\mid N$, $\left(\frac{-DN/2}{p}\right) = -1$, $m = DN/2$, $\left(\frac{-p}{q}\right) = -1$ for all primes $q\mid D$, and $\left(\frac{-p}{q}\right) = 1$ for all primes $q\mid N$ such that $q\ne 2$

\item  $2\mid D$, $\left(\frac{-DN}{p}\right) = -1$,$m= DN$, $p\equiv \pm 3 \bmod 8$, $\left(\frac{-p}{q}\right)= -1$ for all primes $q\mid (D/2)$, and $\left(\frac{-p}{q}\right) = 1$ for all primes $q\mid N$.

\item  $2\mid D$, $\left(\frac{-DN/2}{p}\right) = -1$,$m = DN/2$, $DN \equiv 2,6,$ or $10 \bmod 16$, $p\equiv \pm 3 \bmod 8$, $\left(\frac{-p}{q}\right) = -1$ for all primes $q\mid D$, and $\left(\frac{-p}{q}\right) = 1$ for all primes $q\mid N$.

\end{enumerate}

\end{theorem}

Compare this to the following theorem.

\begin{theorem}\label{OzmanRamified} Let $p$ be a prime, $(p,2N) =1$, $D=1$, and $m = N$. If $C^D(N,d,m)(\QQ_p)$ is nonempty then $H_{-4m}(X) =0$ has a root modulo $p$. If $H_{-4m}(X)$ has a root modulo $p$ and additionally $p\nmid \disc(H_{-4m}(X))$ then $C^D(N,d,m)(\QQ_p)$ is nonempty. \end{theorem} 

\begin{proof} Suppose that $p>2$, $D=1$ and $m=N$. By \cite[Proposition 4.6]{Oz09}, $C^D(N,d,m)(\QQ_p)$ is nonempty if and only if there is a prime $\nu$ of $\mathbb{B} = \QQ[X]/(H_{-4m}(X))$ such that $f(\nu|p) =1$. If fact, if $C^D(N,d,m)(\QQ_p)$ is nonempty, then Ozman shows how to produce an elliptic curve over $\QQ_p$ with good reduction and CM by $\ZZ[\sqrt{-m}]$. Therefore the $j$-invariant of its modulo $p$ reduction is a root modulo $p$ of $H_{-4m}(X) = 0$.

Conversely if $p\nmid\disc(H_{-4m}(X))$ then $p$ does not divide the conductor of $\mathbf{Z[X]}/{H_{-4m}(X)}$ and so if there is a linear factor modulo $p$ of $H_{-4m}(X)$ then there is a $\ZZ/p\ZZ$ factor of $\mathbf{Z}_{\mathbb{B}}/p\mathbf{Z}_{\mathbb{B}}$ \cite[I.8.3]{Neukirch}. Therefore there is a prime $\nu$ of $\mathbf{Z}_{\mathbb{B}}$ such that $\nu \mid p\mathbf{Z}_{\mathbb{B}}$ and $f(\nu|p) = 1$.\end{proof}

We note here that there are numerous counterexamples if $p\mid \disc(H_{-4m}(X))$, as pointed out to the author by Patrick Morton. Perhaps the easiest one is the case of $p = 13$, where $\left(\frac{-20}{13}\right) = \left(\frac{-13}{5}\right) = -1$ but $H_{-20}(X)$ factors as $(X + 8)^2$ modulo 13. Moreover, if $p\mid\disc(H_{-4m}(X))$ there is no guarantee of a root modulo $p$, as demonstrated by $m = 57$ and $p= 43$. In any case, we have unearthed a powerful tool for finding roots of Hilbert Class Polynomials modulo $p$, which may have useful applications in cryptography.


We may combine the results of Theorem \ref{RamifiedMainTheorem}(3) with those of Theorem \ref{OzmanRamified} to yield the following.

\begin{corollary}\label{SplittingHilbertClassPolynomial} Let $p\ne 2$ be a prime and let $N$ be a squarefree integer such that $\left(\frac{-N}{p}\right) = -1$. It follows that $H_{-4N}(X)$ has a root modulo $p$ if and only if for all odd primes $q\mid N$, $\left(\frac{-p}{q}\right) = 1$.\end{corollary}

To establish Theorem \ref{RamifiedMainTheorem} and Corollary \ref{SplittingHilbertClassPolynomial}, we determine a regular model over $\ZZ_p$ of $C^D(N,d,m)_{\QQ_p}$. We shall indeed show the following.

\begin{theorem} There is a regular model $\CX_{/\ZZ_p}$ of $C^D(N,d,m)_{\QQ_p}$ with the following properties. There is an equality of divisors on $\CX$, $$\CX_{\FF_p} = \sum_{i=0}^b d_i\Gamma_i,$$ such that each $\Gamma_i$ is defined over $\FF_p$ and is prime, each $d_i \le 2$, $d_0 =2$, $\Gamma_0 \cong (X^D_0(N)/w_m)_{\FF_p}$, and for all $i>0$, $p_a(\Gamma_i) = 0$. 

Suppose additionally that $p\ne 2$. Then for all $i>0$, $d_i = 1$ and $\Gamma_0$ intersects with $\Gamma_i$ in a unique point $Q_i$. These points $Q_i$ are such that $\sum_{i=1}^b Q_i$ is the branch divisor of $X^D_0(N)_{\FF_p} \to (X^D_0(N)/w_m)_{\FF_p}$.

\end{theorem}

In fact, we shall show that if $p\ne 2$, $\CX$ is the blowup of a scheme $\CZ_{/\ZZ_p}$ such that there is an equality of divisors $\CZ_{\FF_p} = 2\Gamma$ where $\Gamma \cong (X^D_0(N)/w_m)_{\FF_p}$. Therefore there are smooth points of $\CX(\FF_p)$ if and only if $\FF_p = \FF_p(P_i) = \FF_p(\Gamma_i)$ since $\Gamma_i \cong \PP^1_{\FF_p(Q_i)}$. After constructing $\CZ$ and $\CX$, we will describe $\FF_p(Q_i)$, i.e., the $\FF_p$-rationality of $w_m$-fixed points.

\subsection{The first steps towards forming a model}

Let us begin with a few foundational facts.



\begin{lemma}\label{Halle} Let $X_{/K}$ be a curve with potentially semistable reduction realized by a cyclic totally ramified extension $L/K$ of local fields. Let $k$ be their common residue field and let $S/R$ be the corresponding extension of discrete valuation rings. Let $\CY \to \Spec(S)$ be a regular model of $X_L$, $\Gal(L/K) = \langle \sigma \rangle$ and assume that there exists some $\alpha$ an automorphism of $\CY$ above $\sigma:\Spec(S) \to \Spec(S)$ extending the Galois action on $X_L$.

\begin{enumerate}

\item The quotient $\CZ = \CY/\langle \alpha \rangle$ is a scheme of relative dimension one over $\Spec(R)$ with generic fiber $X$,

\item Let $\xi_1, \dots,\xi_n$ be the generic points of the irreducible components $C_1, \dots, C_n$ of $\CY_k$ lying above a component $C$ of $\CZ_k$ with generic point $\xi$. Let $D_i = D(\xi_i|\xi)$, $I_i = I(\xi_i| \xi)$ denote the decomposition and inertia groups, respectively. Then the multiplicity of $\xi$ in $\CZ_k$ is $|D_i| n/|I_i| $.

\end{enumerate}

\end{lemma}

\begin{proof} That $\CZ$ is a $\Spec(R)$-scheme follows from the universal properties of the quotient as outlined in \cite[3.6]{Vi77}. To obtain the multiplicities, we recall \cite[VIII.3.9]{Liu} that the multiplicity of $\xi_i$ is $v_i(s)$ where $v_i$ is the discrete valuation of $\OO_{\CY,\xi_i}$ and $s$ is a uniformizer of $S$. As $\CY$ has semistable reduction, $v_i(s) =1$ for all $i$. Likewise the multiplicity of $\xi$ is $v(r)$ where $v$ is the discrete valuation of $\OO_{\CZ,\xi}$ and $r$ is a uniformizer of $R$. As $\CY \to \CZ$ is Galois, there are positive integers $e,q$ such that $v_i\mid _R = e v$ and $q = |D_i/I_i| $ for all $i$ and $[L:K] = eqn$. As $L/K$ is totally ramified, $rS = s^{eqn}S$. It then follows that $$e v(r) = v_i(r) = v_i(s^{eqn}) = eqn v_i(s)$$ and thus $ v(r) = qn v_i(s) = qn = |D_i/I_i| n = |D_i| n/|I_i|.$ \end{proof}

\begin{lemma}{\cite[5.2]{LoWMC}}\label{Sing} Under the hypotheses of Lemma \ref{Halle}, the non-regular points of $\CZ$ are precisely the branch points $Q_1, \dots , Q_b$ of $\CY_k \to \CZ_k$. 
\end{lemma}


If $K = \QQ_p$ and $L = \QQ_p(\sqrt d)$ then $R = \ZZ_p$, $S = \ZZ_p[\sqrt d]$, $k = \FF_p$, and $\sigma(\sqrt d) = - \sqrt d$. If additionally $X = X^D_0(N)_{\QQ_p}$, then $\CY_{\FF_p}$ is smooth and we may realize $\CY \cong X^D_0(N)_{/\ZZ_p[\sqrt d]}$ from Corollary \ref{DefnCoarseModuliScheme}. If we take $\alpha = w_m \circ \sigma$ and take $\CZ = \CY/\langle \alpha \rangle$, then the following holds.

\begin{theorem}\label{CZDefnTheorem} The scheme $\CZ_{/\ZZ_p} = \CY/\langle \alpha \rangle$ has generic fiber $C^D(N,d,m)_{\QQ_p}$, and there is an equality of divisors $\CZ_{\FF_p} = 2 \Gamma$ where $\Gamma \cong (X^D_0(N)/w_m)_{\FF_p}$. 
\end{theorem}

\begin{proof}Since there is a unique component of $\CY_{\FF_p}$, $n=1$. Let $\xi', \xi$ be the generic points of the components of $\CY_{\FF_p}$ and $\CZ_{\FF_p}$ respectively. Then $D(\xi'|\xi) = \langle \alpha \rangle$ since $\alpha$ preserves $\CY_{\FF_p}$. Since $w_m$ acts non-trivially on $\CY_{\FF_p}$, $I(\xi'|\xi) = \{ {\rm id}\}$. The multiplicity of the component corresponding to $\xi$ is thus 2.

To determine the $\Gamma$ such that $2\Gamma = \CZ_{\FF_p}$, recall that the pushforward under $f:\CY \to \CZ$ of $\CY_{\FF_p}$ forms a prime divisor of $\CZ$ in $\CZ_{\FF_p}$ and must therefore be $\Gamma$. To determine this pushforward, note that the induced action of $\sigma$ on $\Spec(\FF_p)$ is trivial and consider the following commutative square.

\[
\xymatrix{ 
\CY \ar@{->}[r]^\alpha \ar@{->}[d] & \CY\ar@{->}[d] \\ 
\Spec(\ZZ_p[\sqrt d]) \ar@{->}[r]^\sigma & \Spec(\ZZ_p[\sqrt d])
}
\]

The fiber product of this square with $\Spec(\FF_p) \to \Spec(\ZZ_p[\sqrt d])$ is simply the $\Spec(\FF_p)$-involution $w_m$ on $\CY_{\FF_p} = X^D_0(N)_{\FF_p}$. It follows that $f$, when restricted to $\CY_{\FF_p}$ becomes simply the quotient map $X^D_0(N)_{\FF_p} \to (X^D_0(N)/w_m)_{\FF_p}$, and therefore $\Gamma \cong (X^D_0(N)/w_m)_{\FF_p}$. 
%

\end{proof}

We note by Lemma \ref{Sing} that $\CZ$ is not generally a regular scheme. To make the resolution of its singularities easier, we fix the following.

\begin{definition}\label{BranchPointsDefinition} Fix an ordering $\{Q_i\}$ of the branch points of the quotient map $f:X^D_0(N)_{\FF_p} \to (X^D_0(N)/w_m)_{\FF_p}$. Let $P_i$ denote the unique preimage of $Q_i$ under $f$.\end{definition}

Note that by definition, the $P_i$ are exactly the points of $X^D_0(N)_{\FF_p}$ fixed by $w_m$. When $p\ne 2$, we will explicitly describe a desingularization of $\CZ$ in the strong sense \cite[Definition 8.3.39]{Liu}. This will be a regular model of $C^D(N,d,m)_{\QQ_p}$. We will first describe the branch points $\{Q_i\}$ and their $\FF_p$-rationality.


\subsection{Atkin-Lehner fixed points over finite fields}

Throughout this section, we will keep the notation of Definition \ref{BranchPointsDefinition}. Note that since $\QQ_p[\sqrt d]$ is totally ramified over $\QQ_p$, $\FF_p(Q_i) \cong \FF_p(P_i)$. 

\begin{lemma}\label{RamifiedKeyLemma} Let $\CZ$ be non-regular and $\pi: \CX \to \CZ$ a desingularization in the strong sense and assume that for all $i$, $\pi\inv(Q_i)$ is a chain of rational curves such that at least one has multiplicity one. Then $C^D(N,d,m)(\QQ_p)$ is nonempty if and only if either 

\begin{enumerate}

\item $\left(\frac{-m}{p}\right) = 1$ and one of the following holds:

\begin{itemize}

\item $m=2$ or

\item $H_{-4m}(X)$ has a root modulo $p$ or

\item $m\equiv 3\bmod 4$ and $H_{-m}(X)$ has a root modulo $p$,

\end{itemize} 

\item or $\left(\frac{-m}{p}\right) = -1$ and one of the conditions of Corollary \ref{SuperspecialFpRationalALFixed} are satisfied.\end{enumerate}\end{lemma}

\begin{proof} Note first that each component in $\pi\inv(Q_i)$ must be isomorphic to $\PP^1_{\FF_p(Q_i)}$. Therefore by our assumption on $\pi$, $\FF_p = \FF_p(Q_i)$ if and only if there is a reduced copy of $\PP^1_{\FF_p}$ in $\pi\inv(Q_i)$. 

By Lemma \ref{ordinaryorsupersingular}, any QM abelian surface over a finite field must be either ordinary or supersingular. Suppose first that $(A,\iota)$ is supersingular and fixed by $w_m$. By Lemma \ref{pnmidDsupersingular}, if $(A,\iota)$ is a supersingular QM-abelian surface over a finite field of characteristic $p$, then $(A,\iota)$ is superspecial. Therefore, one of the conditions of Corollary \ref{SuperspecialFpRationalALFixed} hold if and only if there is a QM abelian surface $(A,\iota)$ fixed by $w_m$ whose corresponding point $P_i$ is $\FF_p$-rational. 

Now suppose that $(A,\iota)$ is an ordinary QM-abelian surface over a finite field $k$ fixed by $w_m$. By Theorem \ref{OrdinaryCM}, there are elliptic curves $E$ and $E'$ such that $\End_k(E) \cong \End_k(E') \cong R' = \ZZ[\sqrt{-m}]$ or $\ZZ[\frac{1 + \sqrt{-m}}{2}]$ (or $\ZZ[\zeta_4]$ if $m=2$) and $A \cong E\times E'$. Now note that the $j$-invariants of $E$ and $E'$ are roots of $H_{-4m}(X) \bmod p$, $H_{-m}(X)\bmod p$ if $m\equiv 3\bmod 4$, or $H_{-4}(X)$ if $m=2$. If $m=2$, then $H_{-4}(X)$ and $H_{-8}(X)$ have degree one so for all $p$, $H_\Delta(X)$ has a root modulo $p$. Since the $j$-invariants of $E$ and $E'$ are defined over $\FF_p$, $(A,\iota)$ is defined over $\FF_p$. Therefore if $P_i$ corresponds to the surface $(A,\iota)$ then $\FF_p(P_i) = \FF_p$. 

Recall now the classical theorem of Deuring that if $K$ is a number field, $\p\mid p$ is a prime, and $E_{/K}$ is an elliptic curve with CM by $R_\Delta$ then $E\bmod \p$ is ordinary if and only if $\left(\frac{\Delta}{p}\right) = 1$ \cite[Theorem 13.12]{EllipticFunctions}. Therefore $(A,\iota)$ is ordinary if and only if $\left(\frac{-m}{p}\right) = 1$.


We have thus shown that either (1) or (2) holds if and only if there is a reduced copy of $\PP^1_{\FF_p}$ in some $\pi\inv(Q_i)$. Since the strict transform of $\Gamma$ in $\CX$ has multiplicity two, the presence of a reduced copy of $\PP^1_{\FF_p}$ in some $\pi\inv(Q_i)$ is equivalent to the presence of a smooth point of $\CX(\FF_p)$. By Hensel's Lemma \cite[Lemma 1.1]{JoLi}, the presence of a smooth point in $\CX(\FF_p)$ is equivalent to $C^D(N,d,m)(\QQ_p)$ being nonempty. \end{proof}


\subsection{Tame Potential Good Reduction}

In this section we construct a regular model of $C^D(N,d,m)_{\QQ_p}$. Let $\CX_{\ZZ_p} := \mathrm{Bl}_{\{Q_i\}}(\CZ)$, the blowup of $\CZ$ along the branch divisor of $\CY_{\FF_p} \to \CZ_{\FF_p}$ \cite[Definition 8.1.1]{Liu}. Since the blowup construction gives a map $\CX \to \CZ$ which is an isomorphism away from $\{Q_i\}$, $\CX$ is a regular model if and only if $\CX \to \CZ$ is a desingularization in the strong sense if and only if $\CX$ is a regular scheme. 

To see that this is a regular scheme, let $\overline R = \ZZ_p^{nr}$, a strict henselization of $\ZZ_p$. We will construct in this section an auxiliary scheme $\CX'_{/\overline R}$. If we can show that $\CX_{\overline R} \cong \CX'$, it will follow that $\CX$ is regular \cite[Lemma 2.1.1]{CES}. Thus, the hypotheses of Lemma \ref{RamifiedKeyLemma} would be satisfied and thus Theorem \ref{RamifiedMainTheorem} would be proved.

%
%
%
%
%
%

Also fix $\overline S =\overline R[\sqrt d]$, $k'$ the residue field of $\overline S$, $k$ the residue field of $\overline R$, and note that both $k$ and $k'$ must be isomorphic to $\overline \FF_p$. We note the following.

\begin{lemma}\label{RamifiedALFixedPointTame} Suppose that $p\ne 2$ and let $Q$ be a point of $Q_i \times_{\ZZ_p} \overline R$. Then $Q$ is a tame cyclic quotient singularity \cite[Definition 2.3.6]{CES} with $n=2$ and $r =1$.\end{lemma}

\begin{proof}Let $\overline \alpha$ denote the extension of $\alpha$ from $\CY$ to $\CY_{\overline S}$. We wish to show that $\widehat{\OO_{\CZ,Q}^{sh}}$ is the ring of invariants of a $\mu_2$ (or since $p\ne 2$, $\ZZ/2\ZZ$) action. Fix an isomorphism $\overline S[[X]] \cong \widehat{\OO_{\CY_{\overline S},P}}$ where $P$ is the unique preimage of $Q$ under $\overline f: \CY_{\overline S} \to \CZ_{\overline R}$. Since $w_m$ is always Galois-equivariant, $\overline\alpha(\sqrt d) = -\sqrt d$. Since $\overline \alpha$ induces an isomorphism $\overline S[[T]] \cong \overline S[[ \overline\alpha(T)]]$, $\overline\alpha(T) = P_\alpha(T) = \sum_{j \ge 1} \alpha_j T^j$. Since $\overline\alpha$ is an involution, $\alpha_1 = -1$. Then $\overline\alpha(T) - T = -2T(1 + O(T))$, i.e. $\overline\alpha(T) - T \equiv -2T \bmod (T^2)$. Since $-2\not\in \mathfrak{m}_{\overline S}$, $\overline S[[T]] \cong \overline S[[T']]$ where $T' := \overline\alpha(T) -T$. Note also that $\overline\alpha(T') = \overline\alpha(\overline\alpha(T) - T) = T - \overline\alpha(T) = -(T')$. Therefore $\sqrt d$ and $T'$ form a basis of uniformizers for the two-dimensional local ring $\widehat{\OO_{\CY_{\overline S},P}}$ and $\overline\alpha$ acts as $-1$ on both $\sqrt d$ and $T'$. 
 
Note now that $\widehat{\OO_{\CZ_{\overline R},Q}}$ is the ring of invariants of the $\mu_2$-action given by $\overline\alpha$ on $\overline S[[T']]$. Recall that since $p\ne 2$ is a uniformizer for $R$ and $p$ is ramified in $\QQ(\sqrt d)$ where $d$ is square-free, $d$ is also a uniformizer. Therefore $\overline S[[T']] \cong \overline R[[t_1,t_2]]/(t_1^{m_1}t_2^{m_2} - d)$ where $m_1 = 2$, $t_2 = T'$, and $m_2 = 0$ in the notation of \cite{CES}. It follows that $Q$ is a tame cyclic quotient singularity with $n=2$ and $r = 1$.
\end{proof}
 
From here on, let $b'$ be such that $\sum_{i=1}^b Q_i\times_{\ZZ_p} \overline R = \sum_{i=1}^{b'} Q_i'$.

\begin{definition}\label{HirzebruchJungDefinition} Let $R$ be a discrete valuation ring with algebraically closed residue field, $X_{/R}$ be a scheme, and $P$ a tame cyclic quotient singularity of $X$ of type $n,r$. Then \cite[Theorem 2.4.1]{CES} we can inductively produce a chain of divisors $E_1, \dots E_\lambda$ and a set of integers $b_1, \dots, b_\lambda$ such that 
\begin{center}\begin{itemize}

\item There is a resolution $\tilde X_P \to X$ of the singularity at $P$ whose fiber over $P$ is the chain made up of the $E_i$'s

\item $E_i\cdot E_j = \delta_{i,j\pm 1}$ if $i \ne j$, $E_j^2 = -b_j<-1$,

\item $\frac{n}{r} = b_1 - \frac{1}{b_2 - \frac{1}{\dots - \frac{1}{b_\lambda}}}$.

\end{itemize}\end{center}

This $\tilde X_P$ is called the Hirzebruch-Jung desingularization at $P$. \end{definition}

\begin{theorem}\label{ramifiedtheorem} If $p\ne 2$ there is a desingularization of $\overline R$-schemes $\CX' \to \CZ_{\overline R}$ such that $\CX'_k$ has the form

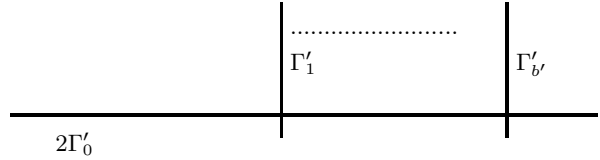
\setlength{\unitlength}{0.6cm}\linethickness{.3mm}
\small
\begin{picture}(18,5)
\put(5,2){\line(1,0){6}}
\put(11,1.5){\line(0,1){3}}\put(11.2,3){$\Gamma_1'$}
\put(11.2,3.8){$.........................$}
\put(16,1.5){\line(0,1){3}}\put(16.2,3){$\Gamma_{b'}'$}
\put(11,2){\line(1,0){7}}\put(6,1.2){$2\Gamma_0'$}
\end{picture}

where $\Gamma_0'$ is the strict transform of $\Gamma_{\overline R}$ and for all $i>0$, $\Gamma_i' \cong \PP^1_k$. This is to say that there is an equality of divisors on $\CX'$ between $\CX'_k$ and $2\Gamma_0' + \sum_{i=1}^{b'} \Gamma_i'$, $\Gamma_0'\cap \Gamma_i' = Q_i' \in Q_i \times_{\ZZ_p} \overline R$, and all intersections are transverse. Moreover $\CX_{\overline R} \cong \CX'$, and since $\CX'$ is a regular scheme, so is $\CX$. It follows that $\CX$ is a regular $\ZZ_p$ model for $C^D(N,d,m)_{\QQ_p}$.

\end{theorem}

\begin{proof} We construct $\CX'$ by performing the Hirzebruch-Jung desingularization at $Q$ for all $Q$ in all $Q_i \times \overline R$. By Lemma \ref{RamifiedALFixedPointTame}, $n=2$, $r=1$ and thus $\lambda =1$ and $b_1 = \frac{2}{1}$ in Definition \ref{HirzebruchJungDefinition}. Therefore $\CX'_k$ has the form above \cite[Theorem 2.4.1]{CES}.

Recall now that $\CX' \to \CZ_{\overline R}$, $\CX_{\overline R} \to \CZ_{\overline R}$ are birational morphisms and so there is a birational map $f:\CX_{\overline R} \dashedrightarrow \CX'$ making the following diagram commute.

\[
\xymatrix{
\CX_{\overline R} \ar@{-->}[r]^f \ar@{->}[rd] & \CX' \ar@{-->}[r]^{f\inv} \ar@{->}[d] & \CX_{\overline R}\ar@{->}[ld] \\
 &  \CZ_{\overline R}  &
}
\]
 
Since $\overline R$ is Dedekind, $f\inv\mid _{\Gamma_0'}$ is the identity and $f$ can be extended so that the preimage of each divisor on either $\CX_{\overline R}$ or $\CX'$ is again a divisor. We thus find that $f$ is a morphism and thus an isomorphism \cite[Theorem 8.3.20]{Liu}. It follows that $\CX_{\overline R}$ is regular and therefore $\CX$ is regular \cite[Lemma 2.1.1]{CES}.\end{proof}
 
\begin{corollary} Theorem \ref{RamifiedMainTheorem} holds.\end{corollary}

\begin{proof} By Theorem \ref{ramifiedtheorem}, the conditions of Lemma \ref{RamifiedKeyLemma} hold.\end{proof}


\begin{remark} In the case that $X^D_0(N)/w_m \cong \PP^1_{\FF_p}$ we may deduce this theorem from work of Sadek \cite{Sadek}.\end{remark}


\begin{remark} Retaining the notation of Lemma \ref{Sing}, if $p=2$ we still have that $\CZ_{/\ZZ_2}$ is a normal scheme, non-regular precisely at the fixed points on the special fiber of $w_m$. Moreover, these singularities are still $\ZZ/2\ZZ$-quotient singularities. Once more, we may resolve these singularities to give a regular model of $C^D(N,d,m)$. Unfortunately Lemma \ref{RamifiedALFixedPointTame} no longer holds as these singularities are \emph{wild}, and it is not known under what circumstances a resolution will have non-reduced components.\end{remark} 



\section{Local points when $p|D$}

Throughout this section we will fix $D$ the discriminant of an indefinite quaternion $\QQ$-algebra, $N$ a squarefree integer coprime to $D$, a squarefree integer $d$, an integer $m\mid DN$ and a prime $p\mid D$ unramified in $\QQ(\sqrt d)$. Let $w_m$ be as in Definition \ref{AtkinLehnerOO}. Let $X^D_0(N)_{/\QQ}$ be as defined in Corollary \ref{DefnCoarseModuliScheme}, and let $C^D(N,d,m)_{/\QQ}$ be its twist by $\QQ(\sqrt d)$ and $w_m$. 
The purpose of this section is to prove the following theorem. 

\begin{theorem}\label{mainCDtheorem} Suppose that $p\mid D$ is unramified in $\QQ(\sqrt d)$ and $m\mid DN$. Let $p_i$, $q_j$ be primes such that $D/p = \prod_i p_i$ and $N = \prod_j q_j$. 

\begin{itemize}

\item Suppose $p$ is split in $\QQ(\sqrt d)$. Then $C^D(N,d,m)(\QQ_p)$ is nonempty if and only if one of the following two cases occurs [Theorem \ref{CDpnmidmTheorem}].

\begin{enumerate}

\item $p=2$, $p_i\equiv 3 \bmod 4$ for all $i$, and $q_j \equiv 1 \bmod 4$ for all $j$

\item $p\equiv 1 \bmod 4$, $D = 2p$, and $N=1$

\end{enumerate}

\item Suppose that $p$ is inert in $\QQ(\sqrt d)$. 

\begin{itemize}

\item If $p\mid m$, $C^D(N,d,m)(\QQ_p)$ is nonempty if and only if one of the following four cases occurs.

\begin{enumerate}


\item $m = p$, $p_i \not\equiv 1\bmod 3$ for all $i$, and $q_j \not\equiv 2\bmod 3$ for all $j$ [Lemma \ref{CDpeqm}]

\item $m = 2p$ and one of $e_{D/p,N}(-4)$ or $e_{D/p,N}(-8)$ is nonzero [Lemma \ref{pdividesmpnem}]

\item $m/p \not\equiv 3\bmod 4$ and $e_{D/p,N}(-4m/p)$ is nonzero [Lemma \ref{pdividesmpnem}]

\item $m/p \equiv 3\bmod 4$ and one of $e_{D/p,N}(-4m/p)$ or $e_{D/p,N}(-m/p)$ is nonzero [Lemma \ref{pdividesmpnem}]

\end{enumerate}

\item If $p\nmid m$, $C^D(N,d,m)(\QQ_p)$ is nonempty if and only if one of the following four cases occurs [Theorem \ref{CDpnmidmTheorem}].

\begin{enumerate}

\item $p=2$, $m=1$, $p_i \equiv 3 \bmod 4$ for all $i$, and $q_j \equiv 1\bmod 4$ for all $j$

\item $p\equiv 1 \bmod 4$, $m = DN/(2p)$, for all $i$, $p_i \not\equiv 1\bmod 4$, and for all $j$, $q_j \not\equiv 3\bmod 4$

\item $p =2$, $m = DN/2$, $p_i \equiv 3\bmod 4$ for all $i$, and $q_j \equiv 1\bmod 4$ for all $i$

\item $p\equiv 1\bmod 4$, $m = DN/p$, for all $i$, $p_i \not\equiv 1\bmod 4$, and for all $j$, $q_j \not\equiv 3\bmod 4$

\end{enumerate}

\end{itemize}

\end{itemize}

\end{theorem}

As opposed to the case where $p\mid N$, all conditions here are determined by congruences. For completeness, we record the following.

\begin{corollary}\label{CDfullALCor} Let $p_i$, $q_j$ be primes such that $D/p = \prod_i p_i$ and $N = \prod_j q_j$.

\begin{itemize}

\item If $p$ is split in $\QQ(\sqrt d)$, then $C^D(N,d,DN) \cong X^D_0(N)$ over $\QQ_p$ and $X^D_0(N)(\QQ_p)$ is nonempty if and only if one of the following two cases occurs.

\begin{enumerate}

\item $p=2$, $p_i\equiv 3 \bmod 4$ for all $i$, and $q_j \equiv 1 \bmod 4$ for all $j$

\item $p\equiv 1 \bmod 4$, $D = 2p$, and $N=1$

\end{enumerate}

\item If $p$ is inert in $\QQ(\sqrt d)$ then $C^D(N,d,DN)(\QQ_p)$ is nonempty.

\end{itemize}
\end{corollary}

\begin{proof} Note that $e_{D/p,N}(-4DN/p)$ is always nonzero by Theorem \ref{EichlerEmbedding}.\end{proof}

To prove Theorem \ref{mainCDtheorem}, we shall need to work with regular models for $X^D_0(N)_{\QQ_p}$ and $C^D(N,d,m)_{\QQ_p}$.

\begin{definition}\label{CDdesingularization} Let $\pi: \CX \to X^D_0(N)_{/\ZZ_p}$ denote a minimal desingularization. If $x$ is a superspecial point on $X^D_0(N)_{\overline\FF_p}$ let $\ell = \ell(x)$ be as in Definition \ref{PicardLength}. If $\ell>1$, $\pi^*(x(\Spec(\overline\FF_p))) = \bigcup_{i=1}^{\ell-1} C_i$ where for all $i$, $C_i\cong \PP^1_{\overline\FF_p}$ and exactly two points of $C_i$ are singular in $\CX_{\overline\FF_p}$.\end{definition}


For $n\mid DN$, let $w_n$ denote the automorphism of Definition \ref{AtkinLehnerOO}. Note that extending the automorphism $w_n$ from Definition \ref{AtkinLehnerOO} to $\CX$ makes sense because $w_n: X^D_0(N) \to X^D_0(N)$ induces a birational morphism $\CX \dashedrightarrow \CX$ permuting the components of $\CX_{\FF_p}$. Therefore $w_n$ on $X^D_0(N)$ induces an isomorphism $\CX \to \CX$ \cite[Remark 8.3.25]{Liu}.

We note also that the components of $X^D_0(N)_{\overline\FF_p}$ are in $W$-equivariant bijection with $\Pic(D/p,N)\coprod \Pic(D/p,N)$ by Theorem \ref{CerednikDrinfeldModelTheorem}. The intersection points, which can only link a component in one copy of $\Pic(D/p,N)$ to a component in the other copy of $\Pic(D/p,N)$ are in $W$-equivariant bijection with $\Pic(D/p,Np)$ as in Theorem \ref{CerednikDrinfeldModelTheorem}. The bijection of the components with two copies of $\Pic(D/p,N)$ is $W/\langle w_p\rangle$-equivariant. As in Lemma \ref{ComponentsIntersectionPointsLemma}, $w_p$ interchanges the two copies of $\Pic(D/p,N)$. We define the length of a component of $X^D_0(N)_{\overline\FF_p}$ by the length of the associated element of $\Pic(D/p,N)$ as in Definition \ref{PicardLength}.


\begin{definition} Let $\sigma$ be such that $\langle \sigma\rangle = \Aut_{\ZZ_p}(\ZZ_{p^2})$. We denote by $\CZ_{/\ZZ_p}$ the regular model of $C^D(N,d,m)_{\QQ_p}$ obtained as the \'etale quotient $\CZ$ of $\CX_{\ZZ_{p^2}}$ by the action of $w_m \circ \sigma$.\end{definition}

Note that if $p$ is inert in $\QQ(\sqrt d)$ then $\ZZ_p[\sqrt d] \cong \ZZ_{p^2}$ and thus the generic fiber of $\CZ$ is $C^D(N,d,m)_{\QQ_p}$. Therefore $\CZ$ is a regular model of $C^D(N,d,m)_{\QQ_p}$ if $p$ is inert in $\QQ(\sqrt d)$. 

We also note that if $p$ is split in $\QQ(\sqrt d)$, or if $p$ is inert and $m=1$, then $C^D(N,d,m)_{\QQ_p} \cong X^D_0(N)_{\QQ_p}$. If $p$ is split in $\QQ(\sqrt d)$, we can consider $d'$ to be any square-free integer such that $p$ is inert in $\QQ(\sqrt{d'})$ and $\CZ'$ to be the regular model of $C^D(N,d',1)_{\QQ_p} \cong X^D_0(N)_{\QQ_p}$. Therefore, we shall obtain our results when $p$ is split as a corollary to our results when $p\nmid m$. 


\subsection{The proof when $p\mid m$}

\begin{lemma}\label{CDpdividesmcases}Let $p\mid D$ be unramified in $\QQ(\sqrt d)$ and $p\mid m$. Then $C^D(N,d,m)(\QQ_p)$ is nonempty if and only if one of the following occurs.

\begin{enumerate}[(1)]

\item $p = m$ and there is some component of $X^D_0(N)_{\overline\FF_p}$ with length greater than one

\item $p\ne m$ and there is a component of $X^D_0(N)_{\overline\FF_p}$ fixed by $w_{m/p}$

\end{enumerate} 
\end{lemma}


\begin{proof}If $p=m$ this is the obvious extension of a result of Rotger-Skorobogatov-Yafaev \cite[Proposition 3.4]{RSY}. 




Now suppose that $p\mid m$ but $p\ne m$ and recall the curve $M_{/\ZZ_p}$ of Theorem \ref{CerednikDrinfeldModelTheorem}. Let $\pi':N \to M$ be a minimal desingularization, so that $N_{\FF_p}$ is the twist of $\CZ_{\FF_p}$ by $\FF_{p^2}$ and $w_{m/p}$. Since $m\ne p$, $w_{m/p}$ is not the identity. 
Recall that a non-identity involution of $\PP^1$ fixes exactly 2 points of $\PP^1(\overline\FF_p)$. Suppose that a component of $N_{\overline\FF_p}$ is fixed by $w_{m/p}$ (under the isomorphism $N_{\overline\FF_p}\cong \CZ_{\overline\FF_p} \cong \CX_{\overline\FF_p}$). Therefore there is a component $y \cong \PP^1_{\FF_p}$ of $\CZ_{\FF_p}$. Since all intersection points are rational and at most 2 singular intersection points stayed $\FF_p$-rational, $y$ contains the image of a smooth $\FF_p$ rational point. Since there is a smooth point of $\CZ(\FF_p)$, $C^D(N,d,m)(\QQ_p)$ is nonempty by Hensel's Lemma.

Finally we note that if a component $C$ of $\CX_{\overline\FF_p}$ is fixed by $w_{m/p}$ then so is its image $\pi(C)$. If $\pi(C)$ is a component of $X^D_0(N)_{\overline\FF_p}$, we are done. If $\pi(C)$ is not a component then it is an intersection point of two components $C_1,C_2$ of $X^D_0(N)_{\overline\FF_p}$. It follows that $w_{m/p}$ either fixes both of them or interchanges them. However, Theorem \ref{CerednikDrinfeldModelTheorem} tells us that under the bijection between components of $X^D_0(N)_{\overline\FF_p}$ and $\Pic(D/p,N) \coprod \Pic(D/p,N)$, $C_1$ must lie in one copy and $C_2$ in the other. Since these bijections are $W/\langle w_p\rangle$-equivariant, $w_{m/p}$ cannot interchange $C_1$ and $C_2$ and must therefore fix them.\end{proof}


\begin{example} Let $\CX = X^{26}_0(1)_{/\ZZ_2}$, which is regular over $\ZZ_2$. Depicted below is the dual graph of $\CX_{\overline\FF_2}$. This tells us that $\CX_{\overline\FF_2}$ is simply two copies of $\PP^1_{\overline\FF_2}$ glued along the $\FF_2$-rational points of each.
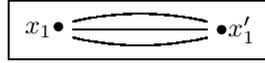
\[\fbox{\xygraph{!~-{@{-}@[|(2)]}
*+{x_1 \bullet}="x1" & & *+{\bullet x_1'}="x1p"
"x1"-"x1p"
"x1"-@/^/"x1p"
"x1"-@/_/"x1p"
}}\]

Since the action of $w_2\Frob_2$ fixes each component and intersection point, the only fixed points are non-smooth, and thus $C^{26}(1,d,2)(\QQ_2)$ is empty for all $d\equiv \pm 3 \bmod 8$. On the other hand, since the action of $w_{26}\Frob_2$ cannot interchange $x_1$ and $x_1'$, it must act non-trivially on each component, and thus there must be a smooth fixed point of $w_{26}\Frob_2$. It follows that $C^{26}(1,d,26)(\QQ_2)$ is nonempty for all $d\equiv \pm 3\bmod 8$.
\end{example}

\begin{lemma}\label{CDpeqm} If $p=m$ and $p$ is inert in $\QQ(\sqrt{d})$, then $C^D(N,d,m)(\QQ_p)\ne \emptyset$ if and only if one of the following occurs.

\begin{enumerate}[(1)]

\item For all primes $q\mid (D/p)$, $q \not\equiv 1\bmod 4$, and for all primes $q\mid N$, $q\not\equiv 3\bmod 4$.

\item For all primes $q\mid (D/p)$, $q\not\equiv 1\bmod 3$, and for all primes $q\mid N$, $q\not\equiv 2\bmod 3$.

\end{enumerate}

\end{lemma}

\begin{proof} By Theorem \ref{EichlerEmbedding}, condition (1) is equivalent to $e_{D/p,N}(-4)\ne 0$ and condition (2) is equivalent to $e_{D/p,N}(-3) \ne 0$. 
We know that $e_{D/p,N}(-4)\ne 0$ if and only if there is a component of $X^D_0(N)_{\overline\FF_p}$ of length divisible by two and $e_{D/p,N}(-3)\ne 0$ if and only if there is a component of $X^D_0(N)_{\overline\FF_p}$ of length divisible by three. This is to say that one of the two conditions of the Lemma occur if and only if there is a component $y$ of $X^D_0(N)_{\overline\FF_p}$ such that $\ell(y)>1$. But then by Lemma \ref{CDpdividesmcases} there is such a component if and only if $C^D(N,d,m)(\QQ_p)$ is nonempty.\end{proof}

\begin{lemma}\label{pdividesmpnem} If $p\mid m$ and $p\ne m$, then $C^D(N,d,m)(\QQ_p)$ is nonempty if and only if one of the following occurs.

\begin{itemize}

\item $m = 2p$ and one of $e_{D/p,N}(-4),e_{D/p,N}(-8)$ is nonzero.

\item $m/p\not\equiv 3 \bmod 4$ and $e_{D/p,N}(-4m/p)$ is nonzero.

\item $m/p \equiv 3\bmod 4$ and one of $e_{D/p,N}(-4m/p)$ or $e_{D/p,N}(-m/p)$ is nonzero.

\end{itemize}
\end{lemma}

\begin{proof} Suppose that $p\mid m$ and $p\ne m$. After Lemma \ref{CDpdividesmcases}, $C^D(N,d,m)(\QQ_p)$ is nonempty if and only if a component of $X^D_0(N)_{\overline\FF_p}$ is fixed by $w_{m/p}$. After Lemma \ref{ComponentsIntersectionPointsLemma}, such a component corresponds to an element of $\Pic(D/p,N)$. After Lemma \ref{EmbeddingFixedPoint}, such a component is fixed by $w_{m/p}$ if and only if there is an embedding of $\ZZ[\sqrt{-m/p}]$ (or $\ZZ[\zeta_4]$ if $m/p =2$) into the QM endomorphisms of $(A,\iota)$. Such an embedding of an order $R$ exists if and only if there is an optimal embedding of an order $R'\supset R$. In this case, the only orders which contain $\ZZ[\sqrt{-m/p}]$ are itself or $\ZZ\left[\frac{1 + \sqrt{-m/p}}{2}\right]$ if $m/p\equiv 3\bmod 4$. Respectively, their discriminants are $-4m/p$ and $-m/p$, so the result follows from Theorem \ref{EichlerEmbedding}.\end{proof}







\subsection{The proof when $p\nmid m$}

Once more, we shall use Hensel's Lemma to determine whether $C^D(N,d,m)(\QQ_p)$ is nonempty in terms of $\CX_{\overline\FF_p}$. If $p\nmid m$ then the action of $\Frob_p$ on the components and intersection points of $\CZ_{\overline\FF_p}\cong\CX_{\overline\FF_p}$ coincides with the action of $w_{mp}$. However, by Lemma \ref{ComponentsIntersectionPointsLemma}, the action of $w_{mp}$ on $X^D_0(N)_{\overline\FF_p}$ fixes no component. In fact, we conclude the following.

\begin{lemma}\label{CDpnmidmeven} Suppose that $p\nmid m$ is unramified in $\QQ(\sqrt d)$. Then $C^D(N,d,m)(\QQ_p)$ is nonempty if and only if there is a superspecial $w_{mp}$-fixed intersection point $x$ of even length in $X^D_0(N)_{\overline\FF_p}$.\end{lemma}

\begin{proof} If $C^D(N,d,m)(\QQ_p)$ is nonempty, then by Hensel's Lemma there is a smooth point of $\CZ(\FF_p)$. Therefore, there is a smooth point $P$ of $\CX(\overline\FF_p)$ fixed by $P \mapsto w_m P \Frob_p= w_{mp} P$. By Lemma \ref{ComponentsIntersectionPointsLemma}, the action of $w_{mp}$ on $X^D_0(N)_{\overline\FF_p}$ fixes no component. Therefore, $\pi(P)=x$ is the intersection point of two components. Since $P$ is smooth, $\pi^*(x(\Spec(\overline\FF_p)))\ne P(\Spec(\overline\FF_p))$. Therefore $\ell= \ell(x)>1$ and thus $\pi^*(x(\Spec(\overline\FF_p))) = \bigcup_{i=1}^{\ell-1} C_i$ with $C_i\cong \PP^1_{\overline\FF_p}$ as in Definition \ref{CDdesingularization}. Since $w_{mp}(x) = x$, $w_{mp} C_i = C_{\ell-i}$. Therefore, the only component which could be fixed by $w_{mp}$ is $C_{\ell/2}$. If such a component exists, then $\ell$ must be even. 

Conversely, if there is a superspecial $w_{mp}$-fixed intersection point $x$ of even length then 
$w_{mp}C_{\ell/2} = C_{\ell/2}$. There is thus a component of $\CZ_{\overline\FF_p}$ which is defined over $\FF_p$. It follows that there is a smooth point in $\CZ(\FF_p)$ and therefore $C^D(N,d,m)(\QQ_p)$ is nonempty.\end{proof}

The following example will illustrate ways that this can happen.

\begin{example} Let $\CX$ denote the regular $\ZZ_{13}$ model of $X^{26}_0(1)$ and let $\CY$ denote the regular $\ZZ_2$ model of $X^6_0(5)$. Depicted in figure \ref{twoCDfigure} are the dual graphs of $\CX_{\overline\FF_{13}}$ (on the left) and $\CY_{\overline\FF_2}$ (on the right). Respectively the arrows denote the action of $w_{2}\Frob_{13}$ and $\Frob_2$.
\begin{figure}
\[\fbox{ \xygraph{
!~-{@{-}@[|(2)]}
 & *+{\bullet}="b1" & & *+{\bullet}="d1" & & *+{x_1 \bullet}="x1" &  *+{\bullet}="y1"   & *+{\bullet x_1'}="x1p" \\
*+{a \bullet}="a" & & *+{\bullet}="c" & &*+{\bullet a'}="ap" &                      &                     &\\
 & *+{\bullet}="b2" & & *+{\bullet}="d2" & & *+{x_2 \bullet}="x2" &  *+{\bullet}="y2"   & *+{\bullet x_2'}="x2p"
"x1"-"y1"-"x1p"
"x2"-"y2"-"x2p"
"x1"-"x2p"
"x2"-"x1p"|-{\hole}
"x1":@{<->}@/_/"x1p"
"x2":@{<->}@/^/"x2p"
"a"-"c"-"ap"
"a"-"b1"-"d1"-"ap"
"a"-"b2"-"d2"-"ap"
"a":@{<->}@/_/"ap"
"b1":@{<->}@/_/"d1"
"b2":@{<->}@/^/"d2"
}}\]
\caption{The dual graphs of $\CX_{\overline\FF_{13}}$ and $\CY_{\overline\FF_2}$}
\label{twoCDfigure}
\end{figure}

It is easy to see that even though the intersection points of length 3 on $X^{26}_0(1)_{\overline\FF_{13}}$ are fixed by the action of $w_{2}\Frob_{13}$, they can not yield smooth rational points as the action exchanges $a$ with $a'$. The rational points here can only come from a fixed intersection point of length 2. Since there is such an intersection point on $X^{26}_0(1)_{\overline\FF_2}$, there is a component of $\CX_{\overline\FF_2}$ fixed by the action of $w_2\Frob_{13}$. Thus $C^{26}(1,d,2)(\QQ_{13})$ is nonempty for all $d$ such that $\left(\frac{d}{13}\right) = -1$. Similarly, because the two intersection points of length 2 on $X^6_0(5)_{\overline\FF_2}$ are not interchanged by the action of $\Frob_2$, there are components of $\CY_{\overline\FF_2}$ fixed by the action of $\Frob_2$. Therefore $X^6_0(5)(\QQ_2)$ is nonempty.\end{example}

The reader is encouraged to keep these examples in mind while reading the conditions of the following theorem.

\begin{theorem}\label{CDpnmidmTheorem} If $p\nmid m$, $C^D(N,d,m)(\QQ_p)$ is nonempty if and only if one of the following occurs.

\begin{enumerate}

\item $p=2$, $m=1$, $q \equiv 3 \bmod 4$ for all $q\mid (D/2)$, and $q \equiv 1\bmod 4$ for all $q\mid N$.

\item $p\equiv 1 \bmod 4$, $m = DN/(2p)$, $q\not\equiv 1\bmod 4$ for all $q\mid (D/p)$, and $q \not\equiv 3\bmod 4$ for all $q\mid N$.

\item $p =2$, $m = DN/2$, $q \equiv 3\bmod 4$ for all $q\mid (D/2)$ and $q \equiv 1\bmod 4$ for all $q\mid N$.

\item $p\equiv 1\bmod 4$, $m = DN/p$, , $q\not\equiv 1\bmod 4$ for all $q\mid (D/p)$, and $q \not\equiv 3\bmod 4$ for all $q\mid N$.

\end{enumerate}
\end{theorem}
\begin{proof} By Lemma \ref{CDpnmidmeven}, $C^D(N,d,m)(\QQ_p)$ is nonempty if and only if there is a superspecial $w_{mp}$-fixed intersection point of even length. By Corollary \ref{SuperspecialFixedZetaFour}, this can occur if and only if all of the following occur.

\begin{itemize}

\item $mp = 1,2,DN/2$ or $DN$.

\item for all $q\mid (D/p)$, either $q=2$ or $q\equiv 3\bmod 4$

\item for all $q\mid Np$, either $q=2$ or $q\equiv 1\bmod 4$

\end{itemize}

If $mp =2$ then $m=1$ and $p=2$. 
If $mp = DN/2$ then $p\ne 2$ and since $p\mid Np$, we must have $p\equiv 1\bmod 4$. 
If $mp = DN$ then $m = DN/p$ and either $p=2$ or $p\equiv 1\bmod 4$.\end{proof} 

To give an idea of the power of this theorem, let us show how it gives a new proof of the Theorem of Jordan-Livn\'e and Ogg.

\begin{corollary}\label{JordanLivneOggCor} Let $D$ be the discriminant of an indefinite $\QQ$-quaternion algebra, $N$ a square-free integer coprime to $D$ and $p\mid D$. Then $X^D_0(N)(\QQ_p)$ is nonempty if and only if one of the following occurs.

\begin{itemize}

\item $p=2$, $q\equiv 3\bmod 4$ for all $q\mid (D/2)$ and $q\equiv 1\bmod 4$ for all $q\mid N$

\item $p\equiv 1\bmod 4$, $D = 2p$ and $N=1$ 
\end{itemize}

\end{corollary}

\begin{proof} If $p=2$ we are in case (1) of Theorem \ref{CDpnmidmTheorem}. We cannot have $p= DN$ for any $p$ since $p\mid D$ and thus $D$ is divisible by at least two primes, so Theorem \ref{CDpnmidmTheorem} (3) or (4) cannot occur. If $DN = 2p$ with $p\equiv 1\bmod 4$ then we must at least have $(2p)\mid D$, but then $D= 2p$ and $N=1$.\end{proof}

Finally we give a family of examples of twists of $X^D_0(N)$ which have points everywhere locally.

\begin{example}\label{Shimura2qExample} Let $q$ be an odd prime, consider the curve $X^{2q}_0(1)$ and let $g$ be its genus. Let $p\equiv 3\bmod 8$ such that $\left(\frac{-p}{q}\right) = -1$ and for all odd primes $\ell$ less than $4g^2$, $\left(\frac{-p}{\ell}\right) = -1$. Consider the twist $C^{2q}(1,-p.2q)$ of $X^{2q}_0(1)$.
 
Note that since $p \equiv 3\bmod 8$ and $\left(\frac{-p}{q}\right) = -1$, both $2$ and $q$ are inert in $\QQ(\sqrt{-p})$. Therefore $C^{2q}(1,-p,2q)(\QQ_2)$ and $C^{2q}(1,-p,2q)(\QQ_q)$ are both nonempty by Corollary \ref{CDfullALCor}. 

Since $\left(\frac{-p}{q}\right) = -1$ and $p\equiv 3\bmod 4$, $\left(\frac{q}{p}\right) = -1$. Since $p \equiv 3\bmod 8$, $\left(\frac{-1}{p}\right) = -1$ and $\left(\frac{2}{p}\right) = -1$. Therefore $\left(\frac{-2q}{p}\right) = -1$ and $\left(\frac{-p}{2}\right) = \left(\frac{p}{2}\right) = -1$. Since we already had $\left(\frac{-p}{q}\right) = -1$, we may apply Theorem \ref{RamifiedMainTheorem} to say $C^{2q}(1,-p,2q)(\QQ_p)\ne \emptyset$.

Let $\ell\nmid 2pq$ be a prime. If $\ell > 4g^2$ then we may apply Theorem \ref{WeilBoundsTwist} to see that $C^{2q}(1,-p.2q)(\QQ_\ell)$ is nonempty. If $\ell < 4g^2$ then we may apply Corollary \ref{SuperspecialFullAL} to see that $C^{2q}(1,-p,2q)(\QQ_\ell)$ is nonempty. 

Finally, since $-p<0$, $C^{2q}(1,-p,2q) \not\cong_{\RR} X^{2q}_0(1)$, the latter of which does not have real points \cite[Theorem 55]{Cl03}. Therefore $(X^{2q}_0(1)/w_{2q})(\RR)\ne \emptyset$ if and only if $C^{2q}(1,-p,2q)(\RR)$ is nonempty. But then by Theorem \ref{EichlerEmbedding}, there is an embedding of $\ZZ[\sqrt{-2q}]$ into any maximal order in $B_{2q}$ and thus $X^{2q}_0(1)/w_{2q}$ has real points \cite[Theorem 3]{OggReal}.

\end{example}

\section{Local points when $p\mid N$}

Throughout this section we will fix $D$ the discriminant of an indefinite quaternion $\QQ$-algebra, $N$ a square-free integer coprime to $D$, a square-free integer $d$, an integer $m\mid DN$, and a prime $p\mid N$ unramified in $\QQ(\sqrt d)$. Let $w_m$ be as in Definition \ref{AtkinLehnerOO}. Let $X^D_0(N)_{/\QQ}$ be as defined in Definition \ref{DefnCoarseModuliScheme}, and let $C^D(N,d,m)_{/\QQ}$ be its twist by $\QQ(\sqrt d)$ and $w_m$. 
The purpose of this section is to prove the following theorem. 

\begin{theorem}\label{mainDRtheorem} Let $p\mid N$ be unramified in $\QQ(\sqrt d)$ and $m\mid DN$. We have $C^D(N,d,m)(\QQ_p)$ nonempty if and only if the conditions of (a) or (b) hold.

\begin{enumerate}[(a)]

\item $p$ is split in $\QQ(\sqrt d)$ and one of the following conditions holds.

\begin{itemize}

\item $D=1$ [Lemma \ref{cusplemma}].

\item $p=2$, $D = \prod_i p_i$ with each $p_i \equiv 3\bmod 4$, and $N/p = \prod_j q_j$ with each $q_j\equiv 1\bmod 4$ [Lemma \ref{DRpnmidmzetafourconditions}].

\item $p=3$, $D = \prod_i p_i$ with each $p_i \equiv 2\bmod 3$, and $N/p = \prod_j q_j$ with each $q_j\equiv 1\bmod 3$ [Lemma \ref{DRpnmidmzetasixconditions}].

\item $TF'(D,N,1,p)>0$ [Definition \ref{defntfprime}, Lemma \ref{DRtracepoints}].

\end{itemize}

\item $p$ is inert in $\QQ(\sqrt d)$, and there are prime factorizations $Dp = \prod_i p_i$, $N/p = \prod_j q_j$ such that one of the following two conditions holds

\begin{enumerate}[(i)]

\item $p\mid m$, and one of the following two conditions holds [Theorem \ref{DRpdividesminert}].

\begin{itemize}

\item $p=2$, $m = p$ or $DN$, for all $i$, $p_i\equiv 3\bmod 4$, and for all $j$, $q_j \equiv 1\bmod 4$.

\item $p \equiv 3 \bmod 4$, $m = p$ or $2p$, for all $i$, $p_i \not\equiv 1\bmod 4$, and for all $j$, $q_j \not\equiv 3\bmod 4$.

\end{itemize}

\item $p\nmid m$ and one of the following nine conditions holds.

\begin{itemize}

\item $m= D =1$ [Lemma \ref{cusplemma}].

\item $p=2$, $m=1$, for all $i$, $p_i\equiv 3\bmod 4$, and for all $j$, $q_j \equiv 1\bmod 4$ [Lemma \ref{DRpnmidmzetafourconditions}].

\item $p=3$, $m=1$, for all $i$, $p_i\equiv 2\bmod 3$, and for all $j$, $q_j\equiv 1\bmod 3$ [Lemma \ref{DRpnmidmzetasixconditions}].

\item $p\equiv 3\bmod 4$, $m = DN/2p$, $p_i \not\equiv 1\bmod 4$ for all $i$, and $q_j \not\equiv 3\bmod 4$ for all $j$ [Lemma \ref{DRpnmidmzetafourconditions}].

\item $p\equiv 2\bmod 3$, $m = DN/3p$, $p_i \not\equiv 1\bmod 3$ for all $i$, and $q_j \not\equiv 2\bmod 3$ for all $j$ [Lemma \ref{DRpnmidmzetasixconditions}].

\item $m = DN/p$, $p_i \not\equiv 1\bmod 4$ for all $i$, and $q_j \not\equiv 3\bmod 4$ for all $j$ [Lemma \ref{DRpnmidmzetafourconditions}].

\item $m= DN/p$, $p_i \not\equiv 1\bmod 3$ for all $i$, and $q_j\not\equiv 2\bmod 3$ for all $j$ [Lemma \ref{DRpnmidmzetasixconditions}].

\item $TF'(D,N,m,p)>0$ [Definition \ref{defntfprime}, Lemma \ref{DRtracepoints}]

\end{itemize}

\end{enumerate}

\end{enumerate}

\end{theorem}

As a special case, we recover the following explicit numerical conditions.

\begin{corollary}\label{mainDRCor}
Let $p$ be a prime dividing $N$ such that $p$ is unramified in $\QQ(\sqrt d)$. Then $C^D(N,d,DN)(\QQ_p)$ is nonempty if and only if 

\begin{itemize}

\item $p$ is split in $\QQ(\sqrt d)$ and one of the following conditions holds.

\begin{itemize}

\item $D=1$.

\item $p=2$, $D = \prod_i p_i$ with each $p_i \equiv 3\bmod 4$, and $N/p = \prod_j q_j$ with each $q_j\equiv 1\bmod 4$.

\item $p=3$, $D = \prod_i p_i$ with each $p_i \equiv 2\bmod 3$, and $N/p = \prod_j q_j$ with each $q_j\equiv 1\bmod 3$.

\item $TF'(D,N,1,p)>0$.

\end{itemize}

\item $p$ is inert in $\QQ(\sqrt d)$ with $Dp = \prod_i p_i$, $N/p = \prod_j q_j$ such that one of the following holds.

\begin{itemize}

\item $p=2$, for all $i$, $p_i\equiv 3\bmod 4$ and for all $j$, $q_j \equiv 1\bmod 4$.

\item $p \equiv 3 \bmod 4$, $D=1$ and $N = p$ or $2p$.
\end{itemize}
\end{itemize}
\end{corollary}



To prove Theorem \ref{mainDRtheorem}, we will have to make the following definitions.

\begin{definition}\label{regularDRmodels} Assume that $p\mid N$. Let $X^D_0(N)_{/\ZZ_p}$ be as in Theorem \ref{DeligneRapoportModelTheorem} and let $\pi: \CX \to X^D_0(N)$ be a minimal desingularization, so that $\CX_{\ZZ_p}$ is a regular model for $X^D_0(N)_{\QQ_p}$.\end{definition}

Note that if $n\mid DN$ then extending the automorphism $w_n$ from Definition \ref{AtkinLehnerOO} to $\CX$ makes sense. This is because $w_n: X^D_0(N) \to X^D_0(N)$ induces a birational morphism $\CX \dashedrightarrow \CX$ permuting the components of $\CX_{\FF_p}$. Therefore $w_n$ on $X^D_0(N)$ induces an isomorphism $\CX \to \CX$ \cite[Remark 8.3.25]{Liu}.

The model $\CX$ is equipped with a closed embedding $c': X^D_0(N/p)_{/\FF_p} \to \CX$ such that $\pi c' = c$, the embedding defined in Theorem \ref{DeligneRapoportModelTheorem}. Let $\sigma$ be such that $\langle \sigma\rangle = \Aut_{\ZZ_p}(\ZZ_{p^2})$.

\begin{definition}Let $\CZ$ be the \'etale quotient of $\CX_{\ZZ_{p^2}}$ by the action of $w_m \circ \sigma$.\end{definition}

Note that if $p$ is inert in $\QQ(\sqrt d)$ then $\ZZ_p[\sqrt d] \cong \ZZ_{p^2}$ and thus the generic fiber of $\CZ$ is $C^D(N,d,m)_{\QQ_p}$. Therefore $\CZ$ is a regular model of $C^D(N,d,m)_{\QQ_p}$ if $p$ is inert in $\QQ(\sqrt d)$. 

We also note that if $p$ is split in $\QQ(\sqrt d)$, or if $p$ is inert and $m=1$, then $C^D(N,d,m)_{\QQ_p} \cong X^D_0(N)_{\QQ_p}$. Therefore, if $p$ is split in $\QQ(\sqrt d)$, we can consider $d'$ to be any square-free integer such that $p$ is inert in $\QQ(\sqrt{d'})$ and $\CZ'$ to be the regular model of $C^D(N,d',1)_{\QQ_p} \cong X^D_0(N)_{\QQ_p}$. Therefore, we shall obtain our results when $p$ is split as a corollary to our results when $p\nmid m$. 

We shall organize our results into two sections. In the first, we will consider the case when $p\mid m$. In that case, $w_m$ and thus the twisted action of Galois will permute $c'(X^D_0(N/p)_{\FF_p})$ and $w_p c'(X^D_0(N/p)_{\FF_p})$ on the special fiber. 
In the second, we will consider the case when $p\nmid m$ and we may have to additionally allow for points on $c'(X^D_0(N/p)_{\FF_p})$. Note also that if $X^o$ denotes the complement of the superspecial points in $X$, $X^D_0(N)_{\FF_p}^o = c'(X^D_0(N/p)_{\FF_p}^o)\coprod w_p c' (X^D_0(N/p)_{\FF_p}^o)$. 

\subsection{The proof when $p\mid m$ is inert}

Suppose that $D$ is the discriminant of an indefinite quaternion $\QQ$-algebra, $N,d$ are square-free integers with $(D,N)=1$, $m\mid DN$, and $p\mid m$ is inert in $\QQ(\sqrt d)$. Fix $\CX$ and $\CZ$ as in Definition \ref{regularDRmodels}. If $p\mid m$, the action of $w_m$ on the regular model $\CX$ interchanges $c'(X^D_0(N/p)_{\FF_p})$ and $w_pc'(X^D_0(N)_{\FF_p}$. Therefore if $P$ denotes an element of $\CZ(\FF_p)$ then $\pi(P(\Spec(\FF_p))$ must lie on both copies of $X^D_0(N/p)_{\FF_p}$. This is to say that the base change to $\overline\FF_p$ of $\pi P$ is a superspecial point, say $x$.

\begin{lemma}\label{DRpdividesmcondition} If $D,N,d,m,p$ are as described in the beginning of this chapter and $p\mid m$ is inert in $\QQ(\sqrt d)$, then $C^D(N,d,m)(\QQ_p)\ne \emptyset$ if and only if there is a superspecial $w_{m/p}$-fixed point $x\in X^D_0(N)(\overline \FF_p)$ of even length.
\end{lemma}

\begin{proof}
By abuse of notation, let $\Frob_p = \phi_1^*:\Spec(\overline\FF_p) \to \Spec(\overline\FF_p)$ where $\phi_1: \overline\FF_p \to\overline\FF_p$. Note that under the bijection from $\CZ(\overline \FF_p)$ to $\CX(\overline \FF_p)$, the Galois action $P \mapsto P\Frob_p$ on $\CZ(\overline\FF_p)$ translates to the action of $P \mapsto w_m P\Frob_p$ on $\CX(\overline\FF_p)$.

Suppose that $C^D(N,d,m)(\QQ_p)$ is nonempty. Then by Hensel's Lemma \cite[Lemma 1.1]{JoLi} there must be an element of $\CZ^{sm}(\FF_p)$, or rather a smooth point such that $P = w_m P \Frob_p$ in $\CX(\overline\FF_p)$. Since $p\mid m$, $w_m$ interchanges $c(X^D_0(N/p)_{\overline\FF_p})$ with $w_pc(X^D_0(N/p)_{\overline\FF_p})$. A smooth fixed point $P$ of $w_m\circ \Frob_p$ must therefore map to a superspecial point under $\pi$. 

Suppose there is such a smooth fixed point $P$. Let $\ell = \ell(x)$, so that smoothness implies $\ell >1$. 
We have $\pi^* x(\Spec(\overline\FF_p)) = \bigcup_{i=1}^{\ell-1} C_i$ with $C_i \cong \PP^1_{\overline \FF_p}$ and if $i< j$, $$C_i \cdot C_j = \begin{cases} 1 & j = i + 1, 1 \le i < \ell \\ 0 & else \end{cases}.$$  By Lemma \ref{ALGaloisLemma} $x\Frob_p = w_p(x)$, so we have $w_m \circ \Frob_p(x) = w_{m}w_p(x) = w_{m/p}(x)$. Therefore by continuity, $w_{m/p}$ fixes each $C_i$ and for each $i$, $w_p C_i = C_{\ell -i}$. 
Therefore, unless $\ell$ is even we arrive at a contradiction.

Conversely suppose that there is a superspecial point $x$ such that $\ell = \ell(x)$ is even and $w_{m/p}(x) = x$. Then we have $C_1, \dots C_{\ell -1}$ fixed by $w_{m/p}$ by assumption. Since $w_p$ fixes $C_{\ell/2}$, it follows that $C_{\ell/2}$ is defined over $\FF_p$. 
Therefore by Hensel's Lemma, $C^D(N,d,m)(\QQ_p) \ne\emptyset$.\end{proof}

To illustrate this Lemma, consider the following example.

\begin{example} The diagram in figure \ref{DR39figure} depicts the special fiber of $\CX$ over $\overline\FF_3$ where $\CX$ denotes the regular $\ZZ_3$-model of $X^1_0(39) = X_0(39)$ with the action $w_{39}\Frob_3$ given by the arrows.

\begin{figure}
\[ \fbox{ \xygraph{
!~-{@{-}@[|(2)]}
&*+{}="x1" & &   &   & *+{}="x1p"& \\
*+{\PP^1_{\overline\FF_3}}="x2"& & & *+{}="y2" &   & & *+{}="x2p" \\
*+{\PP^1_{\overline\FF_3}}="x3"& & & *+{}="y3" &   & & *+{}="x3p" \\
*+{\PP^1_{\overline\FF_3}}="x4b"&*+{}="x4" &  & &   & *+{}="x4p"&*+{\PP^1_{\overline\FF_3}}="x4bp" \\
&*+{X_0(13)_{\overline\FF_3}}="x5" & *+{\phantom{AA}}="x5m" &  & *+{\phantom{AA}}="x5mp"  & *+{X_0(13)_{\overline\FF_3}}="x5p"& \\
"x1"-@`{"x1p"+(0,0.5)}"x5p"
"x1p"-@`{"x1"+(0,0.5)}"x5"
"x2"-"x2p"
"x3"-"x3p"
"x4b"-"x5mp"
"x4bp"-"x5m"
"x4" :@{<->}@/^/ "x4p"
"y2" :@{<->} "y3"
"x5m" :@{<->} "x5mp"
} } \]
\caption{The $\overline\FF_3$ special fiber of $\CX$}
\label{DR39figure}
\end{figure}

Note the resolutions of the four superspecial points of $X_0(39)_{\overline\FF_3}$: 1 of length 1, 2 of length 2 and 1 of length 3. Note also that while there are superspecial points of length 2, and there are some superspecial points fixed by the action of $w_{39}\Frob_3$, there are no superspecial points of length 2 fixed by the action of $w_{13}$. As a consequence, if $3$ is inert in $\QQ(\sqrt d)$ then $C^1(39,d,39)(\QQ_3)$ is empty, because there are no smooth fixed points of $w_{39}\Frob_3$ on $\CX_{\overline\FF_3}$.\end{example}

This example illustrates an error in the criterion of Theorem 1.1(3) in the recent paper of Ozman \cite{Oz09}. The correct numerical criterion is properly given by Corollary \ref{mainDRCor}, via the following Theorem.

\begin{theorem}\label{DRpdividesminert} Suppose that $D,N,d,m$ and $p$ are as in Theorem \ref{mainDRtheorem} and $p\mid m$ is inert in $\QQ(\sqrt d)$. Then $C^D(N,d,m)(\QQ_p)\ne \emptyset$ if and only if
 
\begin{itemize}
 \item $p=2$, $m = p$ or $DN$, for all $q\mid D$, $q\equiv 3 \bmod 4$, and for all $q\mid (N/2)$, $q\equiv 1\bmod 4$, or

 \item $p\equiv 3 \bmod 4$, $m= p$ or $2p$, for all $q\mid D$ $q\not\equiv 1\bmod 4$, and for all $q\mid (N/p)$, $q\not\equiv 3\bmod 4$.
\end{itemize}

\end{theorem}

\begin{proof} By Lemma \ref{DRpdividesmcondition}, $C^D(N,d,m)(\QQ_p)$ is nonempty if and only if there is a superspecial $w_{m/p}$-fixed point of even length in $X^D_0(N)(\overline\FF_p)$. By Lemma \ref{ComponentsIntersectionPointsLemma}, the QM endomorphism ring of a superspecial point on $X^D_0(N)(\overline\FF_p)$ has discriminant $D' = Dp$ and level $N' = N/p$. Note that $D'N' = DN$. By Lemma \ref{SuperspecialFixedZetaFour}, there is a superspecial $w_{m/p}$-fixed point of even length if and only if

\begin{itemize}

\item $m/p = 1,2,DN/2$ or $DN$ and 

\item for all $q\mid Dp$, $q=2$ or $q\equiv 3\bmod 4$ and

\item for all $q\mid (N/p)$, $q=2$ or $q\equiv 1\bmod 4$.
\end{itemize}

We may immediately see that $(m/p)\mid (DN/p)< DN$ so $m/p \ne DN$. If $m/p=1$ then $m=p$ and either $p=2$, or $p\equiv 3\bmod 4$. If $p=2$, $2\nmid (DN/2)$ so for all $q\mid D$, $q\equiv 3 \bmod 4$, and for all $q\mid (N/2)$, $q\equiv 1\bmod 4$. If $m/p = 2$ then $m= 2p$ and we conclude that $p\equiv 3\bmod 4$. If $m/p = DN/2$ then $DNp/2 = m\mid DN$ and we conclude that $p=2$.\end{proof}




\subsection{The proof when $p\nmid m$ is split or inert}

We begin with the following observation regarding cusps, which are points that can only exist if $D=1$.

\begin{lemma}{\cite[Proposition 3]{OggHyperelliptic}}\label{cusplemma} If $N$ is square-free and $m\mid N$, then $w_m$ fixes a cusp of $X^1_0(N)$ if and only if $m=1$. \end{lemma}

Therefore if $N,d$ are square-free and $p\mid N$ is a prime, then $C^1(N,d,m)(\QQ_p)$ contains a cusp if and only if either $p$ is split in $\QQ(\sqrt d)$ or $m=1$. We now illustrate the three ways in which $C^D(N,d,m)(\QQ_p)$ could be nonempty via the following example.

\begin{example} Let $D = 6, N = p = 11$, $m=1$ and $d$ any integer such that $11$ is inert in $\QQ(\sqrt d)$, say $-1$ for instance. Note that $C^6(11,-1,1) \cong_{\QQ_{11}} X^6_0(11)$. Let $\CX$ denote the regular $\ZZ_{11}$-model of $X^6_0(11)$. Then the diagram in figure \ref{DR611figure} depicts $\CX_{\overline\FF_{11}}$ where the arrows describe the action of $w_1\Frob_{11} = \Frob_{11}$.

\begin{figure}
\[ \fbox{ \xygraph{
!~-{@{-}@[|(1)]}
*+{\PP^1_{\overline\FF_{11}}}="x1" & *+{}="x1b" & & & *+{}="x1bp" & *+{\PP^1_{\overline\FF_{11}}}="x1p" \\
*+{\PP^1_{\overline\FF_{11}}}="x2" & & *+{}="y1p" & *+{}="y1" & & *+{\PP^1_{\overline\FF_{11}}}="x2p" \\
 & *+{}="z1" & *+{}="y2p" & *+{}="y2" & *+{}="z1p" & \\
*+{\PP^1_{\overline\FF_{11}}}="a1" &  & & &  & *+{}="a1p" \\
*+{\PP^1_{\overline\FF_{11}}}="b1" &  & & &  & *+{}="b1p" \\
 & *+{X^6_0(1)_{\overline\FF_{11}}}="w1b" & & & *+{X^6_0(1)_{\overline\FF_{11}}}="w1bp" & \\
"x1b"-"w1b"
"x1bp"-"w1bp"
"a1"-"a1p"
"b1"-"b1p"
"x1"-"y1"
"x1p"-"y1p"
"x2"-"y2"
"x2p"-"y2p"
"x1" :@{<->} "x2"
"x1p" :@{<->} "x2p"
"a1" :@{<->} "b1"
"z1" :@(lu,ld) "z1"
"z1p" :@(ru,rd) "z1p"
} } \]
\caption{The $\overline\FF_{11}$ special fiber of $\CX$}
\label{DR611figure}
\end{figure}

We check to see if $X^6_0(11)(\QQ_{11})$ is nonempty as follows. Although there are intersection points of length 2 and 3 on $X^6_0(11)_{\overline\FF_{11}}$, none are fixed by the action of $\Frob_{11}$. Therefore the only way that $X^6_0(11)(\QQ_{11})$ could be nonempty would be if there were a non superspecial point in $X^6_0(1)(\FF_{11})$. Using the trace formula, or the fact that there are only four superspecial points on $X^6_0(1)_{\FF_{11}}\cong \PP^1_{\FF_{11}}$, we conclude that $X^6_0(11)(\QQ_{11})$ is nonempty.

\end{example}

\begin{lemma}\label{DRpmidmlemma} Let $D,N,d,m,p$ be as in Theorem \ref{mainDRtheorem} and suppose $p\nmid m$ is unramified in $\QQ(\sqrt d)$. Suppose that $C^D(N,d,m)(\QQ_p)$ does not contain a cusp. Then $C^D(N,d,m)(\QQ_p)\ne \emptyset$ if and only if one of the following occurs.

\begin{itemize}

\item There is a superspecial $w_{mp}$-fixed point of even length on $X^D_0(N)(\overline\FF_p)$.

\item There is a superspecial $w_{mp}$-fixed point of length divisible by three on $X^D_0(N)(\overline\FF_p)$.

\item There is a non-superspecial point of $C^D(N/p,d,m)(\FF_p)$.

\end{itemize}
\end{lemma}

\begin{proof}Recall the regular models $\CX,\CZ$ of Definition \ref{regularDRmodels}. Recall that there is a bijection from $\CZ(\overline\FF_p)$ to $\CX(\overline\FF_p)$, under which the Galois action $P \mapsto P\Frob_p$ on $\CZ(\overline\FF_p)$ translates to the action $P \mapsto w_m P\Frob_p$ on $\CX(\overline\FF_p)$. 

By Lemma \ref{ALGaloisLemma}, the action of $\Frob_p$ on the superspecial points of $X^D_0(N)_{\overline\FF_p}$ is the action of $w_p$. Therefore a superspecial $\FF_p$-rational point of $\CZ$ corresponds to a superspecial $w_{mp}$-fixed point of $X^D_0(N)_{\overline\FF_p}$.

Suppose now that $C^D(N,d,m)(\QQ_p)$ is nonempty, or equivalently by Hensel's Lemma that $\CZ^{sm}(\FF_p)$ is nonempty. Suppose further that there are no superspecial $w_{mp}$-fixed points of length divisible by 2 or 3, that is, all superspecial points fixed by $w_{mp}$ have length 1. It follows that if $P$ is a smooth fixed point of $w_{mp}$ in $\CX(\overline\FF_p)$, then $\pi( P)= x$ is not superspecial. 

Conversely, suppose first that there is an $\FF_p$-rational point of $\CZ$ which is not superspecial. By the embedding $c:X^D_0(N/p)_{\FF_p}\to X^D_0(N)_{\FF_p}$, there is a non-superspecial $\FF_p$-rational point of $\CZ$. Since $X^D_0(N)_{\overline\FF_p}$ is smooth away from superspecial points, $C^D(N,d,m)(\QQ_p)$ is nonempty by Hensel's lemma.

Now suppose there is a superspecial $w_{mp}$-fixed point $x$ with $\ell = \ell(x) >1$. It follows that $\pi^*(x(\Spec(\overline\FF_p))) = \bigcup_{i=1}^{\ell-1} C_i$ with $C_i \cong \PP^1_{\overline\FF_p}$ and at most two singular points in $\CX_{\overline\FF_p}$ on each $C_i$. Since $w_m x \Frob_p  = w_{mp}(x) = x$, for all $i$, $w_m\Frob_p C_i = w_{mp} C_i = C_i$ by continuity of $\pi$. Therefore $C_i$ defines an $\FF_p$-rational component of $\CZ_{\overline\FF_p}$ with at most two singular points. Therefore $\CZ^{sm}(\FF_p)$ is nonempty and by Hensel's Lemma, $\CZ(\QQ_p)$ is nonempty.\end{proof}


\begin{lemma}\label{DRpnmidmzetafourconditions}

There is a superspecial $w_{mp}$-fixed point of even length on $X^D_0(N)_{\overline\FF_p}$ if and only if one of the following occurs.

\begin{enumerate}

\item $p=2$, $m=1$, $q\equiv 3\bmod 4$ for all primes $q\mid D$, and $q \equiv 1\bmod 4$ for all primes $q\mid (N/2)$.

\item $p\equiv 3\bmod 4$, $2\mid DN/p$, $m = DN/2p$, $q \not\equiv 1\bmod 4$ for all primes $q\mid D$, and $q \not\equiv 3\bmod 4$ for all primes $q\mid (N/p)$.

\item $m = DN/p$, $p \not\equiv 1\bmod 4$, $q \not\equiv 1\bmod 4$ for all primes $q\mid D$, and $q \not\equiv 3\bmod 4$ for all primes $q\mid (N/p)$.

\end{enumerate}

\end{lemma}

\begin{proof} By Lemma \ref{SuperspecialFixedZetaFour}, there is a superspecial $w_{mp}$-fixed point of even length if and only if all of the following occur:

\begin{itemize}

\item $mp = 1,2,DN/2$ or $DN$,

\item for all primes $q\mid Dp$, $q=2$ or $q\equiv 3\bmod 4$,

\item for all primes $q\mid (N/p)$, $q=2$ or $q\equiv 1\bmod 4$.
\end{itemize}

If $mp = 2$ then $p=2$, so $2\nmid(DN/p)$, and $m=1$. If $mp = DN/2$ then $p\mid (DN/2)$ and thus $p\ne 2$ because $DN$ is square-free. It follows that $m= DN/(2p)$ with $p\equiv 3\bmod 4$. The only remaining case is $mp = DN$, in which case $p\not\equiv 1\bmod 4$.
\end{proof}

\begin{lemma}\label{DRpnmidmzetasixconditions}

There is a superspecial point of length divisible by three in $X^D_0(N)(\overline\FF_p)$ fixed by $w_{mp}$ if and only if one of the following occurs.

\begin{itemize}

\item $p=3$, $m=1$, $q\equiv 2\bmod 3$ for all primes $q\mid D$, and $q \equiv 1\bmod 3$ for all primes $q\mid (N/3)$.

\item $p\equiv 2\bmod 3$, $3\mid DN/p$, $m = DN/3p$, $q \not\equiv 1\bmod 3$ for all primes $q\mid D$, and $q \not\equiv 2\bmod 3$ for all primes $q\mid (N/p)$.

\item $m = DN/p$, $p \not\equiv 1\bmod 3$, $q \not\equiv 1\bmod 3$ for all primes $q\mid D$, and $q \not\equiv 2\bmod 3$ for all primes $q\mid (N/p)$.

\end{itemize}

\end{lemma}

\begin{proof} 
By Lemma \ref{SuperspecialFixedZetaSix}, there is a superspecial $w_{mp}$-fixed point of length divisible by three if and only if all of the following occur:

\begin{itemize}

\item $mp = 1,3,DN/3$ or $DN$, 

\item for all primes $q\mid Dp$, $q=3$ or $q\equiv 2\bmod 3$,

\item for all primes $q\mid (N/p)$, $q=3$ or $q\equiv 1\bmod 3$.
\end{itemize}

If $mp = 3$ then $p=3$, so $3\nmid(DN/p)$, and $m=1$. If $mp = DN/3$ then $p\mid (DN/3)$ and thus $p\ne 3$ because $DN$ is square-free. It follows that $m= DN/(3p)$ with $p\equiv 2\bmod 3$. The only remaining case is $mp = DN$, and thus $p=3$ or $p\equiv 2\bmod 3$.
\end{proof}

We note that since $p\nmid (DN/p)$ and $m \mid (DN/p)$, we may recall $TF(D,N/p,m,p) = (p+1) - \tr(T_{pm})$ as in Definition \ref{newTF}. With this in mind we make the following definition.

\begin{definition}\label{defntfprime} If $p\mid N$ and $m \mid (DN/p)$, we let $$TF'(D,N,m,p) := \begin{cases} TF(D,N/p,m,p) - (\frac{e_{Dp,N/p}(-4)}{w(-4)} + \frac{e_{Dp,N/p}(-8)}{w(-8)}) & mp = 2 \\ TF(D,N/p,m,p) - (\frac{e_{Dp,N/p}(-4mp)}{w(-4mp)} + \frac{e_{Dp,N/p}(-mp)}{w(-mp)}) & mp \ne 2, mp \not\equiv 3 \bmod 4 \\ TF(D,N/p,m,p) - \frac{e_{Dp,N/p}(-4mp)}{w(-4mp)}  & mp \equiv 3\bmod 4\end{cases}. $$
\end{definition}

\begin{lemma}\label{DRtracepoints} There is a non-superspecial $\FF_p$-rational point of $\CZ$ if and only if $TF'(D,N,m,p)>0$.
\end{lemma}

\begin{proof} Let $\CY_{/\ZZ_p}$ denote the smooth model of $C^D(N/p,d,m)$. By Theorem \ref{NumberPointsGoodReductionTwist}, $\#\CY(\FF_p) = (p+1) - \tr(T_{pm}) = TF(D,N/p,m,p)$. By Lemma \ref{ALGaloisLemma}, there is a superspecial point in $\CY(\FF_p)$ if and only if there is a superspecial point fixed by $w_{mp}$ in $X^D_0(N)(\overline\FF_p)$. By Corollary \ref{EmbeddingFixedPoint}, there is a superspecial point $x$ in $X^D_0(N/p)(\overline\FF_p)$ fixed by $w_{mp}$ if and only if $\ZZ[\sqrt{-mp}]$ (or $\ZZ[\zeta_4]$ if $mp =2$) embeds into $\End_{\iota(\OO)}(A)$ where $(A,\iota)$ corresponds to $x$. 

We now count the number $n_{mp}$ of $w_{mp}$-fixed superspecial points so we can subtract them off. Suppose that $\OO'$ is an Eichler order $\OO'$ of level $N/p$ in $B_{Dp}$, $\wp_m$ is the unique two-sided ideal of norm $mp$ in $\OO'$, and $M_1, \dots, M_h$ are right ideals of $\OO'$ which form a complete set of representatives of $\Pic(D/p,Np)$. Under Lemma \ref{ComponentsIntersectionPointsLemma}, $n_{mp}$ is the number of indices $i$ such that $M_i \cong M_i \otimes \wp_m$. Thus \cite[p.152]{Vigneras}, the number of such superspecial fixed points is the number of embeddings of $\ZZ[\sqrt{-mp}]$ (or $\ZZ[\zeta_4]$ if $mp =2$) into any left order of an $M_i$. If $mp=2$ the number of these is $\frac{e_{Dp,N/p}(-4)}{w(-4)} + \frac{e_{Dp,N/p}(-8)}{w(-8)}$. If $mp\ne 2$ and $mp \not \equiv 3\bmod 4$ then the number of these is $\frac{e_{Dp,N/p}(-4mp)}{w(-4mp)}$. If $mp \equiv 3\bmod 4$ then the number of these is $\frac{e_{Dp,N/p}(-mp)}{w(-mp)} + \frac{e_{Dp,N/p}(-4mp)}{w(-4mp)}$.\end{proof}

Unlike the case $p\mid D$, the conditions under which $X^D_0(N)(\QQ_p)$ is nonempty were not previously known when $p\mid N$ and $D>1$. In the following, we eschew the $TF'$ notation to show how it is possible to directly compute on the special fiber of this Shimura curve. Note that condition (4) is simply the inequality $TF'(D,N,1,p) >0$.


\begin{theorem}\label{DRsplitconditions} Let $D$ be the discriminant of an indefinite $\QQ$-quaternion algebra, $N$ a square-free integer coprime to $D$ and $p\mid N$. Then $X^D_0(N)(\QQ_p)$ is nonempty if and only if one of the following occurs.

\begin{enumerate}

\item $D=1$.

\item $p=2$, for all $q\mid D$, $q\equiv 3\bmod 4$, and for all $q\mid (N/2)$, $q \equiv 1\bmod 4$.

\item $p=3$, $m=1$, for all $q\mid D$, $q\equiv 2\bmod 3$, and for all $q\mid (N/3)$, $q \equiv 1\bmod 3$.

\item The following inequality holds $$\sum_{\stackrel{s = -\lfloor 2\sqrt{p}\rfloor}{s\ne 0}}^{\lfloor2\sqrt{p}\rfloor}\left(\sum_{f\mid f(s^2-4p)}\frac{e_{D,N/p}\left(\frac{s^2 - 4p}{f^2}\right)}{w\left(\frac{s^2 -4p}{f^2}\right)}\right)>0.$$

\end{enumerate}
\end{theorem}
\begin{proof} First we note that if $D=1$, then there is a $\QQ$-rational cusp by Lemma \ref{cusplemma}. Set $m=1$ and assume $D\ne 1$. By Lemma \ref{DRpmidmlemma}, $X^D_0(N)(\QQ_p)$ is non-empty if and only if one of the following occurs.

\begin{itemize}
\item There is a superspecial $w_{p}$-fixed point of even length in $X^D_0(N)(\overline\FF_p)$.

\item There is a superspecial $w_{p}$-fixed point of length divisible by three in $X^D_0(N)(\overline\FF_p)$.

\item There is a non-superspecial $\FF_p$-rational point.

\end{itemize}

By Lemma \ref{DRpnmidmzetafourconditions}, there is a $w_p$ fixed point of even length if and only if one of the following occurs.

\begin{itemize}

\item $p=2$, for all $q\mid D$, $q\equiv 3\bmod 4$ and for all $q\mid (N/2)$, $q \equiv 1\bmod 4$

\item $p\equiv 3\bmod 4$ and $DN = 2p$

\item $DN = p$ and $p =2 $ or $p \equiv 3\bmod 4$

\end{itemize}

However, if either of the latter two occurs, $D=1$ in contradiction to our assumption.

By Lemma \ref{DRpnmidmzetasixconditions}, there is a $w_p$ fixed point of length divisible by three if and only if one of the following occurs.

\begin{itemize}

\item $p=3$, for all $q\mid D$, $q\equiv 2\bmod 3$ and for all $q\mid (N/3)$, $q \equiv 1\bmod 3$

\item $p\equiv 2\bmod 3$ and $DN = 3p$

\item $DN = p$ and $p =3 $ or $p \equiv 2\bmod 3$

\end{itemize}

Once again, if either of the latter two occurs, $D=1$. Suppose now that in addition to $D\ne 1$, all superspecial points have length 1, so the number of non-superspecial $\FF_p$-rational points on $X^D_0(N/p)$ can be written as $$(p+1) - \tr(T_{p}) - \sum_{f\mid f(-4p)} \frac{e_{Dp,N/p}\left(\frac{-4p}{f^2}\right)}{w\left(\frac{-4p}{f^2}\right)}.$$

Recall now Theorem \ref{EichlerSelbergTraceFormula}, the Eichler-Selberg trace formula on $H^0(X^D_0(N/p)_{\overline\FF_p},\Omega)$: $$\tr(T_p) = (p+1) - \sum_{s = -\lfloor 2 \sqrt{p}\rfloor}^{\lfloor 2\sqrt{p}\rfloor}\left(\sum_{f\mid f(s^2 - 4p)} \frac{e_{D,N/p}\left(\frac{s^2 - 4p}{f^2}\right)}{w\left(\frac{s^2 - 4p}{f^2}\right)}\right).$$

Therefore, if $p\ne 2$ there is a non-superspecial $\FF_p$-rational point of $X^D_0(N/p)$ if and only if the following quantity is nonzero.
\begin{eqnarray*}(p+1) 
& - & 
\left((p+1) - \sum_{s = -\lfloor 2 \sqrt{p}\rfloor}^{\lfloor 2\sqrt{p}\rfloor}\left(\sum_{f\mid f(s^2 - 4p)} \frac{e_{D,N/p}\left(\frac{s^2 - 4p}{f^2}\right)}{w\left(\frac{s^2 - 4p}{f^2}\right)}\right) \right)
- \sum_{f\mid f(-4p)} \frac{e_{Dp,N/p}\left(\frac{-4p}{f^2}\right)}{w\left(\frac{-4p}{f^2}\right)} \\ 
& = & 
\left(\sum_{\stackrel{s = -\lfloor 2 \sqrt{p}\rfloor}{s\ne 0}}^{\lfloor 2\sqrt{p}\rfloor}\left(\sum_{f\mid f(s^2 - 4p)} \frac{e_{D,N/p}\left(\frac{s^2 - 4p}{f^2}\right)}{w\left(\frac{s^2 - 4p}{f^2}\right)}\right)\right) 
+ \sum_{f\mid f(-4p)}
\frac{e_{D,N/p}\left(\frac{-4p}{f^2}\right) - e_{Dp,N/p}\left(\frac{-4p}{f^2}\right)}
{w\left(\frac{-4p}{f^2}\right)} 
\end{eqnarray*}

Now recall that $e_{D,N}(\Delta) = h(\Delta) \prod_{p\mid D} \left(1 - \left\{\frac{\Delta}{p}\right\}\right)\prod_{q\mid N} \left(1 + \left\{\frac{\Delta}{p}\right\}\right)$ and $f(\Delta)$ is the conductor of $R_\Delta$. Therefore $e_{Dp,N/p}(\Delta) = \left(1 - \left\{\frac{\Delta}{p}\right\}\right)e_{D,N/p}(\Delta)$ and thus $e_{D,N/p}(\Delta) - e_{Dp,N/p}(\Delta) = \left\{\frac{\Delta}{p}\right\}e_{D,N/p}(\Delta)$. 
However, consider that $f(-4p) = 1$ or $2$, depending on $p\bmod 4$. Moreover, if $p=2$ then $f(-8) =1$. 
Therefore, since $p\mid \frac{-4p}{f^2}$ for all $f\mid f(-4p)$, $\left\{\frac{\frac{-4p}{f^2}}{p}\right\} = 0$. Finally, if $p=2$ and (2) does not hold, then $e_{Dp,N/p}(-4) = 0$, and the formula of (4) still suffices.\end{proof}

We now find, for infinitely many pairs of integers $D$ and $N$, infinitely many nontrivial twists of $X^D_0(N)$ which have points everywhere locally.

\begin{example}\label{PrimeLevelTwistExample} Let $q$ be a prime which is $3\bmod 4$ and consider the curve $X^1_0(q)$. We will show that if $p\equiv 1\bmod 4$ is a prime such that $\left(\frac{q}{p}\right) = -1$ then $C^1(q,p,q)(\QQ_v)$ is nonempty for all places $v$ of $\QQ$. Since $p>0$, $C^1(q,p,q) \cong_\RR X^1_0(q)$ and thus $C^1(q,p,q)(\RR)\ne \emptyset$. We note that since $p\equiv 1\bmod 4$, $\QQ(\sqrt p)$ is ramified precisely at $p$. Therefore if $\ell \nmid pq$ is a prime, then $\ell$ is unramified in $\QQ(\sqrt p)$. If $\ell$ splits in $\QQ(\sqrt p)$, then $C^1(q,p,q) \cong_{\QQ_\ell} X^1_0(q)$ and thus $C^1(q,p,q)(\QQ_\ell)\ne \emptyset$. If $\ell$ is inert in $\QQ(\sqrt p)$, then $C^1(q,p,q)(\QQ_\ell)\ne\emptyset$ by Corollary \ref{SuperspecialFullAL}.

Since $p\equiv 1\bmod 4$, $\left(\frac{p}{q}\right) = \left(\frac{q}{p}\right) = -1$, and thus $q$ is inert in $\QQ(\sqrt p)$. Therefore by Theorem \ref{mainDRtheorem}(b), $C^1(q,p,q)(\QQ_q)$ is nonempty. Moreover, $\left(\frac{-q}{p}\right) = \left(\frac{q}{p}\right) = -1$ and so by Theorem \ref{RamifiedMainTheorem}, $C^1(q,p,q)(\QQ_p)\ne \emptyset$.\end{example}

\bibliographystyle{amsalpha}
\bibliography{MainThesisPreprint}

\newcommand{\etalchar}[1]{$^{#1}$}
\def\cprime{$'$}
\providecommand{\bysame}{\leavevmode\hbox to3em{\hrulefill}\thinspace}
\providecommand{\MR}{\relax\ifhmode\unskip\space\fi MR }
\providecommand{\MRhref}[2]{%
  \href{http://www.ams.org/mathscinet-getitem?mr=#1}{#2}
}
\providecommand{\href}[2]{#2}
\begin{thebibliography}{BCP97}

\bibitem[BCP97]{Magma}
Wieb Bosma, John Cannon, and Catherine Playoust, \emph{The {M}agma algebra
  system. {I}. {T}he user language}, J. Symbolic Comput. \textbf{24} (1997),
  no.~3-4, 235--265, Computational algebra and number theory (London, 1993).

\bibitem[BLR90]{BLR}
Siegfried Bosch, Werner L{\"u}tkebohmert, and Michel Raynaud, \emph{N\'eron
  models}, Ergebnisse der Mathematik und ihrer Grenzgebiete (3) [Results in
  Mathematics and Related Areas (3)], vol.~21, Springer-Verlag, Berlin, 1990.

\bibitem[CES03]{CES}
Brian Conrad, Bas Edixhoven, and William Stein, \emph{{$J_1(p)$} has connected
  fibers}, Doc. Math. \textbf{8} (2003), 331--408 (electronic).

\bibitem[Cla03]{Cl03}
Pete~L. Clark, \emph{Rational points on {A}tkin-{L}ehner quotients of {S}himura
  curves}, Ph.D. thesis, Harvard, 2003.

\bibitem[Cox89]{Cox}
David~A. Cox, \emph{Primes of the form {$x^2 + ny^2$}}, A Wiley-Interscience
  Publication, John Wiley \& Sons Inc., New York, 1989, Fermat, class field
  theory and complex multiplication.

\bibitem[Del69]{VAO}
Pierre Deligne, \emph{Vari\'et\'es ab\'eliennes ordinaires sur un corps fini},
  Invent. Math. \textbf{8} (1969), 238--243.

\bibitem[Dri76]{Dr76}
V.~G. Drinfel{\cprime}d, \emph{Coverings of {$p$}-adic symmetric domains},
  Funkcional. Anal. i Prilo\v zen. \textbf{10} (1976), no.~2, 29--40.

\bibitem[Eic56]{Eichler}
Martin Eichler, \emph{Modular correspondences and their representations}, J.
  Indian Math. Soc. (N.S.) \textbf{20} (1956), 163--206.

\bibitem[Eic73]{BasisProblem}
\bysame, \emph{The basis problem for modular forms and the traces of the
  {H}ecke operators}, Modular functions of one variable, {I} ({P}roc.
  {I}nternat. {S}ummer {S}chool, {U}niv. {A}ntwerp, {A}ntwerp, 1972), Springer,
  Berlin, 1973, pp.~75--151. Lecture Notes in Math., Vol. 320.

\bibitem[GR91]{Go91}
Josep Gonz{\`a}lez~Rovira, \emph{Equations of hyperelliptic modular curves},
  Ann. Inst. Fourier (Grenoble) \textbf{41} (1991), no.~4, 779--795.

\bibitem[GR04]{genustwoshimura}
Josep Gonz{\'a}lez and Victor Rotger, \emph{Equations of {S}himura curves of
  genus two}, Int. Math. Res. Not. (2004), no.~14, 661--674.

\bibitem[Hel07]{HelmPublished}
David Helm, \emph{On maps between modular {J}acobians and {J}acobians of
  {S}himura curves}, Israel J. Math. \textbf{160} (2007), 61--117.

\bibitem[JL85]{JoLi}
Bruce~W. Jordan and Ron~A. Livn{\'e}, \emph{Local {D}iophantine properties of
  {S}himura curves}, Math. Ann. \textbf{270} (1985), no.~2, 235--248.

\bibitem[Kan11]{Kani}
Ernst Kani, \emph{Products of {CM} elliptic curves}, Collect. Math. \textbf{62}
  (2011), no.~3, 297--339.

\bibitem[KR08]{Kontogeorgis}
Aristides Kontogeorgis and Victor Rotger, \emph{On the non-existence of
  exceptional automorphisms on {S}himura curves}, Bull. Lond. Math. Soc.
  \textbf{40} (2008), no.~3, 363--374.

\bibitem[Kur79]{Kurihara}
Akira Kurihara, \emph{On some examples of equations defining {S}himura curves
  and the {M}umford uniformization}, J. Fac. Sci. Univ. Tokyo Sect. IA Math.
  \textbf{25} (1979), no.~3, 277--300.

\bibitem[Lan87]{EllipticFunctions}
Serge Lang, \emph{Elliptic functions}, second ed., Graduate Texts in
  Mathematics, vol. 112, Springer-Verlag, New York, 1987, With an appendix by
  J. Tate.

\bibitem[Liu02]{Liu}
Qing Liu, \emph{Algebraic geometry and arithmetic curves}, Oxford Graduate
  Texts in Mathematics, vol.~6, Oxford University Press, Oxford, 2002,
  Translated from the French by Reinie Ern{\'e}, Oxford Science Publications.

\bibitem[Lor11]{LoWMC}
Dino~J. Lorenzini, \emph{Wild models of curves},
  http://www.math.uga.edu/~lorenz/Paper2.pdf, 2011.

\bibitem[Mes72]{Messing}
William Messing, \emph{The crystals associated to {B}arsotti-{T}ate groups:
  with applications to abelian schemes}, Lecture Notes in Mathematics, Vol.
  264, Springer-Verlag, Berlin, 1972.

\bibitem[Mil80]{LEC}
James~S. Milne, \emph{\'{E}tale cohomology}, Princeton Mathematical Series,
  vol.~33, Princeton University Press, Princeton, N.J., 1980.

\bibitem[Mil86]{MilneJac}
\bysame, \emph{Jacobian varieties}, Arithmetic geometry ({S}torrs, {C}onn.,
  1984), Springer, New York, 1986, pp.~167--212.

\bibitem[Mol10]{MolinaBimodules}
Santiago Molina, \emph{Ribet bimodules and the specialization of {H}eegner
  points}, http://www.crm.es/Publications/10/Pr928.pdf, 2010.

\bibitem[Neu99]{Neukirch}
J{\"u}rgen Neukirch, \emph{Algebraic number theory}, Grundlehren der
  Mathematischen Wissenschaften [Fundamental Principles of Mathematical
  Sciences], vol. 322, Springer-Verlag, Berlin, 1999, Translated from the 1992
  German original and with a note by Norbert Schappacher, With a foreword by G.
  Harder.

\bibitem[Ogg74]{OggHyperelliptic}
Andrew~P. Ogg, \emph{Hyperelliptic modular curves}, Bull. Soc. Math. France
  \textbf{102} (1974), 449--462.

\bibitem[Ogg83]{OggReal}
\bysame, \emph{Real points on {S}himura curves}, Arithmetic and geometry,
  {V}ol. {I}, Progr. Math., vol.~35, Birkh\"auser Boston, Boston, MA, 1983,
  pp.~277--307.

\bibitem[Ogg85]{OggMauvaise}
\bysame, \emph{Mauvaise r\'eduction des courbes de {S}himura}, S\'eminaire de
  th\'eorie des nombres, {P}aris 1983--84, Progr. Math., vol.~59, Birkh\"auser
  Boston, Boston, MA, 1985, pp.~199--217.

\bibitem[Ozm12]{Oz09}
Ekin Ozman, \emph{Local points on quadratic twists of ${X}_0({N})$}, Acta
  Arith. \textbf{152} (2012), no.~4, 323--348.

\bibitem[Piz76]{Pizer}
Arnold Pizer, \emph{On the arithmetic of quaternion algebras}, Acta Arith.
  \textbf{31} (1976), no.~1, 61--89.

\bibitem[Rib89]{RibetBimodules}
Kenneth~A. Ribet, \emph{Bimodules and abelian surfaces}, Algebraic number
  theory, Adv. Stud. Pure Math., vol.~17, Academic Press, Boston, MA, 1989,
  pp.~359--407.

\bibitem[RS11]{RibetStein}
Ken Ribet and William Stein, \emph{Lectures on modular forms and {H}ecke
  operators}, http://wstein.org/books/ribet-stein/main.pdf, 2011.

\bibitem[RSY05]{RSY}
Victor Rotger, Alexei Skorobogatov, and Andrei Yafaev, \emph{Failure of the
  {H}asse principle for {A}tkin-{L}ehner quotients of {S}himura curves over
  {$\mathbb Q$}}, Mosc. Math. J. \textbf{5} (2005), no.~2, 463--476, 495.

\bibitem[S{\etalchar{+}}12]{sage}
W.\thinspace{}A. Stein et~al., \emph{{S}age {M}athematics {S}oftware ({V}ersion
  4.8)}, The Sage Development Team, 2012, {\tt http://www.sagemath.org}.

\bibitem[Sad10]{Sadek}
Mohammad Sadek, \emph{On quadratic twists of hyperelliptic curves},
  http://arxiv.org/abs/1010.0732, 2010.

\bibitem[Shi67]{Sh67}
Goro Shimura, \emph{Construction of class fields and zeta functions of
  algebraic curves}, Ann. of Math. (2) \textbf{85} (1967), 58--159.

\bibitem[Shi71]{Sh71}
\bysame, \emph{Introduction to the arithmetic theory of automorphic functions},
  Publications of the Mathematical Society of Japan, No. 11. Iwanami Shoten,
  Publishers, Tokyo, 1971, Kan{\^o} Memorial Lectures, No. 1.

\bibitem[Shi79]{Shioda}
Tetsuji Shioda, \emph{Supersingular {$K3$} surfaces}, Algebraic geometry
  ({P}roc. {S}ummer {M}eeting, {U}niv. {C}openhagen, {C}openhagen, 1978),
  Lecture Notes in Math., vol. 732, Springer, Berlin, 1979, pp.~564--591.

\bibitem[Vie77]{Vi77}
Eckart Viehweg, \emph{Invarianten der degenerierten {F}asern in lokalen
  {F}amilien von {K}urven}, J. Reine Angew. Math. \textbf{293/294} (1977),
  284--308.

\bibitem[Vig80]{Vigneras}
Marie-France Vign{\'e}ras, \emph{Arithm\'etique des alg\`ebres de quaternions},
  Lecture Notes in Mathematics, vol. 800, Springer, Berlin, 1980.

\end{thebibliography}

\end{document}